\documentclass[11pt]{article}

\usepackage[a4paper, left=3cm, right=3cm, top=3cm]{geometry}
\usepackage{amssymb}
\usepackage{amsmath}
\usepackage{amsthm}
\usepackage{graphicx}
\usepackage{color}
\usepackage{xcolor}
\usepackage{scrextend}
\usepackage{dsfont}
\usepackage{mathrsfs}
\usepackage{mathtools}
\usepackage{appendix}
\usepackage{blindtext}
\usepackage{caption}
\usepackage[hidelinks]{hyperref}
\usepackage{overpic}
\usepackage{mathtools}
\definecolor{Green}{cmyk}{1,0,1,0.25}
\usepackage{makecell}

\title{Stochastic Flows with Strong Shear -- Part I:\\
Strong Completeness and Set Attractors}
\author{Dennis Chemnitz, Maximilian Engel and Michael Scheutzow} 

\date{\today}

\newcommand{\fa}          {\quad \text{for all } \,}
\newcommand{\faa}          {\quad \text{for almost all } \,}

\newcommand{\eqdist}{\stackrel{\mathclap{\normalfont\mbox{d}}}{=}}

\newcommand{\R}{\mathbb{R}}
\newcommand{\dd}{\mathrm{d}}

\renewcommand{\P}{\mathbb{P}}

\newcommand{\vertiii}[1]{{\left\vert\kern-0.25ex\left\vert\kern-0.25ex\left\vert #1 
    \right\vert\kern-0.25ex\right\vert\kern-0.25ex\right\vert}}

\definecolor{dorange}{RGB}{179, 98, 0}

\theoremstyle{plain}
\newtheorem{theorem}{Theorem}[section]
\newtheorem{lemma}[theorem]{Lemma}
\newtheorem{proposition}[theorem]{Proposition}

\newtheorem{corollary}[theorem]{Corollary}
\newtheorem{maintheorem}{Theorem}

\newtheorem{theoremC}{Theorem}

\numberwithin{equation}{section}

\theoremstyle{definition}
\newtheorem{definition}[theorem]{Definition}

\newtheorem{remark}[theorem]{Remark}

\begin{document}

\maketitle
\begin{abstract}
    We study strong completeness and the existence of a set attractor for the stochastic flows induced by a class of two-dimensional stochastic differential equations resembling a planar Ornstein-Uhlenbeck process with an additional radius-dependent rotational drift term. Our main results give both sufficient and necessary conditions for the stochastic flows to be strongly complete and for the existence of set attractors. In particular, we demonstrate that if the derivative of the angular velocity $\rho(r)$ with respect to the radius $r$ satisfies $|\rho'(r)| \geq K_2\, r^3$, for large $r$ and some universal constant $K_2$, the diameter of a compact set can grow exponentially fast with positive probability. Furthermore, we show that for $|\rho'(r)|\geq r^{3+\varepsilon}$, $\varepsilon>0$, with probability one, exceptional initial conditions diverge to infinity in finite time, ruling out the existence of a global stochastic flow. 
\end{abstract}


\section{Introduction}
This is the first paper of a two-part study investigating the stochastic flow of stochastic differential equations (SDEs) of the form
\begin{align}\label{SDE}
\begin{split}
  \dd X(t)=-c\,X(t)\,\dd t+\rho(|X(t)|) \begin{pmatrix}0&1\\-1&0 \end{pmatrix}X(t)\,\dd t+\sigma\,\dd W(t),\\
\end{split}
\end{align}
  where $W(t)$ is a 2-dimensional Brownian motion, $c, \sigma > 0$ are positive constants and ${\rho:[0,\infty)\to \mathbb R}$ is a $C^2$-function with $\rho'(0)=0$, which ensures that the drift is also $C^2$.
\subsection{Stochastic Flows}

Since the foundational work of Kunita \cite{kunita1990stochastic}, \emph{stochastic flows} induced by SDEs have been a central object of study in probability theory and random dynamical systems (RDS). While Kunita's original work is restricted to SDEs whose drift and diffusion are globally Lipschitz, many subsequent works investigate the local existence of stochastic flows for SDEs with less regular drift. At the same time, the global-in-time existence of stochastic flows with superlinearly growing drift is still insufficiently understood. 

Unlike deterministic flows induced by ODEs, for stochastic flows, there are two different notions of completeness. An SDE is called \emph{weakly complete} if, for every initial condition, there exists a global-in-time solution almost surely. Crucially, for weak completeness, the probability-zero event on which the solution might blow up is allowed to depend on the initial condition. In contrast, \emph{strong completeness} requires that there exists an event that has probability 1 and on which there exists a global-in-time solution for every initial condition,
and that there is a version of these solutions that depends continuously on the initial condition. Strong completeness is equivalent to the existence of a global-in-time stochastic flow. A question closely related to strong completeness is whether the stochastic flow admits a set attractor. Loosely speaking, for a global stochastic flow, a set attractor is a random compact set that depends only on the noise and to which any bounded set is attracted.

The goal of this work is to better understand under which conditions the stochastic flow induced by an SDE is strongly complete and under which conditions it admits a set attractor.
For this we consider the (local) stochastic flow $(t, \omega,x) \mapsto \varphi(t,\omega,x)$ induced by the SDE \eqref{SDE}.
  A direct calculation shows that, independently of the choice of $\rho$, the Gaussian
\begin{equation*} \label{eq:stat_dens_Cartesian}
    \dd \mu(x,y) = \frac{c}{\pi \sigma^2} \exp\left( -c\,\frac{x^2 + y^2}{\sigma^2} \right) \dd x\, \dd y
\end{equation*} 
is the unique stationary measure for~\eqref{SDE}. 
Ignoring, for now, the possibility for $R(t)$ to pass through the origin, we may translate Equation \eqref{SDE} to polar coordinates.
 Setting $R(t) := |X(t)|$ and introducing the angular stochastic process $\phi(t)$ given by the polar coordinate representation
\begin{equation*}\label{polar}
    X(t) = R(t) \begin{pmatrix}\cos(\phi(t))\\\sin(\phi(t))\end{pmatrix},
\end{equation*}
we obtain, as a straightforward application of It\^{o}'s formula,
\begin{align} \label{eq:SDE_polar}
\begin{split}
    \dd R(t) &= \left(-c\,R(t)+\frac{\sigma^2}{2R(t)}\right) \dd t +\sigma\,[\cos(\phi(t))\,\dd W_1(t)+\sin(\phi(t))\,\dd W_2(t)],\\
    \dd \phi(t) &= -\rho(R(t))\, \dd t + \frac{\sigma}{R(t)}\,[-\sin(\phi(t)) \,\dd W_1(t) + \cos(\phi(t)) \,\dd W_2(t)].
    \end{split}
\end{align}
 In law, the SDE \eqref{eq:SDE_polar} is equivalent to
\begin{align} \label{eq:SDE_polar_simple}
\begin{split}
    \dd R(t) &= \left(-c\, R(t)+\frac{\sigma^2}{2R(t)}\right) \,\dd t +\sigma\,\dd W_R(t),\\
    \dd \phi(t) &= -\rho(R(t)) \,\dd t + \frac{\sigma}{R(t)} \,\dd W_\phi(t),
    \end{split}
\end{align}
Note that the law of $R(t)$ does not depend on $\rho$ and is thus identical to the law of the radius of the Ornstein-Uhlenbeck process \eqref{SDE} with $\rho=0$. In particular, the SDE \eqref{SDE} is weakly complete for any $\rho$. 
 On the other hand, the law of $\phi(t)$ depends on $R(t)$ both through the diffusion and, if $\rho$ is non-constant, through the drift. We call the latter effect \emph{shearing}.
This skew-product structure makes the analysis of Equation \eqref{eq:SDE_polar_simple} comparatively simple.
However, this analysis is only valid for the law of a single solution, while for the pathwise dynamics, the phase $\phi$ does influence the amplitude $R$ through the noise term, see Equation \eqref{eq:SDE_polar}. This interplay between phase and amplitude is known to have the potential to induce chaotic dynamics, so-called \emph{shear-induced chaos} \cite{LinYoung08,DoanEngeletal2018, EngelEtAl19,ChemnitzEngel23,Baxendale24}. 
In this article, we show that, if the shear is sufficiently strong, a similar mechanism can cause solutions to go to infinity. Depending on the asymptotic growth rate of $|\rho'|$, this can cause the stochastic flow to no longer admit a set attractor, or even to lose strong completeness. We call the latter effect \emph{shear-induced blow-up}.

Without loss of generality, we restrict our analysis to the case $c = \sigma = 1$, that is, the study of the SDE
\begin{equation}\label{SDE2}
    \dd X(t)=-X(t)\,\dd t+\rho(|X(t)|) \begin{pmatrix}0&1\\-1&0 \end{pmatrix}\,X(t)\,\dd t+\dd W(t).
\end{equation}
Indeed, using linear transformations of the state space, as well as time, we can transform the SDE \eqref{SDE} for any values of $c$ and $\sigma$ back to $c = \sigma = 1$. Note that neither strong completeness nor the existence of set attractors are affected by such a coordinate transformation.

\subsection{Main results}
For two functions $f,g$, the notation $f(r) \lesssim g(r)$ means that there exists a constant $C$ independent of $r$, for which $f(r)\leq C\,g(r)$, for all $r\geq 0$.
\begin{maintheorem}\label{theo:attractors1}
    The following hold:
    \begin{enumerate}
        \item[(a)] If the function $\rho$ is such that there exists an $\varepsilon>0$ for which
            \begin{equation}
                |\rho (r)| \lesssim r^{3/2-\varepsilon}+1,  \tag{Compl-I}\label{eq:Compl-I}
            \end{equation}
            then the RDS $\varphi$ induced by \eqref{SDE2} is strongly complete (cf.~Definition \ref{def:compl}).
        \item[(b)] If the function $\rho$ satisfies 
    \begin{equation}
        |\rho (r)| \lesssim r+1,  \tag{PBAtt}\label{eq:PBAtt}
    \end{equation}
    then the RDS $\varphi$ induced by \eqref{SDE2} admits a pull-back set attractor (cf.~Definition \ref{def:Attr}).
    \end{enumerate}
\end{maintheorem}
Theorem \ref{theo:attractors1} (a) is a direct consequence of a recent paper by Li and Ying, specifically \cite[Theorem E (i)]{li2025strong}. In particular, this is the strongest result establishing strong completeness for \eqref{SDE2} available in the literature.
The proof of Theorem \ref{theo:attractors1} (b) uses a well-established method \cite{imkellerlederer02,IS01, DoanEngeletal2018} based on a random transformation of the SDE \eqref{SDE2} into a random ordinary differential equation (RDE). It is carried out in Section \ref{sec:ProofAttrI}. 

Using information not only on the absolute value of $\rho$, but on the absolute value of its derivative $\rho'$, we give a different set of conditions for strong completeness and the existence of a set attractor, which allow for a higher-order polynomial growth of $\rho$. As a trade-off, we can only show the existence of an attractor according to a weaker notion of set attraction. 

\begin{maintheorem}\label{theo:attractors2}
    The following hold:
    \begin{enumerate}
        \item[(a)] If the function $\rho$ satisfies
        \begin{equation}
           \int_0^\infty \left(1\vee\sup_{r' \in [0, r]} |\rho'(r')|\right)^{-1/3} \dd r = \infty,\tag{Compl-II} \label{eq:Compl}
        \end{equation}
        then the RDS $\varphi$ induced by \eqref{SDE2} is strongly complete.
        \item[(b)] There exists a constant $K_1>0$, independent of $\rho$, with the property that if the function $\rho$, for some $r^*>0$,  satisfies 
        \begin{equation}
            |\rho'(r)| \leq K_1\,r^3,\quad\forall\, r\geq r^*,\tag{Attr} \label{eq:Attr}
        \end{equation}
       then the RDS $\varphi$ induced by \eqref{SDE2} admits a weak set attractor (cf.~Definition \ref{def:Attr}).
    \end{enumerate}
\end{maintheorem}
The proof, which is based on tracking the growth of a discrete set of bounded balls, similar to \cite{LeimbachMattinglyScheutzow22}, is carried out in Section \ref{sec:Attr2}. 

The criteria introduced in Theorem \ref{theo:attractors2} are tight in terms of the exponent, as the following result shows.
\begin{theoremC}\label{theo:BlowUpShort}
The following hold:
\begin{enumerate} 
\item[(a)] There exists a constant $K_2>0$, independent of $\rho$, with the property that if the function $\rho$, for some $r^*>0$,  satisfies
\begin{equation}
 |\rho'(r)|\geq K_2 \,r^3 \quad \forall \, r\geq r^*,
\tag{NoAtt} \label{eq:NoAttr}
\end{equation}
as well as
\begin{equation}\label{eq:BasicAssumption}
    \limsup_{r \to \infty}\frac{|\rho''(r)|}{|\rho'(r)|}<\infty, \quad \text{and}\quad |\rho'(r)| \text{ increases monotonically on }[r^*,\infty),
\end{equation}
then the RDS $\varphi$ induced by \eqref{SDE2} does not admit a weak set attractor.

\item[(b)]If, moreover, 
\begin{equation}
\int_{r^*}^\infty \big|\rho'(r)\big|^{-1/3}\,\dd r<\infty\tag{NoCom} \label{eq:NoCompl},
\end{equation}
then the RDS $\varphi$ induced by \eqref{SDE2} is not strongly complete. In particular, almost surely, there exists an initial condition for which the solution to \eqref{SDE2} goes to infinity in finite time.
\end{enumerate}
\end{theoremC}
The proof, which is based on an iterative scheme tracking two points at a time, loosely inspired by \cite{CranstonScheutzowSteinsaltz99}, is carried out in Section \ref{sec:BlowUp}. 

Together, Theorems \ref{theo:attractors2} and \ref{theo:BlowUp} allow us to fully determine whether the induced RDS $\varphi$ is strongly complete for all functions $\rho$ that satisfy \eqref{eq:BasicAssumption}. This class includes, among others, all polynomials of degree at least 2.

\subsection{Shear-induced Blow-up}
Theorem \ref{theo:BlowUp} is somewhat surprising. Note that for any $c>0$ and any $\rho$, the ODE given by \eqref{SDE} with $\sigma = 0$ is not only strongly complete, but even coercive, which means that any solution $X(t)$ satisfies $|X(t)| \leq e^{-ct}\, |X(0)|$. If $\rho$ satisfies \eqref{eq:BasicAssumption} and \eqref{eq:NoCompl}, such as, for example, $\rho(r) = r^5$, then, by Theorem \ref{theo:BlowUp}, the addition of a small amount of additive noise can destroy the strong completeness. In other words, even though both the drift and the diffusion are well-behaved individually, their interaction can cause solutions to exceptional initial conditions to blow up in finite time. To the authors' knowledge, the SDE \eqref{SDE} with $\rho$ satisfying the conditions of Theorem \ref{theo:BlowUp} (b) is the first example of a drift $F$ satisfying $\langle x, F(x)\rangle \leq c |x|^2$ for some $c>0$,that allows for finite time blow-up once noise is added. This possibility was already suggested in \cite[Section 7]{SS17}.

Since the drift only decreases the radius, the noise must be responsible for pushing the trajectory away from the origin. In fact, the two-dimensional Brownian motion $W(t)$ already travels infinitely far in finite time, which is reflected in the fact that the path of Brownian motion is almost surely not of finite variation. Yet, without drift, the trajectory of Brownian motion over a finite amount of time is almost surely bounded. Understanding Brownian motion as the path of a particle subjected to infinitely many infinitesimal, independent, identically distributed kicks, this can be attributed to a destructive interference between the individual kicks. Even though the drift does not increase the radius directly, the $\rho$-dependent shear term can disrupt this deconstructive interference through the $\phi$-dependence of $W_R$, see Equations \eqref{eq:SDE_polar} and \eqref{eq:SDE_polar_simple}. Theorem \ref{theo:BlowUp} now states that for almost every realization of the Brownian motion, there exists some initial condition, for which the solution trajectory of the SDE \eqref{SDE} ``untwists" the Brownian motion enough to cause the solution trajectory to reach infinity in finite time, see Figure \ref{fig:untangle}.
\begin{figure}
    \centering
    \begin{overpic}[width=0.8\linewidth]{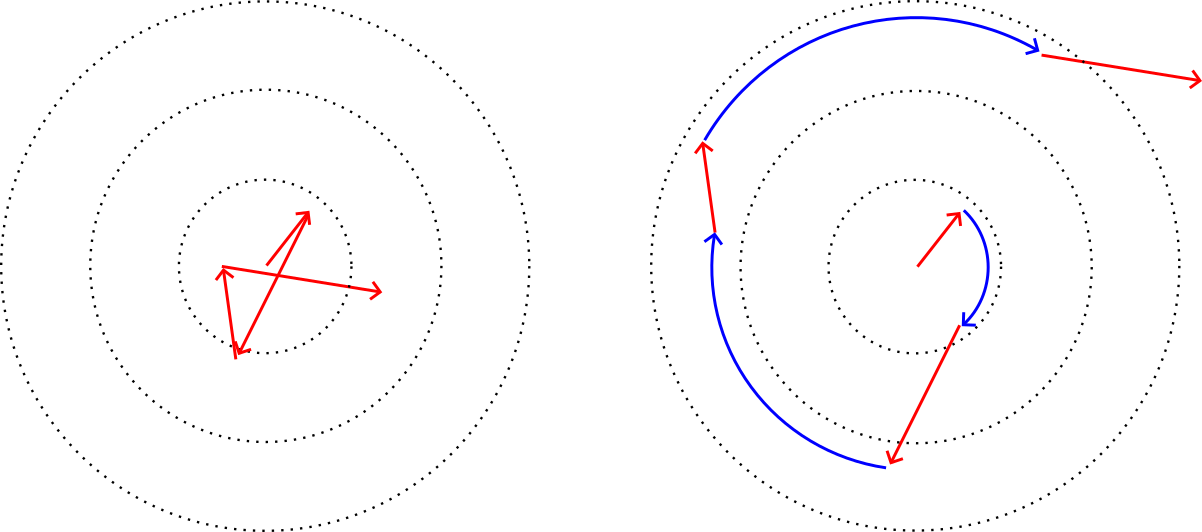}
        \put(20,-5){(A)}
        \put(74,-5){(B)}
    \end{overpic}\vspace{1.5em}
    \caption{In (A), the red arrows show the trajectory of a Brownian motion started at the origin, schematically drawn as four consecutive kicks. Due to cancellation, the trajectory stays comparatively close to the origin. In (B), a schematic solution of \eqref{SDE} is shown, driven by the same realization of the Brownian motion as in (A). For simplicity, the dissipation is neglected and the trajectory is drawn as a composition of the four Brownian motion kicks shown in red interposed by the $\rho$-dependent drift shown in blue. This drift reduces the cancellation causing the radius to grow much more quickly compared to the Brownian motion in (A).
    The dotted circles are centered at the origin and serve as a reference.}
    \label{fig:untangle}
\end{figure}

The exponents in \eqref{eq:NoAttr}, and \eqref{eq:NoCompl} can be explained by the following heuristic. Consider the SDE \eqref{eq:SDE_polar} started at a point $X(0) = X\in \R^2$ with $|X|=R$ large enough, such that the terms in \eqref{eq:SDE_polar} involving $R^{-1}$ can be neglected. Suppose that, with prior knowledge of the noise, the angle $\phi$ of $X$ was chosen in such a way that the diffusion $\dd W_R= \cos(\phi)\,\dd W_1+\sin(\phi)\,\dd W_2$ increases the radius $R$ over a small time $\Delta T$. We denote this increase by $\Delta_{\operatorname{Diff}}R$ and note that $\Delta_{\operatorname{Diff}}R \sim \sqrt{\Delta T}$. During the same time, the drift pushes $X$ closer to the origin, leading to a decrease of $R$ on the order of $\Delta_{\operatorname{Drift}}R \sim R\,\Delta T$. If we want $R$ to grow, we need $\Delta_{\operatorname{Diff}}R>\Delta_{\operatorname{Drift}}R$ or equivalently $\Delta T \lesssim R^{-2}$. Over this period of time, we can expect $R$ to grow, assuming that $\phi$ started at a favorable angle. For $R$ to continue growing, the shear must move $\phi$ to the next favorable angle sufficiently quickly, that is, in at most a time of order $\Delta T$. 
Since during this time $R$ changes by an amount of order $$\Delta R \sim \Delta_{\operatorname{Diff}}R \sim \sqrt{\Delta T},$$ 
the angle displacement of two different trajectories starting at radius $R$ is of the order $$\Delta \phi \sim \Delta T\,\Delta \rho \sim \Delta T\, |\rho'(R)|\, \Delta R \sim \Delta T^{3/2} \,|\rho'(R)|.$$
To ensure that we can reach any desired angle in time $\Delta T$, we require $\Delta \phi$ to be of order at least 1. Recalling the relation $\Delta T \lesssim R^{-2}$ this is equivalent to
$$|\rho'(R)| \gtrsim \Delta T^{-3/2} \gtrsim R^3,$$
heuristically explaining the exponent in \eqref{eq:NoAttr}. To also explain the condition \eqref{eq:NoCompl}, we need to consider the speed at which $R$ can grow, which is given by
$$\frac{\Delta R}{\Delta T} \sim  \frac{\sqrt{\Delta T}}{\Delta T}= \Delta T^{-1/2}.$$
Choosing $\Delta T$ as small as possible, that is $\Delta T \sim |\rho'(R)|^{-2/3}$, we get 
$$\frac{\Delta R}{\Delta T} \sim |\rho'(R)|^{1/3}.$$ 
If $\hat R> R(0)$ is an even larger radius, then the time $T_{\hat R}$ required to reach $R(T_{\hat R}) = \hat R$ is given by
$$ T_{\hat R} = \int_{R(0)}^{\hat R} \frac{\Delta T}{\Delta R}\,\dd R \sim \int_{R(0)}^{\hat R} |\rho'(R)|^{-1/3}\,\dd R.$$
Under \eqref{eq:NoCompl}, this remains bounded as $\hat R \to \infty$ implying a blow-up in finite time.

\subsection{Outlook to Part II}
In an upcoming second part of this work, we will discuss linear stability, determined by the leading Lyapunov exponent, and its implications for synchronization \cite{FGS1} for model~\eqref{SDE}. In particular, we will use Equation \eqref{SDE} as an example to study synchronization for SDEs that are not strongly complete. Furthermore, we will investigate how the geometry of the blow-up set of initial conditions without global solutions is affected by the sign of the top Lyapunov exponent.

\subsection{Related literature}
\emph{Strong Completeness for SDEs:} A general criterion for strong completeness of SDEs based on the moments of the Jacobian of the stochastic flow was derived by Li in \cite{Li94}. We leave it as an open problem to determine for which functions $\rho$ this condition is satisfied. 
Li and Scheutzow \cite{LiScheutzow11} constructed an SDE with no drift and degenerate multiplicative noise that is not strongly complete. Notably, the diffusion coefficients are smooth and globally bounded.
Leimbach and Scheutzow \cite{leimbach2017blow} constructed an SDE with additive noise for which the stochastic flow is weakly but not strongly complete. Without the noise, solutions to the deterministic ODE blow up in finite time for some initial conditions. The addition of noise restores weak but not strong completeness.
Scheutzow and Schulze \cite{SS17} showed strong completeness for SDEs with additive noise under a one-sided Lipschitz condition and at most quadratic growth of the rotational component of the drift.
Riedel and Scheutzow \cite{riedel2017rough} derived a similar condition for rough differential equations.
Recently, Li and Ying \cite{li2025strong} relaxed these assumptions, allowing a stronger growth of the rotational component of the drift. Their conditions can be shown to be tight in the regularity of the driving signal. Applying their conditions to our setting exactly yields Theorem \ref{theo:attractors1} (a). In contrast to the argument we use in the proof of Theorem \ref{theo:attractors2} (a), their argument is purely deterministic and holds even if the Brownian motion in \eqref{SDE2} is replaced by any deterministic path of the class $C^{1/2-}$.

\emph{Attractors for Random Dynamical Systems:} The notion of attractors used in this work was originally introduced by Crauel and Flandoli \cite{cf94}. General criteria for their existence were introduced by Crauel, Dimitroff, and Scheutzow in \cite{cds09} (see also \cite{crauelgeissscheutzow2023} for a corrigendum). Leimbach, Mattingly, and Scheutzow \cite{LeimbachMattinglyScheutzow22} constructed a planar velocity field, for which the deterministic flow experiences blow-up in finite time, while additive noise stabilizes the flow in the sense that it is strongly complete and admits a weak set attractor. This construction can be seen as a converse to the drift of \eqref{SDE2} with $\rho$ satisfying \eqref{eq:NoCompl}, where the deterministic flow is globally well-posed while the SDE is not strongly complete. The proof of Lemma \ref{lemm:ballcrit} below is heavily inspired by the proof for the existence of a set attractor in \cite{LeimbachMattinglyScheutzow22}.

\emph{Shear-Induced Chaos:} The interplay between radius-dependent rotational drift and additive white noise, that is responsible for shear-induced blow-up, can also lead to chaotic dynamics indicated by a positive top Lyapunov exponent. This effect is known as shear-induced chaos \cite{LinYoung08} and inspired the name ``shear-induced blow-up''. Shear-induced chaos has been demonstrated in both deterministic \cite{wang2003strange} and stochastic \cite{EngelEtAl19, breden2023computer,ChemnitzEngel23,Baxendale24} settings. Depending on the function $\rho$, shear-induced chaos is possible in the SDE \eqref{SDE} as well, an aspect that will be discussed in detail in the upcoming second part of this work.

\emph{Strong Completeness and Attractors for SPDEs:} Global well-posedness of stochastic flows has also been studied in the setting of (singular) stochastic partial differential equations, with positive results for the $\Phi^4_3$ equation by Mourrat and Weber \cite{MourratWeber17b} and for the 2-dimensional Navier--Stokes equation with additive space-time white noise  by Hairer and Rosati \cite{HairerRosati24}. Additionally, Gess and Tsatsoulis \cite{gess2020synchronization} showed synchronization, i.e., the existence of a singleton attractor, for the stochastic quantisation equation in 2 and 3 dimensions.
\subsection{Structure of the paper}
In Section \ref{sec:preliminaries}, we recall some basic notions from the theory of random dynamical systems. In Section \ref{sec:ProofAttrI} we prove Theorem \ref{theo:attractors1}, in Section \ref{sec:Attr2} we prove Theorem \ref{theo:attractors2}, and in Section \ref{sec:BlowUp} we state and prove the full version of Theorem \ref{theo:BlowUp}. In Sections \ref{sec:Attr2GC} and \ref{sec:BlowUpCrit}, we first introduce new general criteria for strong completeness and the existences of weak set attractors, respectively the lack thereof, which we then use to prove Theorems \ref{theo:attractors2} and \ref{theo:BlowUp}. Parts of Section \ref{sec:Attr2} already appeared in the first author's PhD thesis \cite{Chemnitz2025Stability}.

\section{Preliminaries}\label{sec:preliminaries}

In this section, we recall the definitions of local random dynamical systems, of weak and strong completeness, and of weak set attractors and pull-back set attractors. Furthermore, we discuss the construction of the local RDS induced by an SDE with additive noise and locally Lipschitz continuous drift.

\subsection{Local RDS, Strong Markov Property, and Completeness}
We start by recalling the definition of a local random dynamical system, see \cite[Section 1.2]{a98}.
\begin{definition}[Metric Dynamical System]
    A \emph{metric dynamical system} (MDS) is a tuple $(\Omega,\mathcal F,\mathbb P,(\theta_t)_{t\in \mathbb R})$, such that the following hold. 
    \begin{enumerate}
        \item[(i)] $(\Omega, \mathcal F, \mathbb P)$ is a probability space.
        \item[(ii)] For each $t \in \mathbb R$, $\theta_t: \Omega \to \Omega$ is a map such that $(\omega, t)\mapsto \theta_t \omega$ is jointly measurable, $\theta_0 = \operatorname{id}_\Omega$, and $\theta_t \circ \theta_s = \theta_{s+t}$ for all $s, t\in \mathbb R$.
        \item[(iii)] The probability measure is $\theta$ invariant, that is $\mathbb P(\theta_t^{-1}(E)) = \mathbb P(E)$ for all $t\in \mathbb R$ and $E \in \mathcal F$.
    \end{enumerate}
\end{definition}
\begin{remark}
    We purposefully avoid asking for $(\Omega, \mathcal F, \mathbb P)$ to be complete. In fact, replacing $\mathcal F$ by its $\mathbb P$-completion $\bar {\mathcal F}$ can destroy the joint measurability of $\theta$ in many relevant cases, see Remark \ref{ref:Completion} below.
\end{remark}
For some topological space $\mathcal X$, let $\mathcal B(\mathcal X)$ be the Borel $\sigma$-algebra on $\mathcal X.$ To introduce local random dynamical systems on $\mathbb R^d$, we augment the state space with a cemetery state $\Delta$. Let $\mathbb R^d_\Delta\coloneqq\mathbb R^d \cup\{\Delta\}$ be the one-point compactification of $\mathbb R^d$, such that $x_n \to \Delta$ if and only if $|x_n| \to \infty$.

\begin{definition}[Continuous Local Random Dynamical System]\label{def:LocalRDS}
    Let $(\Omega,\mathcal F,\mathbb P,(\theta_t))$ be a MDS. A \emph{continuous local random dynamical system} on $\mathbb R^d$ over this MDS is a map $\varphi:[0,\infty) \times \Omega \times \mathbb R^d_\Delta \to \mathbb R^d_\Delta$ such that the following hold.
    \begin{enumerate}
        \item[(i)] The map $\varphi$ is jointly measurable with respect to $\mathcal B([0,\infty)) \otimes \mathcal F \otimes \mathcal B(\mathbb R^d_\Delta)$.
        \item[(ii)] The state $\Delta \in \mathbb R^d_\Delta$ is a cemetery state, that is, for all $t \geq 0$ and all $\omega \in \Omega$ we have $\varphi(t, \omega, \Delta) = \Delta$.
        \item[(iii)] The map $\varphi$ satisfies the (perfect) cocycle property, that is, for all $t,s \geq 0$, $\omega \in \Omega$ and $x \in \mathbb R^d_\Delta$, we have
        $$\varphi(0, \omega, x) = x \quad \text{and}\quad \varphi(s+t, \omega, x) = \varphi(t, \theta_s \omega, \varphi(s,\omega, x)).$$
        \item[(iv)] For every $\omega \in \Omega$ and every $x \in \mathbb R^d$ the path $t\mapsto \varphi(t, \omega, x) \in \mathbb R^d_\Delta$ is continuous, for every $\omega \in \Omega$ the survival domain 
        $$D_\varphi(\omega) = \{(t,x): \varphi(t, \omega, x) \neq \Delta\}$$
        is open, and for each $\omega \in \Omega$, the map $\varphi(\cdot,\omega, \cdot)|_{D_\varphi(\omega)}$ restricted to the survival domain is jointly continuous. 
    \end{enumerate}
    In the following, we often omit the qualifier ``continuous'' and simply write ``local RDS''.
\end{definition}

For every $x\in \mathbb R^d$ and every subset $E \subseteq \mathbb R^d$, we define the blow-up time $\tau^\star$ by 
\begin{align*}
    \tau^\star(\omega,x) \coloneqq \inf \{t \geq 0: \varphi(t, \omega, x) = \Delta\} \in (0,\infty] \quad \text{and}\quad \tau^\star(\omega, E) &:= \inf_{x \in E} \tau^\star(\omega, x)\in [0,\infty].
\end{align*}

\begin{definition}[Filtered Random Dynamical System]\label{def:filtered}
    A \emph{filtered local random dynamical system} is a local RDS $(\theta, \varphi)$ for which the probability space $(\Omega, \mathcal F, \mathbb P)$ is equipped with a two-sided filtration $(\mathcal F_{s,t})_{-\infty < s\leq t < \infty}$ satisfying the following.
    \begin{enumerate}
        \item[(i)] For each $-\infty< s\leq t\leq u \leq v < \infty$, we have $\mathcal F_{t,u} \subseteq \mathcal F_{s, v}$.
        \item[(ii)] For each $-\infty< s\leq t< \infty$ and each $u \in \mathbb R$, we have $\theta_u^{-1}(\mathcal F_{s,t}) = \mathcal F_{s+u, t+u}.$
        \item[(iii)] For every $t \geq 0$, the map $\varphi(t, \cdot, \cdot): \Omega \times \mathbb R^d_\Delta \to \mathbb R^d_\Delta$ is measurable with respect to $\mathcal F_{0,t} \otimes \mathcal B(\mathbb R^d_\Delta)$.
        \item[(iv)] For each $t\geq 0$ and $s \geq 0$, the restricted shift $\theta: \Omega \times [0,t] \to \Omega$ is measurable as a map between the measurable spaces
        \begin{equation}\label{eq:prog}
            \theta: \big(\Omega \times [0,t], \mathcal F_{0, t+s} \otimes \mathcal B([0,t])\big) \to \big(\Omega, \mathcal F_{0,s}\big).
        \end{equation}
    \end{enumerate}
\end{definition}

For simplicity, we write $(\mathcal F_t)_{t \geq 0}$ for the forward filtration $\mathcal F_t\coloneqq \mathcal F_{0,t}$ and $(\mathcal F_t^+)_{t\geq 0}$ for its right-continuous modification
$$\mathcal F_t^+\coloneqq \bigcap_{\varepsilon>0} \mathcal F_{t+\varepsilon}.$$
We collect some basic properties of filtered local RDS.
\begin{proposition}\label{prop:tau}
    Let $(\theta,\varphi)$ be a filtered local RDS. Then the following hold.
    \begin{enumerate}
        \item[(i)] For every compact $C\subset \mathbb R^d$, the random time $\tau^\star(\omega,C)$ is an $(\mathcal F_t)$-stopping time, so, in particular, $\tau^\star(\omega,x)$ is an $(\mathcal F_t)$-stopping time for every $x \in \mathbb R^d$.
        \item[(ii)] For every set $E \subseteq \mathbb R^d$ that is a countable union of compact sets, the random time $\tau^\star(\omega,E)$ is an $(\mathcal F_t^+)$-stopping time, so, in particular, $\tau^\star(\omega,\mathbb R^d)$ is an $(\mathcal F_t^+)$-stopping time.
        \item[(iii)] For every finite $(\mathcal F_t^+)$-stopping time $\tau$, we have the inclusion
        \begin{equation}\label{eq:TauShift}
            \theta_\tau^{-1}(\mathcal F_t^+) \coloneqq \big\{\{\omega \in \Omega: \theta_{\tau(\omega)} \omega \in E\} : E \in \mathcal F_t^+\big\} \subseteq \mathcal F_{t+\tau}^+.
        \end{equation}
        Here $\mathcal F_{t +\tau}^+$ denotes the stopped $\sigma$-algebra
        $$\mathcal F_{t+ \tau}^+\coloneqq \big\{E\in \mathcal F: E\cap \{\omega \in \Omega :\tau(\omega) \leq s\} \in \mathcal F_{t+s}^+, ~\forall\, s\geq 0\big\}$$
        In particular, if $\tau_1$ and $\tau_2$ are finite $(\mathcal F_t^+)$-stopping times, their concatenation 
        \begin{equation}\label{eq:tauSum}
            \tau_3(\omega) \coloneqq \tau_1(\omega) + \tau_2(\theta_{\tau_1(\omega)}\omega)
        \end{equation}
        is an $(\mathcal F_t^+)$-stopping time as well.
    \end{enumerate}
\end{proposition}

\begin{proof}
   (i): Let $C\subset \mathbb R^d$ be compact. We want to show that $\tau^\star(\omega,C)$ is an $ (\mathcal{F}_t)$-stopping time, that is, $\{\omega \in \Omega: \tau^\star(\omega, C)\leq t\} \in \mathcal F_t$ for every $t \geq 0$. 
     Since the survival domain $D_\varphi(\omega)$ is open, $\tau^\star(\omega,x)$ is lower semicontinuous in $x$ for every $\omega \in \Omega$ and attains its minimum on the compact set $C$. Furthermore, by the time-continuity of trajectories, we have that if $\tau^\star(\omega,C)\leq t$, then the set 
     \begin{equation}\label{eq:imageSet}
         \{\varphi(s,\omega,x) : (s,x) \in ([0,t]\times C) \cap D_\varphi(\omega)\}
     \end{equation}
     is unbounded. Conversely, if $\tau^\star(\omega,C) >t$, then the compact set $[0,t]\times C$ is in the survival domain $D_\varphi$ so, by restricted continuity, the set \eqref{eq:imageSet} is bounded. Hence, we can write
     \begin{align*}
         \{\omega \in \Omega: \tau^\star(\omega, C)\leq t\} &= \bigcap_{n \in \mathbb N} \bigcup_{s \in [0,t]} \bigcup_{x \in C} \{\omega \in \Omega: (s,x) \in D_\varphi(\omega) \text{ and } |\varphi(s,\omega,x)| > n\}\\
         &= \bigcap_{n \in \mathbb N} \bigcup_{s \in [0,t]\cap \mathbb Q} \bigcup_{x \in C^*} \{\omega \in \Omega: (s,x) \in D_\varphi(\omega) \text{ and } |\varphi(s,\omega,x)| > n\},
     \end{align*}
     where $C^*\subseteq C$ is a countable dense subset and we use restricted continuity again. Since every event in the countable intersection of countable unions is in $\mathcal F_t$ by definition, we can conclude $\{\omega \in \Omega: \tau^\star(\omega, C)\leq t\} \in \mathcal F_t$ completing the proof.

    (ii) follows directly from (i), since countable infimums of $(\mathcal F_t)$-stopping times are $(\mathcal F_t^+)$-stopping times.

    (iii) Let $\tau$ be an $(\mathcal F^+_t)$-stopping time, $t\geq 0$ and $E \in \mathcal F_t^+$. Define the stochastic process $(E_s)_{s\geq 0}$ by 
    $$E_s(\omega) \coloneqq \mathds 1_E(\theta_s\omega).$$
    By Definition \ref{def:filtered} (iv), it is progressively measurable with respect to $(\mathcal F^+_{t+s})_{s\geq 0}$, that is, for each $u \geq 0$, the restriction $(E_s)_{s\in [0,u]}$ is $(\mathcal B([0,u])\otimes \mathcal F^+_{t+u})$-measurable. Hence, $E_\tau = \mathds 1_{\theta_\tau^{-1}(E)}$ is $\mathcal F^+_{t+\tau}$-measurable and $\theta_\tau^{-1}(E) \in \mathcal F^+_{t+\tau}$, establishing \eqref{eq:TauShift}. To show \eqref{eq:tauSum}, let $\tau_1$ and $\tau_2$ be $(\mathcal F_t^+)$-stopping times and $\tau_3$ be defined as in \eqref{eq:tauSum}. For each $t \geq 0$, we have
    \begin{align*}
        \{\omega \in \Omega : \tau_3(\omega) <t\} &= \bigcup_{s\in [0,t] \cap \mathbb Q} \{\omega \in \Omega : \tau_1(\omega) < s\} \cap \{\omega \in \Omega : \tau_2(\theta_{\tau_1(\omega)}\omega) < t-s\}\\
        &= \bigcup_{s\in [0,t] \cap \mathbb Q} \{\omega \in \Omega : \tau_1(\omega) < s\} \cap \theta_{\tau_1}^{-1}(\{\omega \in \Omega : \tau_2(\omega) < t-s\}).
    \end{align*}
    By \eqref{eq:TauShift}, the event $\theta_{\tau_1}^{-1}(\{\omega \in \Omega : \tau_2(\omega) < t-s\})$ is in $\mathcal F^+_{t-s+\tau_1}$ and hence 
    $$ \{\omega \in \Omega : \tau_1(\omega) < s\} \cap \theta_{\tau_1}^{-1}(\{\omega \in \Omega : \tau_2(\omega) < t-s\}) \in \mathcal F_{t-s+s}^+=\mathcal F_{t}^+,$$
    for every $s \in [0,t].$ In particular, $\{\omega \in \Omega : \tau_3(\omega) <t\} \in \mathcal F_{t}^+$ and, by right-continuity of the filtration, we also get $\{\omega \in \Omega : \tau_3(\omega) \leq t\} \in \mathcal F_{t}^+$. Since this holds for all $t\geq 0$, the random time $\tau_3$ is a stopping time as claimed.
\end{proof}

\begin{definition}[Weak/Strong Completeness]\label{def:compl}~
    \begin{enumerate}
        \item [(i)] A local RDS $\varphi$ is called \emph{weakly complete}, if we have $\mathbb P(\tau^\star(\omega, x) = \infty) = 1$, for every $x \in \mathbb R^d$.
        \item[(ii)]The local RDS $\varphi$ is called \emph{strongly complete}, if  ${\mathbb P(\tau^\star(\omega, \mathbb R^d) = \infty) = 1}$.
    \end{enumerate}
\end{definition}

Weak completeness corresponds to global  well-posedness of the one-point process induced by $\varphi$, while strong completeness corresponds to global  well-posedness of the flow.
It can be readily seen that strong completeness implies weak completeness. The converse, however, is not true, see \cite{Elworthy78,LiScheutzow11}.

In the theory of random dynamical systems it is often crucial to assume independence of past and future.
The most common assumption is to assume that the RDS is a \emph{Markov RDS}, meaning that the $\sigma$-algebras
$$\mathcal F_{-\infty,0} \coloneqq \sigma\Big( \bigcup_{t\geq 0} \mathcal F_{-t,0}\Big)\quad \text{and} \quad \mathcal F_{0,\infty} \coloneqq \sigma \Big(\bigcup_{t\geq 0} \mathcal F_{0,t}\Big)$$
are independent. However, this notion is too weak for our purposes and we require the following stronger version below.

\begin{definition}[Strong Markov Property]
    A filtered local RDS $(\theta,\varphi)$ has the \emph{strong Markov property} if its filtration satisfies the following. For every finite $(\mathcal F_t^+)$-stopping time $\tau$ and every bounded random variable $Y$ that is measurable with respect to $\mathcal F_{0, \infty}$, we have
    \begin{equation}\label{eq:StrongMarkov}
        \mathbb E[Y(\theta_{\tau(\omega)}\omega)\mid \mathcal F_\tau^+] = \mathbb E[Y], \quad \text{almost surely}.
    \end{equation}

\end{definition}

\subsection{Local RDS generated by SDEs with Additive Noise}\label{sec:RDS}
We consider the SDE
\begin{equation}\label{eq:SDEwithAddNoise}
    \dd X(t) = F(X(t)) \,\dd t + \dd W(t),
\end{equation}
 where $F: \mathbb R^d \to \mathbb R^d$ is a locally Lipschitz continuous vector field and $(W(t))_{t \in \mathbb R}$ is a two-sided $d$-dimensional Brownian motion. 
The appropriate probability space for equation~\eqref{eq:SDEwithAddNoise} is $\Omega=C_0(\mathbb R,\mathbb R^d)$, that is, the space of all continuous functions $\omega:\mathbb R\rightarrow \mathbb R^d$ satisfying $\omega(0)=0$.
We endow $\Omega$ with the compact-open topology and denote by $\mathcal F=\mathcal B(\Omega)$ the Borel $\sigma$-algebra on $\Omega$.
It is well known that there exists the \emph{Wiener probability measure} $\mathbb P$ on $(\Omega,\mathcal F)$ for which the process $(W(t,\omega))_{t \in \mathbb R}$ defined by $W(t,\omega) := \omega(t)$ is a $d$-dimensional Brownian motion.  For each $t\in\mathbb R$, we define the shift map $\theta_t:\Omega\rightarrow \Omega$ by
\begin{equation*}
(\theta_t \omega)(s) = \omega(s+t) - \omega(t) \fa s \in \R\,.
\end{equation*}
It is well known that $(\Omega,\mathcal F,\mathbb P,(\theta_t)_{t\in \mathbb R})$ forms a metric dynamical system, see \cite{a98}. 
We define the sub $\sigma$-algebra $\mathcal{F}_{s,t}$ as the $\sigma$-algebra generated by $\omega(u)- \omega(v)$ for $s \leq v \leq u \leq t$.
It is well known that $(\Omega,\mathcal F,\mathbb P,(\theta_t)_{t\in \mathbb R}, (\mathcal F_{s, t})_{s\leq t \in \mathbb R}))$ forms a filtered metric dynamical system, see \cite{a98}. In particular, the progressive measurability of the shift 
\eqref{eq:prog} is a direct consequence of the joint continuity of the shift with respect to the compact-open topology. Furthermore, the forward filtration $(\mathcal F_t)$ is exactly the canonical filtration of the Brownian motion, and it is a standard fact that it satisfies the strong Markov property, see, for example, \cite[Theorem 2.6.16]{KaratzasShreve1991}. 
\begin{remark}\label{ref:Completion}
    While $\theta$ is jointly measurable with respect to the Borel $\sigma$-algebra $\mathcal B(\Omega)$, it is not jointly measurable if $\mathcal F$ is replaced by its $\mathbb P$-completion $\bar {\mathcal F}$. This can be seen as follows. Let $B\subset \mathbb R$ be some set that is not Lebesgue measurable and set $f(t) \coloneqq \arctan(t) \,e_1$, where $e_1 \in \mathbb R^d$ is the first unit vector. Clearly $f \in C_0(\mathbb R, \mathbb R^d)$. Set $A \coloneqq \{\theta_t f: t\in B\}$. Since $A$ contains only bounded functions, $A \in \bar {\mathcal F}$. However, $\{(t,g) \in \mathbb R\times C_0(\mathbb R, \mathbb R^d): \theta_t g \in A\}$ is not in $\mathcal B(\mathbb R) \otimes \bar{\mathcal F}$, since its section at $g = f$ is equal to $B$. This issue persists if $\mathcal B(\mathbb R)$ is replaced by the $\sigma$-algebra of Lebesgue measurable sets.
\end{remark}

The construction of a stochastic flow for \eqref{eq:SDEwithAddNoise} can be achieved without the use of stochastic calculus. For simplicity, we only consider trajectories forward in time. First, we temporarily fix some $\omega \in \Omega$ and consider the non-autonomous ODE
\begin{equation}\label{eq:RDE}
    \frac{\dd}{\dd t}\tilde X(t) = F\big(\tilde X(t) + W(t,\omega)\big).
\end{equation}
Since $F$ is locally Lipschitz, this ODE admits a unique local solution map, which we denote by $\tilde \varphi_\omega: [0,\infty)\times \mathbb R^d \supseteq D_{\tilde \varphi_\omega} \to \mathbb R^d$, where $D_{\tilde \varphi_\omega}$ is the maximal domain. By setting $\tilde \varphi_\omega(t,x) \coloneqq \Delta$ for each $(t,x) \in [0,\infty) \times \mathbb R^d_\Delta \setminus D_{\tilde \varphi_\omega}$, we can extend $\tilde \varphi_\omega(t,x)$ to a flow $$\tilde \varphi_\omega: [0,\infty) \times \mathbb R^d_\Delta \to \mathbb R^d_\Delta.$$
It is a standard fact from the theory of ordinary differential equations that 
$$\{(t, \omega,x) \in [0,\infty) \times \Omega \times \mathbb R^d: (t,x) \in D_{\tilde \varphi_\omega}\}$$
is open and $\tilde \varphi_\omega(t,x)$ is jointly continuous in $t$, $\omega$ and $x$, when $\Omega$ is equipped with the compact-open topology.  Furthermore, trajectories $t \mapsto \tilde \varphi_\omega(t, x)$ are continuous in $\mathbb R^d_\Delta$ for every $\omega \in \Omega$ and every $x \in \mathbb R^d$.
We define the map $\varphi: [0,\infty) \times \Omega \times \mathbb R^d_\Delta \to \mathbb R^d_\Delta$ by
$$\varphi(t, \omega, x) = \tilde \varphi_\omega(t,x) + W(t, \omega),$$
where we use the convention $\Delta + w \coloneqq \Delta$ for every $w\in \mathbb R^d$. A simple calculation shows that $X(t) = \varphi(t, \omega, x)$ is the maximal solution to the SDE \eqref{eq:SDEwithAddNoise} with $X(0) = x$. Hence, it can be checked that $\varphi$ is a local RDS in the sense of  Definition \ref{def:LocalRDS}.

\subsection{Random Attractor}
Recall that $E(\omega)$ is a random closed set, if $E(\omega)$ is closed for every $\omega \in \Omega$ and for every open set $U$, the event $\{\omega \in \Omega : E(\omega) \cap U \neq \emptyset\} \in \mathcal F$ is measurable.
A random closed set $C(\omega)$ is called \emph{random compact set} if it is compact for every $\omega$. Furthermore, it is called measurable with respect to the past if it is measurable with respect to $\mathcal F_{-\infty,0}$.
A random compact set $C(\omega)$ is called $\varphi$\emph{-invariant}, if $\tau^\star(\omega, C(\omega)) = \infty$ almost surely and for every $t \geq 0$ we have
$$\mathbb P(\varphi(t, \omega, C(\omega)) = C(\theta_t(\omega))) = 1.$$
A random variable $R:\Omega\rightarrow (0,\infty)$ is called \emph{tempered} if
\[
\lim_{t \to \pm \infty} \frac{1}{|t|} \log^{+} R(\theta_t\omega)=0, \faa \omega\in \Omega\,,
\]
(see,~e.g.,~\cite[Definition 4.1.1]{a98}). A random compact set $(C(\omega))_{\omega \in \Omega}$ is called \emph{tempered} if there exists a tempered random variable $R$ such that
\[
  C(\omega)\subseteq B_0(R(\omega)), \faa \omega\in\Omega,
\]
where $B_x(R) \subset \R^d$ denotes the open ball with center $x$ and radius $R$.
\begin{definition}[Weak/Pull-back Set Attractor]\label{def:Attr}
   A non-empty random compact $\varphi$-invariant set $(A(\omega))_{\omega \in \Omega}$ which is measurable with respect to the past is called 
	\begin{enumerate}
		\item[(i)] \emph{weak set attractor}, if for every deterministic bounded set $E \subset \mathbb R^d$, we have
            $$\operatorname{dist} \big(\varphi(t, \omega, E), A(\theta_t \omega)\big) \to 0, ~\text{in probability},$$
            where $\operatorname{dist}(E_1, E_2):= \sup_{x\in E_1}\inf_{y\in E_2}|x-y|$ is the Hausdorff semi-distance.
		\item[(ii)] \emph{pull-back set attractor}, if for every tempered random compact set $(E(\omega))_{\omega \in \Omega}$, we have
		\begin{equation*}
		\operatorname{dist} \big(\varphi(t, \theta_{-t} \omega, E(\theta_{-t}\omega)), A(\omega)\big) \to 0 , \faa \omega \in \Omega.
		\end{equation*}
	\end{enumerate}
\end{definition}
 We only consider the notions of weak/pull-back set attractors for random dynamical systems which are strongly complete in order to ensure that the Hausdorff semi-distance is well-defined. 
\begin{remark}\label{rmk:WeakSetAttr}
 For pull-back set attractors, it is useful to consider attraction by all tempered compact random sets, since this leads to a notion of attractor which is invariant under certain random coordinate transformations, such as the ones considered in Section \ref{sec:ProofAttrI}. For weak set attractors we only require attraction of deterministic sets for simplicity, even though our definition also implies that if $A(\omega)$ is a weak set attractor and $E(\omega)$ is a tempered set, then 
$$\operatorname{dist} \big(\varphi(t, \omega, E(\omega)), A(\theta_t \omega)\big) \to 0, ~\text{in probability}.$$
This can be seen as follows. We need to show that for every $\varepsilon, \delta > 0$, there exists a $T>0$, such that for all $t \geq T$
     \begin{equation}\label{eq:WeakSetAttr}
         \mathbb P \big(\operatorname{dist}\big(\varphi(t, \omega, E(\omega)), A(\theta_t \omega)\big)>\varepsilon \big) < \delta.
     \end{equation}
     Choose $R>0$, such that $\mathbb P(E(\omega) \subseteq \bar B_0(R)) > 1-\delta/2$, where $\bar B_0(R)$ is the closure of $B_0(R)$. Now, use the fact that there exists a $T\geq 0$ such that \eqref{eq:WeakSetAttr} holds for all $t \geq T$ if $E(\omega)$ is replaced with the deterministic set $\tilde E = \bar B_0(R)$ and $\delta$ is replaced with $\delta/2$.
 \end{remark}
Clearly, pull-back set attractors are also weak set attractors. The converse however is not true \cite{Scheutzow02}. Both weak and pull-back set attractors are essentially unique if they exist \cite[Lemma 1.3]{FGS1}, that is, if $A$ and $\tilde A$ are both set attractors, they must agree almost surely. 

In the remainder of this subsection, we present two criteria for the existence of a set attractor, which we use to prove Theorems \ref{theo:attractors1} and \ref{theo:attractors2}, respectively.
For this we assume that the RDS $\varphi$ is strongly complete. A tempered random compact set $B(\omega)$ is called an \emph{absorbing set}, if for almost all $\omega\in\Omega$ and any tempered random compact set $E(\omega)$, there exists a time $T_E(\omega)>0$ such that
\[
\varphi\big(t,\theta_{-t}\omega,E(\theta_{-t}\omega)\big)\subset B(\omega)\fa t\geq T_E(\omega).
\]
A proof of the following theorem, which we use to prove Theorem \ref{theo:attractors1} in Section \ref{sec:ProofAttrI}, can be found in \cite[Theorem~3.5]{flandolischmalfuss96}.

\begin{theorem}[Existence of a pull-back set attractor]\label{theo:ExistencePullback}
  Suppose that, for the RDS $\varphi$, there exists an absorbing set $B$. Then there exists an (essentially unique) pull-back set attractor $(A(\omega))_{\omega \in \Omega}$, given by
	\[
	A(\omega)
	:=
	\bigcap_{\tau\geq 0}\overline{\bigcup_{t\geq \tau} \varphi\big(t,\theta_{-t}\omega,B(\theta_{-t}\omega)\big)}.
	\]
\end{theorem}
For the proof of Theorem \ref{theo:attractors2} presented in Section \ref{sec:Attr2}, we use the following criterion for the existence of a weak set attractor which can be found in \cite[Theorem 4.2]{cds09} (note that \cite{cds09} has received a correction \cite{crauelgeissscheutzow2023}, which, however, does not affect the cited theorem).
\begin{theorem}[Existence of a weak set attractor]\label{theo:existenceWeakSetAtt}
    Suppose that for every $\varepsilon$, there exists a compact set $C\subset \mathbb R^d$, such that for every bounded set $B\subset \mathbb R^d$, there exists a time $T>0$, such that for all $t>T$ we have
    $$\mathbb P(\varphi(t,\omega, B)\subseteq C)\geq 1-\varepsilon.$$
    Then there exists an (essentially unique) weak set attractor.
\end{theorem}
\section{Existence of a Pull-Back Attractor (Proof of Theorem \ref{theo:attractors1})}\label{sec:ProofAttrI}

Theorem \ref{theo:attractors1} (a) is proved in \cite[Theorem E (i)]{li2025strong}. 
It remains to show Theorem \ref{theo:attractors1} (b). For this, we use a transformation into a random differential equation in order to use sample path estimates of solutions (see e.g.~\cite{imkellerlederer02,IS01}). Similarly to \cite{DoanEngeletal2018}, we make use of Theorem \ref{theo:ExistencePullback} for the existence of a random pull-back set attractor.

Recall that our SDE \eqref{SDE2} has the form \eqref{eq:SDEwithAddNoise} with
$$F(x) = -x+\rho(|x|) \begin{pmatrix}0&1\\-1&0 \end{pmatrix}\,x.$$
Consider the two-dimensional Ornstein-Uhlenbeck process
\begin{equation}\label{OU}
  \dd O(t) = -  \alpha O(t) \dd t + \dd W(t)\,,
\end{equation}
for some $\alpha > 0$, and define the random variable $O^*:=\int_{-\infty}^0 e^{ \alpha s} \dd W(s)$. 
Then $t\mapsto O^*(\theta_t\omega)$ solves \eqref{OU}, that is,
\begin{equation}\label{OU_Eq1}
  O^*(\theta_t\omega)
  =
  O^*(\omega)
  - \alpha \int_0^t O^*(\theta_s\omega)\,\dd s + \omega(t)\,.
\end{equation}
By replacing $\Omega$ with a measurable subset $\widehat\Omega\subset \Omega$ of full probability that is invariant under $\theta$, there exist two random variables $L$ and $J$ such that (see~\cite{ks98})
\begin{equation}\label{eq:Bound_OU}
	|O^*(\theta_t\omega)|^2
	\leq
	L(\omega)+ J(\omega)\ln(1+|t|) \text{ for all } t\in\mathbb R \text{ and } \omega \in\widehat\Omega\,,
	\end{equation}
 implying, in particular, that $|O^*(\theta_t\omega)|$ is tempered. 
We define the map $H:\widehat\Omega\times \mathbb R^2\rightarrow \mathbb R^2$ by $H(\omega,x):=x +  O^*(\omega)$.
Under the change of variable $x\mapsto H(\omega,x)$, the SDE~\eqref{SDE2} is transformed into the random differential equation
\begin{equation}\label{RandomDifferentialEquation}
   \dot X(t) =G(\theta_t\omega,X(t))\,,
\end{equation}
where $G(\omega,x):=F\left(H(\omega,x)\right)+ \alpha \,O^*(\omega)$. Depending on the choice of $\rho$, we show below that the solution
$\Psi(t,\omega,x)$ of this random differential equation
\begin{equation}
\label{eq:Psi}
\Psi(t,\omega,x)
=x+ \int_0^t G(\theta_s\omega, \Psi(s,\omega,x))\,\dd s\,,    
\end{equation}
exists for all $t\ge0$ and, hence, forms a global, strongly complete random dynamical system. 

\begin{lemma}\label{TransformSystems}
Assume that the solution $\Psi$ to~\eqref{eq:Psi} exists for all $t\geq0$. Then:
\begin{itemize}
  \item [(i)]
  The random dynamical system $\varphi: [0,\infty)\times \widehat\Omega\times \mathbb R^2\rightarrow \mathbb R^2$, defined by
  \begin{equation}\label{Transformation_Evolution}
   \varphi(t,\omega,x):= H(\theta_t\omega,\Psi(t,\omega,H(\omega)^{-1}x))\,,
  \end{equation}
  is generated by the stochastic differential equation \eqref{SDE2}.
  \item [(ii)]
  If the random dynamical system $\Psi$ has a pull-back set attractor, then also the random dynamical system $\varphi$ has a pull-back set attractor.
  \end{itemize}
\end{lemma}
\begin{proof}
(i)	From \eqref{Transformation_Evolution} and the definition of $H$, we have
\[
\varphi(t,\omega,x)
=
\Psi(t,\omega,x- O^*(\omega))+ O^*(\theta_t\omega)\,,
\]
which together with the fact that $\Psi$ is a solution of \eqref{RandomDifferentialEquation} implies that
\begin{align*}
\varphi(t,\omega,x)
&=
x- O^*(\omega)
+
\int_0^t G(\theta_s\omega,\Psi(s,\omega,x- O^*(\omega)))\,\dd s +  O^*(\theta_t\omega)\\
&=
x+\int_0^t G(\theta_s\omega, H(\theta_s\omega)^{-1}(\varphi(s,\omega,x)))\,\dd s+  (O^*(\theta_t\omega)-O^*(\omega))\,.
\end{align*}
Thus, using \eqref{OU_Eq1}, we obtain that
\begin{align*}
\varphi(t,\omega,x)
&=
x
+
\int_0^t G(\theta_s\omega, H(\theta_s\omega)^{-1}(\varphi(s,\omega,x)))-  \alpha O^*(\theta_s\omega)\,\dd s +  \omega(t)\\
&=
x+ \int_0^t F(\varphi(s,\omega,x))\,\dd s +  \omega(t) \,,
\end{align*}
which completes the proof of this part.

\noindent
(ii) This follows from the definition of a pull-back set attractor and the fact that the shifted term in the transformation $H(\omega, x)$, namely $|O^*(\theta_t\omega)|$, is tempered.
\end{proof}

Assuming that condition~\eqref{eq:PBAtt} holds, that is, there is a $C>0$ such that $|\rho(r)| \leq C (1+r)$, we show that the Euclidean norm of the solutions of \eqref{RandomDifferentialEquation} is bounded above by the growth of the corresponding solutions of the scalar equation of the form
\begin{equation}\label{Estimate_06}
\dot \eta = \gamma(\theta_t \omega) + \eta\, \beta(\theta_t \omega)\,,
\end{equation}
where the stationary processes $\gamma$ and $\beta$ will be determined below.
Note that for each initial value $\eta_0\in\mathbb R$, the explicit solution of \eqref{Estimate_06} is given by the variation of constants formula
\begin{equation}\label{Estimate07}
\eta(t,\omega,\eta_0)
=
e^{\int_0^t \beta (\theta_s \omega)\, \dd s} \left( \eta_0 +\int_0^t e^{ -\int_0^{s} \beta (\theta_u \omega)\, \dd u} \,\gamma(\theta_s \omega)\,\dd s\right)\,.
\end{equation}
\begin{lemma}\label{TechnicalLemma}
 Assuming  condition~\eqref{eq:PBAtt}, there exist stationary processes $\gamma$ and $\beta$
  such that for $x\in\mathbb R^2$ and $\eta_0  \geq  \left|x\right|^2$, we have, with $\eta$ given by~\eqref{Estimate07},
  \begin{equation}\label{etaest}
  \left|\Psi(t,\omega,x)\right|^2
  \leq 2 \eta(t,\omega,\eta_0 )\,,
  \end{equation}
  which implies that the solution $\Psi(t,\omega,x)$ exists for all $t\ge0$.
\end{lemma}
\begin{proof}
Writing $X$ as $(X_1,X_2)^\top$ and $O^*$ as $(O_1^*,O_2^*)^\top$, we rewrite \eqref{RandomDifferentialEquation} as
  \begin{align*}
 \left(
  \begin{array}{l}
  \dot X_1(t)\\
  \dot X_2(t)
  \end{array}
  \right)
  &=
 -
  \left(
  \begin{array}{l}
  X_1(t)+ O_1^*(\theta_t\omega) \\
  X_2(t)+  O_2^*(\theta_t\omega)
  \end{array}
  \right)
  + \alpha
  \left(
  \begin{array}{l}
   O_1^*(\theta_t\omega) \\
   O_2^*(\theta_t\omega)
  \end{array}
  \right)
  \\
  &\quad
  +
  \rho \left( \left|
  \left(
  \begin{array}{l}
  X_1(t)+ O_1^*(\theta_t\omega) \\
  X_2(t)+  O_2^*(\theta_t\omega)
  \end{array}
  \right)
  \right| \right)
  \left(
  \begin{array}{l}
   \phantom{-}X_2(t)+  O_2^*(\theta_t\omega)\\
 -X_1(t)-  O_1^*(\theta_t\omega)
  \end{array}
  \right)\,.
  \end{align*}
    Let $R(t)=\sqrt{X_1(t)^2+X_2(t)^2}$. 
  A direct computation yields
  \begin{align*}
   \frac{\dd}{\dd t} \left(\frac{1}{2} R(t)^2\right)
  = & X_1(t)\dot X_1(t)+ X_2(t)\dot X_2(t)\\
  = &
  -R(t)^2 + (\alpha-1)\,\big(O_1^*(\theta_t\omega) \,X_1(t) +  O_2^*(\theta_t\omega) X_2(t)\big) \\
  &+ \rho \left(\left|
  \left(
  \begin{array}{l}
  	 X_1(t)+ O_1^*(\theta_t\omega) \\
  X_2(t)+  O_2^*(\theta_t\omega)
  \end{array}
  \right)
  \right| \right) \left(   O_2^*(\theta_t\omega) X_1(t) -  O_1^*(\theta_t\omega) X_2(t)\right)\\
   \leq & - R(t)^2+(\alpha-1)\,\big(O_1^*(\theta_t\omega) \,X_1(t) +  O_2^*(\theta_t\omega) X_2(t)\big) \\
 &+ C  \big( 1+R(t)  +  |O^*(\theta_t \omega) |  \big) R(t)\, |O^*(\theta_t \omega) |  \\
   \leq & - R(t)^2 +  \left|\alpha-1\right|\, R(t)\, |O^*(\theta_t \omega) |\\
   &+  C \, |O^*(\theta_t \omega) | \left( R(t) + R(t)^2 + R(t)\, |O^*(\theta_t \omega) |  \right)  \\
  \leq & -  \frac{1}{2} R(t)^2 + \frac 12 \left|\alpha-1\right|^2  |O^*(\theta_t \omega) |^2 \\
  &+  C \, |O^*(\theta_t \omega) | \left( \frac{1}{2} + 2 R(t)^2 + \frac{1}{2}|O^*(\theta_t \omega) |^2 \right),
  \end{align*}
   using Young's inequality for products several times in the last step.
  Setting
    \begin{equation}
    \label{eq:tempered_processes}
    \begin{aligned}
        \gamma(\omega) &\coloneqq \frac{1}{2}\, C\, |O^*(\omega)| + \frac{1}{2} \left|\alpha-1\right|^2  |O^*(\omega) |^2 + \frac{1}{2}\,C\,  |O^*(\omega) |^3, \\
        \beta(\omega) &\coloneqq  -1 + 4 C\, |O^*(\omega) |,
    \end{aligned}
    \end{equation} 
    we can use a comparison argument to obtain for all $t\ge0$, $x\in\mathbb R^2\setminus\{0\}$ and $\eta_0 \geq |x|^2 $ 
    that $\frac{1}{2}|\Psi(t,\omega,x)|^2\leq \eta(t,\omega,\eta_0)$, where $t\mapsto \eta(t,\omega,\eta_0)$ is the solution~\eqref{Estimate07} to the RDE~\eqref{Estimate_06} with initial value $\eta_0$. 
    This shows the claim.
\end{proof}
Given the constant $C > 0$, we can always choose $\alpha$ in~\eqref{OU} large enough such that
$$\frac{1}{4}\, \mathbb E\left[|O^*|\right]^{-1} = \frac{1}{2} \sqrt{\frac{ \alpha}{\pi}} > C.$$
Hence, $\bar \beta \coloneqq \mathbb E [ \beta ] < 0$, for $\beta$ given in~\eqref{eq:tempered_processes}.

\begin{lemma}\label{lemm:Ktempered}
    The random variable $L(\omega)$ given by 
$$ L(\omega) \coloneqq \int_{-\infty}^0 \exp\left(\int_s^0 \beta(\theta_u \omega) \,\dd u\right) \gamma(\theta_s \omega)\,\dd s $$
is tempered.
\end{lemma}
\begin{proof}
Fix some constant $0<c< -\bar \beta$.
Define two random variables $L_1, L_2: \Omega \to [0,\infty]$ by 
\begin{align*}
    L_1(\omega) &\coloneqq \sup_{s \leq 0} \exp\left(-cs+ \int_{s}^0 \beta (\theta_{u} \omega)\, \dd u\right)\\
    L_2(\omega) &\coloneqq \int_{-\infty}^0 e^{cs} \gamma(\theta_s \omega) \, \dd s
\end{align*}
A simple calculation shows that
\begin{align*}
    L(\omega) &= \int_{-\infty}^0 \exp\left(\int_s^0 \beta(\theta_u \omega) \,\dd u\right)\gamma(\theta_s \omega)\,\dd s \\
    &= \int_{-\infty}^0 \exp\left(-cs+\int_s^0 \beta(\theta_u \omega) \,\dd u\right)e^{cs}\, \gamma(\theta_s \omega)\,\dd s \leq L_1(\omega)\, L_2(\omega),
\end{align*}
for each $\omega \in \Omega.$ We complete the proof by showing that both $L_1$ and $L_2$ are tempered.

Let $0<\varepsilon< -\bar \beta -c$. By ergodicity, there exists a full-measure set $\widehat \Omega \subseteq \Omega$ and a random constant $C_\varepsilon(\omega)$, $\omega \in \widehat \Omega$ such that
$$\left|\int_0^t \beta(\theta_s\omega)\, \dd s - \bar \beta t\right| \leq C_\varepsilon(\omega) + \varepsilon|t|, \quad \text{for all } t\in \mathbb R, \, \omega \in \widehat \Omega.$$
For any $t\in \mathbb R$ and $\omega \in \widehat \Omega$ we can now bound
\begin{align*}
    \log L_1(\theta_t \omega) &= \sup_{s \leq 0} \left(-cs+ \int_{s}^0 \beta (\theta_{t+u} \omega)\, \dd u\right) = \sup_{s\leq t}\left(c(t-s)+\int_{s}^t \beta (\theta_{u} \omega)\, \dd u\right)\\
    &= \sup_{s\leq t}\left(c(t-s)+\int_{0}^t \beta (\theta_{u} \omega)\, \dd u - \int_{0}^s \beta (\theta_{u} \omega)\, \dd u\right) \\
    &\leq \sup_{s\leq t}\big(c(t-s)+\bar \beta (t-s) +\varepsilon (|t|+|s|) +2 C_\varepsilon(\omega)\big)\\
    &\leq \sup_{s\leq t} \big((c+\bar\beta +\varepsilon)(t-s) +2 \varepsilon|t| +2 C_\varepsilon(\omega)\big) =2\varepsilon |t| +2 C_\varepsilon(\omega).
\end{align*}
In particular, we have
$$\limsup_{t \to \pm \infty} \frac{1}{|t|}\log L_1(\theta_t \omega) \leq 2\varepsilon, \quad \text{almost surely}.$$
Since $\varepsilon>0$ can be chosen arbitrarily close to zero, $L_1$ is tempered.

For $L_2$, note that $\gamma$ is a polynomial of $|O^*|$ and hence tempered. Thus, for any $0<\varepsilon<c$, there exist a full measure set $\widetilde \Omega\subseteq \Omega$ and a random constant $C'_\varepsilon(\omega),$ $\omega \in \widetilde \Omega$, such that
$$\gamma(\theta_t \omega) \leq C'_\varepsilon(\omega) \,e^{\varepsilon|t|}, \quad \text{for all } t \in \mathbb R, \, \omega \in \widetilde \Omega.$$
For any $t\in \mathbb R$ and $\omega \in \widetilde \Omega$ we can bound
\begin{align*}
    L_2(\theta_t \omega) &= \int_{-\infty}^0 e^{cs}\,\gamma(\theta_{t+s} \omega) \, \dd s =  \int_{-\infty}^t e^{-c(t-s)}\,\gamma(\theta_{s} \omega) \, \dd s\leq C_\varepsilon'(\omega) \int_{-\infty}^t e^{-c(t-s) +\varepsilon |s|}\, \dd s\\
    &\leq  C_\varepsilon'(\omega) \,e^{\varepsilon|t|}\,\int_{-\infty}^t e^{-(c-\varepsilon)(t-s)}\, \dd s = \frac{C_\varepsilon'(\omega)}{c-\varepsilon} e^{\varepsilon|t|}
\end{align*}
In particular, we have
$$\limsup_{t \to \pm \infty} \frac{1}{|t|}\log L_2(\theta_t \omega) \leq \varepsilon, \quad \text{almost surely}.$$
Again, $\varepsilon>0$ can be chosen arbitrarily close to zero, so $L_2$ is tempered as well, completing the proof.
\end{proof}

We can now conclude the announced proof.

\begin{proof}[Proof of Theorem~\ref{theo:attractors1} (b)]
There exists a local random dynamical system generated by solutions of \eqref{RandomDifferentialEquation}. Due to Lemma~\ref{TechnicalLemma}, the solution $\Psi(t,\omega,x)$, indeed, exists for all $t\ge0$ under assumption~\eqref{eq:PBAtt}. Hence, we can apply Lemma~\ref{TransformSystems} (ii) and it sufficient to show that the random dynamical system  $\Psi$ has a pull-back set attractor. For this, we construct an arbsorbing set and use Theorem \ref{theo:ExistencePullback}.

Let $D(\omega)$  be a tempered random compact set. By definition, there exists a tempered random variable $Q:\Omega\rightarrow (0,\infty)$ and a full-measure set $\widehat \Omega \subseteq \Omega$ such that $D(\omega)\subset B_0({Q(\omega)})$ for all $\omega \in \widehat \Omega$.
Choose again some constant $0<c<-\bar \beta$. By ergodicity and possibly restricting to a smaller full-measure set $\widehat \Omega \subseteq \Omega$, for every $\omega \in \widehat \Omega$ there is a time $t_0 > 0$ such that 
\begin{equation}\label{eq:ergodicity}
\int_{-t}^0 \beta (\theta_{s} \omega)\, \dd s \leq -c\,t, \quad \text{for all } t > t_0.
\end{equation} 
By Lemma~\ref{TechnicalLemma}, for all $t > t_0$ and $x\in D(\theta_{-t}\omega)$, we have 
  \begin{align*}
       |\Psi(t,\theta_{-t}\omega,x)|^2&\leq 2 \eta(t,\theta_{-t}\omega, Q(\theta_{-t}\omega)^2)\\
   &\leq
  2 e^{- c t} Q(\theta_{-t}\omega)^2 + 2 \int_{-t}^0 \exp\left(\int_s^0 \beta(\theta_u \omega) \,\dd u\right) \gamma(\theta_s \omega)\,\dd s\\
  &\leq
  2 e^{- c t} Q(\theta_{-t}\omega)^2 + 2 L(\omega),
  \end{align*}
  where we use \eqref{Estimate07} to obtain the second inequality.
Since $Q$ is tempered, it follows that $\lim_{t \to \infty} e^{- c t} Q(\theta_{-t}\omega)^2=0$. Define 
  $$\hat R(\omega):=\sqrt{1+2L(\omega)},$$
  which is also tempered.
  There exists $T\geq t_0 >0$ such that
  \[
  \Psi(t,\theta_{-t}\omega,D(\theta_{-t}\omega))
  \subset
  \bar B_0(\hat R(\omega)) \text{ for all } t\geq T,
  \]
  where $\bar B_0(\hat R(\omega))$ denotes the centered closed ball of radius $\hat R(\omega)$. Consequently, $\bar B_0(\hat R(\omega))$ is an absorbing set for each $\omega \in \widehat \Omega$. The existence of a random pull-back set attractor for $\Psi$ follows via Theorem \ref{theo:ExistencePullback} and, hence, we can apply Lemma~\ref{TransformSystems} (ii) to obtain the claim.
  \end{proof}

  \section{Strong Completeness and Existence of Set Attractors (Proof of Theorem~\ref{theo:attractors2})}\label{sec:Attr2}
In this section we prove Theorem \ref{theo:attractors2}, which gives different conditions for strong completeness and existence of a weak set attractor under growth bounds on the derivative of $\rho$. 

\subsection{General Criterion}\label{sec:Attr2GC}
We first establish a criterion for strong completeness and existence of weak set attractors, which holds for a general filtered local RDS on $\mathbb R^d$ which satisfies the strong Markov property \eqref{eq:StrongMarkov}. The proof is inspired by \cite{LeimbachMattinglyScheutzow22}.

\begin{lemma}\label{lemm:ballcrit}Consider a filtered local RDS $(\theta, \varphi)$ on $\mathbb R^d$ that satisfies the strong Markov property \eqref{eq:StrongMarkov}.
For an increasing sequence of positive radii $0 < r_1 < r_2 < \dots$ with $r_n\to \infty$, let $B_n := B_{0}(r_n)$ be the centered open ball with radius $r_n$ and let $C_n \coloneqq \bar B_n$ be its closure. Define two sequences of random times $\tau^\uparrow_1, \tau^\uparrow_2, \dots$ and $\tau^\downarrow_1, \tau^\downarrow_2, \dots$ by
\begin{align*}
    \tau^\uparrow_n(\omega) &:= \inf\{t \geq 0 : \varphi(t, \omega, C_n) \nsubseteq B_{n+1}\}, ~ \forall\, n\geq 1,\\
    \tau^\downarrow_n(\omega) &:= \inf\{t \geq 0 : t< \tau^\star(\omega,C_n)\text{ and }\varphi(t, \omega, C_n) \subseteq C_{n-1}\}, ~ \forall\, n\geq 2
\end{align*}
and $\tau^\downarrow_1(\omega) = \infty$. 
\begin{enumerate}
    \item[(i)] If there exists a sequence $t_1, t_2, \dots$ of positive times with  $\sum_{n = 1}^\infty t_n = \infty$, such that 
    $$\liminf_{n \to \infty} \mathbb P (\tau^\uparrow_n \geq t_n) > 0,$$
    then $\varphi$ is strongly complete.
    \item[(ii)] If there exists a sequence $t_1, t_2, \dots$ of positive times with $t_n \to 0$, such that 
    \begin{equation}\label{ineq:ballcritattrcond}
        \liminf_{n \to \infty} \mathbb P (\tau^\downarrow_n < \tau^\uparrow_n \wedge t_n) > \frac{1}{2},
    \end{equation}
    then $\varphi$ admits a weak set attractor.
\end{enumerate}
\end{lemma}

\begin{figure}
    \centering
    \includegraphics[width = 0.3 \textwidth]{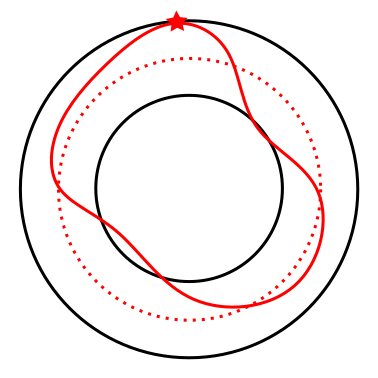}~~~~~~
    \includegraphics[width = 0.3 \textwidth]{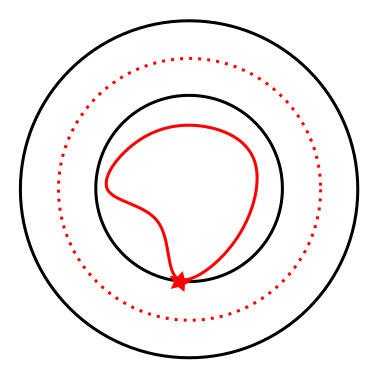}
    \caption{Illustration of hitting times $\tau^\uparrow_n$ (left) and $\tau^\downarrow_n$ (right). The dotted red line represents $\partial B_n$, the inner black circle $\partial B_{n-1}$, the outer black circle $\partial B_{n+1}$ and the solid red line $\partial\varphi(\tau^{\uparrow/\downarrow}_n(\omega), \omega, C_n)$.}
    \label{fig:tauUpDown}
\end{figure}

\begin{proof}
    First note that, by joint continuity of $\varphi(\cdot, \omega, \cdot)$, the process 
    $$(t, \omega) \mapsto \begin{cases} \max_{x\in C_n} |\varphi(t,\omega, x)|,&\text{if } t<\tau^\star(\omega,C_n)\\
    \infty,&\text{otherwise,}
    \end{cases}$$ 
    is continuous as a $[0,\infty]$-valued process and $(\mathcal F_t)$-adapted since the uncountable supremum can be replaced by a supremum over $C_n \cap \mathbb Q^d$. Here, we used the fact that $\tau^\star(\omega,C_n)$ is a $(\mathcal F_t)$-stopping time, see Proposition \ref{prop:tau} (i). Consequently, the random times $\tau^\uparrow_1, \tau^\uparrow_2, \dots$ and $\tau^\downarrow_1, \tau^\downarrow_2, \dots$ are $(\mathcal F_t)$-stopping times. Furthermore, continuity implies
    \begin{align*}
        \tau^\uparrow_n(\omega)<\infty \quad &\Rightarrow \quad \varphi(\tau^\uparrow_n(\omega), \omega, C_n) \subseteq C_{n+1}, \quad \text{and}\\
        \tau^\downarrow_n(\omega)<\infty\quad &\Rightarrow \quad \varphi(\tau^\downarrow_n(\omega), \omega, C_n) \subseteq C_{n-1},
    \end{align*}
    and that these stopping times are almost surely positive.
    
    (i): Since countable intersections of almost sure events are almost sure, it is sufficient to show that for every $n \geq 1$ we have $\mathbb P(\forall\, z \in C_n, \, \tau^\star(\omega, z) = \infty) = 1.$ Let $n \geq 1$. For each $m\geq n$ define a random time $\tau^\uparrow_{n,m}$, recursively, by 
    \begin{align*}
        \tau^\uparrow_{n, n}(\omega) &= 0, \quad \text{and}\\
        \tau^\uparrow_{n,m+1}(\omega) &= \tau^\uparrow_{n,m}(\omega) + \left(\tau^\uparrow_{m}(\theta_{\tau^\uparrow_{n,m}(\omega)}\omega) \wedge t_m\right), ~ \forall \, m \geq n.
    \end{align*}
    According to Proposition \ref{prop:tau} (iii), these are $(\mathcal F_t^+)$-stopping times.
    By construction, we have
    \begin{align*}
        \varphi(t, \omega, C_n) \subseteq B_m, ~ \forall \, t &< \tau^\uparrow_{n,m}(\omega) &&\Rightarrow& \tau^\star(\omega, z) &\geq \tau^\uparrow_{n,m}(\omega), ~\forall \, z \in C_n.
    \end{align*}
    If we can show 
    \begin{equation}\label{waitingtimeconv}
        \sum_{k = n}^\infty \tau^\uparrow_k(\theta_{\tau^\uparrow_{n,k}(\omega)}\omega) \wedge t_k = \lim_{m \to \infty} \tau^\uparrow_{n,m}(\omega) = \infty, \text{ almost surely,}
    \end{equation}
     strong completeness follows. In the following, we write $\tilde \tau^\uparrow_k := \tau_k^\uparrow \circ \theta_{\tau^\uparrow_{n,k}}$ for ease of notation. Note that $\tau_k^\uparrow$ and $\tilde\tau_k^\uparrow$ are identically distributed. By the strong Markov property \eqref{eq:StrongMarkov}, the family of random variables $(\tilde \tau^\uparrow_k \wedge t_k)_{k \geq n+1}$ is independent. Thus \eqref{waitingtimeconv} is a tail event in the sense of Kolmogorov's 0-1-law and it is sufficient to show that \eqref{waitingtimeconv} holds with positive probability. 
     Without loss of generality, we may assume $t_n \leq 1$, as
     $\sum_{n = 1}^\infty t_n = \infty$ also implies $\sum_{n = 1}^\infty (t_n\wedge 1) = \infty$. Now we can bound
     \begin{align*}
         \sum_{k =n}^\infty \tau^\uparrow_k(\theta_{\tau^\uparrow_{n,k}(\omega)}\omega)\wedge t_k =\sum_{k =n}^\infty \tilde \tau^\uparrow_k(\omega) \wedge t_k &\geq \sum_{k =n}^\infty t_k\, \mathds 1_{\tilde \tau^\uparrow_k(\omega) \geq t_k}.
     \end{align*}
     By Kolmogorov's three series theorem, the series on the right hand side diverges with positive probability, as the summands are bounded from above by 1 and satisfy
     \begin{align*}
         \sum_{k =n}^\infty \mathbb E\left[t_k\, \mathds 1_{\tilde \tau^\uparrow_k(\omega) \geq t_k}\right] = \sum_{k =n}^\infty t_k \,\mathbb P\left(\tilde \tau^\uparrow_k(\omega) \geq t_k\right) = \infty.
     \end{align*}
    (ii): The existence of a weak set attractor is established by verifying the conditions of Theorem \ref{theo:existenceWeakSetAtt} and the proof is inspired by the arguments in \cite[Section 7]{LeimbachMattinglyScheutzow22}. We recursively define $(\mathcal F_t^+)_{t \geq 0}$-stopping times (see Proposition \ref{prop:tau} (iii) again) $(\sigma_{n,m}(\omega))_{n \in \mathbb N,\, m \in \mathbb N_0, \, \omega \in \Omega}$ and random variables $(N_{n,m}(\omega))_{n \in \mathbb N,\, m \in \mathbb N_0, \, \omega \in \Omega}$ by $\sigma_{n,0}(\omega) = 0$ and $N_{n,0}(\omega) = n$, for all $n \in \mathbb N$ and all $\omega \in \Omega$, as well as
    $$\sigma_{n,m+1}(\omega) = \sigma_{n,m}(\omega) + \left(\tau^\downarrow_{N_{n,m}(\omega)}(\theta_{\sigma_{n,m}(\omega)}\omega) \wedge \tau^\uparrow_{N_{n,m}(\omega)}(\theta_{\sigma_{n,m}(\omega)}\omega) \wedge t_{N_{n,m}(\omega)}\right)$$
    and
    $$N_{n,m+1}(\omega) = \begin{cases}N_{n,m}(\omega) - 1,& \text{if } \tau^\downarrow_{N_{n,m}(\omega)}(\theta_{\sigma_{n,m}(\omega)}\omega) < \tau^\uparrow_{N_{n,m}(\omega)}(\theta_{\sigma_{n,m}(\omega)}\omega) \wedge t_{N_{n,m}(\omega)} \\
    N_{n,m}(\omega) + 1,& \text{otherwise.}\end{cases}$$
    Since $\tau^\uparrow_n$ and $\tau^\downarrow_n$ are almost surely positive, $\sigma_{n,m}$ is strictly increasing in $m$. Using induction in $m$, one can readily check that
    \begin{equation}\label{eq:BallInBall}
        \varphi(t, \omega, C_n) \subseteq B_{N_{n,m}(\omega) +1}, ~\forall\, t \in [\sigma_{n, m}(\omega), \sigma_{n, m+1}(\omega)).
    \end{equation}
    Furthermore, for each fixed $n \in \mathbb N$, the pair $(N_{n,\cdot}, \sigma_{n, \cdot})$ defines a \emph{Markov renewal process} in the sense of \cite[Chapter 10: Definition 1.1]{Cinlar75} (see~also \cite[Definition 7.2]{LeimbachMattinglyScheutzow22}).    
    Thus, the process $(N_{n,\cdot})$ is a discrete time, time homogeneous Markov process with  transition probabilities given by 
    $$\mathbb P\left(N_{n,m+1} = n'-1 \mid N_{n,m} = n'\right) = p_{n'}$$
    and 
    $$\mathbb P\left(N_{n,m+1} = n'+1 \mid N_{n,m} = n'\right) = 1-p_{n'},$$
    where $p_{n'} := \mathbb P (\tau^\downarrow_{n'} < \tau^\uparrow_{n'} \wedge t_{n'})$. In particular, the transition probabilities do not depend on $n$. 
    By \eqref{ineq:ballcritattrcond} there exists a $p^*> \frac{1}{2}$ and an $n^* \in \mathbb N$, such that $p_{n'} \geq p^*$ for each $n'\geq n^*$. Thus, the set $[1, n^*-1] \cap \mathbb N$ is recurrent and for any $n \in \mathbb N$ and for almost every $\omega \in \Omega$ there are infinitely many $m \in \mathbb N_0$ with $N_{n,m}(\omega) \leq n^*-1$. Furthermore, there exists an $\varepsilon>0$ such that $\mathbb P(\tau^\downarrow_{n'} \wedge \tau^\uparrow_{n'} \wedge t_{n'}\geq \varepsilon) \geq \varepsilon$ for all $1\leq n'<n^*$. By the Borel Cantelli Lemma and the strong Markov property, this implies that 
    $\sigma_{n,m+1}-\sigma_{n,m}\geq \varepsilon$ infinitely often almost surely, and consequently
    $\lim_{m \to \infty}\sigma_{n,m}(\omega) = \infty$, for every $n \in \mathbb N$ and almost every $\omega \in \Omega.$ Thus, for each $n \in \mathbb N$ and each $t \geq 0$, the random variable $M_{n,t}: \Omega \to \mathbb N_0$ uniquely given by 
    $$t \in \left[\sigma_{n, M_{n,t}(\omega)}(\omega), \sigma_{n, M_{n,t}(\omega)+1}(\omega)\right)$$
    is well defined for almost every $\omega \in \Omega$. This allows us to introduce the \emph{semi-Markov process} (see \cite[Definition 10.1.15]{Cinlar75}) $(\hat N_{n,t}(\omega))_{n \in \mathbb N, \, t\geq 0, \omega \in \Omega}$ by $\hat N_{n,t}(\omega) := N_{n,M_{n,t}(\omega)}(\omega).$ Now \eqref{eq:BallInBall} can be rewritten as
    \begin{equation}\label{eq:BallInBall2}
        \varphi(t, \omega, C_n) \subseteq B_{\hat N_{n,t}(\omega)+1}, ~\forall\,t\geq 0.
    \end{equation}

    Let $\nu_{n,m}$ denote the law of $N_{n,m}$. Standard Markov chain arguments (cf.~e.g.~\cite{MeynTweedie93,BenaimHurth22}) show that there exists a probability measure $\nu$ on $\mathbb N$, which is stationary for the Markov processes $(N_{n,\cdot})$, such that
    $$\nu = \lim_{m \to \infty} \frac{1}{m} \sum_{k=1}^m \nu_{n,k}, ~\forall\, n \geq n^*.$$
    Restricting our attention to $n$ which lie in the support of $\nu$, the Markov process $(N_{n,\cdot})$ is irreducible and recurrent. 
    If the support of $\nu$ is bounded, the conditions of Theorem \ref{theo:existenceWeakSetAtt} can be verified easily, so we assume that $\nu$ has unbounded support. 
    Since $t_n \to 0$ as $n \to \infty$, this implies
    $$\liminf_{m \to \infty} \big(\sigma_{n,m+1}(\omega)-\sigma_{n,m}(\omega)\big) \leq \liminf_{m \to \infty} t_{N_{n,m}(\omega)} = 0,$$
    for all $n \in \operatorname{supp}(\nu)$ and almost every $\omega \in \Omega$. Therefore, the Markov renewal process must be aperiodic and we can invoke \cite[Chapter 10: Theorem 5.22]{Cinlar75} (see~also \cite[Theorem 7.5]{LeimbachMattinglyScheutzow22}) to find that
    $$\lim_{t \to \infty} \mathbb P(\hat N_{n,t} = n') = \tilde \nu(n'), ~\forall\, n' \in \operatorname{supp}(\nu),$$
    where $\tilde \nu$ is the probability measure on $\mathbb N$ given by
    $$\tilde\nu(n) = \left(\sum_{k = 1}^\infty \nu(k)\,\mathbb E[\sigma_{k,1}]\right)^{-1} \nu(n)\, \mathbb E[\sigma_{n,1}].$$
    Note that the sum is finite since $\mathbb E[\sigma_{k,1}]$ can be bounded by $t_k\leq\max_{n \geq 1} t_n < \infty$ and positive since $\nu$ is a probability measure and $\sigma_{k,1}$ is almost surely positive.

 We are now in a position to verify the condition of Theorem \ref{theo:existenceWeakSetAtt}. Let $\varepsilon > 0$, choose $\hat n$ such that $\tilde \nu([1, \hat n]) > 1 - \frac{\varepsilon}{2}$ and set $C := C_{\hat n+1}$. Let $B \subset \mathbb R^d$ be some bounded set and choose $n \geq n^*$ such that $B \subseteq C_n$. Choose $T \geq 0$ large enough, such that 
  $$\mathbb P(\hat N_{n,t} \leq \hat n) > 1-\varepsilon,~\forall\, t\geq T.$$
 If $\hat N_{n,t} \leq \hat n$, then, by \eqref{eq:BallInBall2}
 $$\varphi(t, \omega, B) \subseteq \varphi(t, \omega, C_n) \subseteq C_{\hat N_{n,t}(\omega)+1} \subseteq C_{\hat n+1} = C,$$
completing the proof.
\end{proof}

\subsection{Contraction Estimates}

For the remainder of this section, we show that the conditions in Theorem \ref{theo:attractors2} imply the corresponding conditions in Lemma \ref{lemm:ballcrit}. Using the random coordinate transform from Section \ref{sec:ProofAttrI} to reduce the SDE \eqref{SDE2} to a random differential equation, is not helpful here, since it does not leave the disks $C_n$ invariant. However, we can still use random rotations, to reduce the complexity of the drift term.

We denote by $\mathcal R: \mathbb R \to \mathbb R^{2\times 2}$ the rotation matrix
  $$\mathcal R (\chi) := \begin{pmatrix}\cos(\chi)&-\sin(\chi)\\\sin(\chi)&\cos(\chi)\end{pmatrix}.$$
  \begin{lemma}\label{lemm:contraction}
Let $X(t) = \varphi(t, \omega, x)$ be a solution to \eqref{SDE2} and let $\chi_X: [0,\infty) \to \mathbb R$ be given by 
\begin{equation}\label{eq:chidef}
    \chi_X(t) = \int_0^t \rho(|X(s)|)\,\dd s.
\end{equation}
Then 
\begin{equation}\label{contractionBound}
    |\varphi(t, \omega, x)| \leq e^{-t}\,|x| +\sup_{s \in [0,t]} \left|\int_s^t \mathcal R(\chi_X(u)) \,\dd W(u)\right|.
\end{equation}
\end{lemma}
\begin{proof}
Let $Y: [0,\infty) \to \mathbb R^2$ and $\widetilde W: [0, \infty) \to \mathbb R^2$ be given by
\begin{align*}
    Y(t) &:= \mathcal R(\chi_X(t)) X(t) &&\text{and}& \widetilde W(t) := \int_0^t \mathcal R(\chi_X(s)) \,\dd W(s).
\end{align*}
A simple calculation shows that $Y$ satisfies the SDE
\begin{equation*}\label{sde:xtilde}
    \dd Y(t) = - Y(t) \,\dd t + \dd \widetilde W(t).
\end{equation*}
By variation of constants, $Y(t)$ is given by
$$Y(t) = e^{-t}\,Y(0) + \int_0^t e^{s-t} \,\dd\widetilde W(s).$$
Since $ |X(t)| = |Y(t)|$, we get the bound
$$|X(t)| \leq e^{-t}\, |X(0)| + \left| \int_0^t e^{s-t} \,\dd\widetilde W(s)\right|.$$
Finally, using integration by parts, we can estimate
\begin{align*}
   \left| \int_0^t e^{s-t} \,\dd\widetilde W(s)\right| &= \left| \widetilde W(t) - \int_0^t e^{s-t}\, \widetilde W(s) \,\dd s\right|\\
   &= \left| e^{-t}\left(\widetilde W(t)-\widetilde W(0)\right) + \int_0^t e^{s-t}\left(\widetilde W(t)-\widetilde W(s)\right) \,\dd s\right| \\
   &\leq  \sup_{s \in [0,t]}  \left|\widetilde W(t)-\widetilde W(s)\right| = \sup_{s \in [0,t]} \left|\int_s^t \mathcal R(\chi_X(u)) \,\dd W(u)\right|
\end{align*}
to complete the proof.
\end{proof}
The previous lemma is not directly useful, since the integral on the right hand side of \eqref{contractionBound} is highly dependent on $\chi_X$ and thus the exact path of the solution $(X(t))$. In order to control the behavior of bounded sets of initial conditions, we are interested in an upper bound for the integral in \eqref{contractionBound} which may depend on the noise realization, but crucially not on the precise initial condition $X(0)$. For this, it is helpful to note that, as long as $|X(t)|$ is contained in a compact interval $I \subset \mathbb R_{\geq 0}$, the derivative of $\chi_X$ can only take values in $\rho(I)$. Thus, we are looking for a way to control the integral in \eqref{contractionBound}, solely by the Lipschitz constant of $\chi_X$. With this in mind, for any $b\in \mathbb R$ and $\beta \geq 0$, define a set of functions $ L (b,\beta)$ by
$$L(b,\beta) := \{\chi: [0,\infty) \to \mathbb R : (t \mapsto \chi(t)-bt) \in \operatorname{Lip}(\beta)\},$$
where $\operatorname{Lip}(\beta)$ denotes the set of Lipschitz-continuous functions with Lipschitz constant less than or equal to $\beta$.
For each $T\geq 0$, (formally) define a random variable $Q_{b, \beta, T}$ by
\begin{equation}\label{eq:QDef}
    Q_{b, \beta, T} (\omega) := \sup_{0\leq s<t\leq T,~\chi \in L(b,\beta)} \left|\int_{s}^{t} \mathcal R(\chi(u))\,\dd W(u)\right|.
\end{equation}
If the integral on the right hand side is interpreted as an It\^o integral, this definition is problematic since it considers a supremum over uncountably many functions $\chi$ of an object that is defined only up to a $\chi$-dependent null set. Instead, since the integrand is of finite variation, we can interpret the integral as a pathwise Stieltjes integral via the integration by parts formula, leading to the definition
\begin{equation}\label{eq:QDef2}
    Q_{b, \beta, T} (\omega) := \sup_{0\leq s<t\leq T,~\chi \in L(b,\beta)} \left|\mathcal R(\chi(t))W(t)-\mathcal R(\chi(s))W(s) -\int_{s}^{t}\dd \mathcal R(\chi(u))\, W(u) \right|.
\end{equation}
For $Q_{b, \beta, T}$ to be a well-defined random variable, it remains to show measurability.
\begin{lemma}\label{lemm:youngloeve}
For any $b \in \mathbb R$, any $\beta, T\geq 0$ and every $q>0$, $Q_{b, \beta, T}$ is measurable and we have
\begin{enumerate}
    \item[(i)]  $\mathbb P (Q_{b, \beta, T} < q) > 0$ and 
    \item[(ii)] $\mathbb P (Q_{b, \beta, T} = \infty) = 0.$
\end{enumerate}
\end{lemma}
\begin{proof}
    Fix  $b \in \mathbb R$ and $\beta, T\geq 0$. Note that the space $L(b,\beta)$ in \eqref{eq:QDef2} may be replaced by
    $$L_T(b,\beta) := \{\chi: [0,T] \to \mathbb R : (t \mapsto \chi(t)-bt) \in \operatorname{Lip}(\beta)\}.$$
    For every $0\leq s \leq t\leq T$, and every $\chi \in L_T(b,\beta)$, we have 
    \begin{align*}
        &\quad\,\left|\mathcal R(\chi(t))W(t)-\mathcal R(\chi(s))W(s) -\int_{s}^{t} \dd \mathcal R(\chi(u))\, W(u)\right|\\
        &=\left|\mathcal R(\chi(t))(W(t)-W(s))+(\mathcal R(\chi(t))-\mathcal R(\chi(s)))W(s) -\int_{s}^{t} \dd \mathcal R(\chi(u))\,W(u) \right|\\
        &=\left|\mathcal R(\chi(t))(W(t)-W(s))-\int_{s}^{t} \dd \mathcal R(\chi(u))\, (W(u)-W(s))\right|.
    \end{align*}
    In particular, the term is continuous in $\chi$ once $L_T(b,\beta)$ is equipped with the $W^{1,1}$-topology. Since $L_T(b,\beta)$ is separable with respect to that topology, the uncountable supremum in \eqref{eq:QDef2} may be replaced by the supremum over a countable dense subset, giving measurability of $Q_{b, \beta, T}$.

    To show (i) and (ii), note that each function $\chi \in L(b, \beta)$ is $(|b|+\beta)$-Lipschitz continuous. This allows us to estimate
    \begin{align*}
        &\quad\,\left|\mathcal R(\chi(t))(W(t)-W(s))-\int_{s}^{t} \dd \mathcal R(\chi(u))\, (W(u)-W(s))\right|\\
        &\leq (1+(t-s)(|b|+\beta)) \sup_{s\leq u \leq t} |W(u)-W(s)|\\
        &\leq 2\,(1+T(|b|+\beta)) \left\|W|_{[0,T]}\right\|_\infty.
    \end{align*}
    In particular, we get
    $$Q_{b, \beta, T} \leq 2\,(1+T(|b|+\beta)) \left\|W|_{[0,T]}\right\|_\infty.$$
    Now, (ii) follows from the fact that Brownian paths are bounded almost surely on compact intervals and (i) from the support theorem for Brownian motion.
\end{proof}

The following lemma connects $Q_{b, \beta, T}$ to the stopping times $\tau^\uparrow_n$ and $\tau^\downarrow_n$ appearing in Lemma \ref{lemm:ballcrit}.

\begin{lemma}\label{lemm:contr2}
    Assume the setting of Lemma \ref{lemm:ballcrit}. Let $q_1, q_2, \dots$ be a sequence of positive real numbers satisfying 
    \begin{align}\label{eq:RQCond1}
       0\leq r_n-q_n < r_n + q_n &\leq r_{n+1}, ~\forall \, n \geq 1.
    \end{align}
    Define two sequences $b_1, b_2, \dots$ and $\beta_1, \beta_2, \dots$ by
    \begin{align*}
         b_n &:= \rho(r_n-q_n)&&
        \text{and}&\beta_n &:= 2q_n \sup_{r \in [r_n-3q_n, r_n + q_n]\cap [0,\infty)} |\rho'(r)|.
    \end{align*}
    \begin{enumerate}
        \item[(i)] For each $n\geq 1$ and each $t>0$, if $Q_{b_n,\beta_n, t}< q_n$ then $\tau^\uparrow_n > t.$ 
        \item[(ii)] Assume furthermore that
        \begin{align}\label{eq:RQCond2}
          0< r_n-3q_n <r_n-2q_n &\leq r_{n-1}, ~\forall \, n \geq 2,
    \end{align}
    and let the sequence $(T_n)$ be given by
   $$T_n = \frac{3q_n}{r_n-3q_n}.$$
        For each $n \geq 2$, if $Q_{b_n,\beta_n, T_n}< q_n$ then $\tau_n^\downarrow<T_n$, and by (i) even $\tau_n^\downarrow<\tau_n^\uparrow \wedge T_n$.
    \end{enumerate}
\end{lemma}

\begin{figure}
    \centering
    \begin{overpic}[width = 0.4 \textwidth, tics = 10]{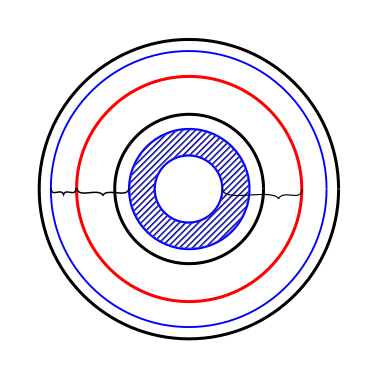}
    \put(15,43) {$q_n$}
    \put(22,43) {$2q_n$}
    \put(69,42) {$3q_n$}
    \end{overpic}
    \caption{Illustration of the relevant balls in the proof of Lemma \ref{lemm:contr2}. The red circle represents $\partial B_n$, the inner black circle $\partial B_{n-1}$ and the outer black circle $\partial B_{n+1}$. The three blue circles represent $\partial B_{0}(r_n-3q_n)$, $\partial B_{0}(r_n-2q_n)$ and $\partial B_{0}(r_n+q_n)$. The ``buffer zone" (shaded region) between the two innermost blue circles is crucial for the proof of Lemma \ref{lemm:contr2} (ii). It allows to control the trajectories of points coming from a region close to the origin before they enter the problematic region $\mathbb R^2 \setminus C_{n-1}$.}
    \label{fig:buffer}
\end{figure}

\begin{proof}
First, note that by the mean value theorem we have
\begin{equation}\label{eq:rhoRange}
    {\rho(r) \in [b_n-\beta_n, b_n+\beta_n]}, \quad \forall\, r \in [r_n-3q_n, r_n +q_n]\cap [0,\infty).
\end{equation}

(i): Let $n\geq 1$, $T\geq 0$ and $\omega\in \Omega$ be such that $Q_{b_n, \beta_n, T}< q_n$. Fix an arbitrary $x\in C_n$ and let $X(t) = \varphi(t,\omega,x)$.
Set
$$t^* := T \wedge \inf \{t \in [0,T] : |X(t)|\geq r_n+q_n\}.$$
Our goal is to show that $|X(t^*)| < r_n+q_n$ which implies $t^* = T$ and thus $X(t) \in B_{n+1}$ for all $t \in [0,T]$. 

If $X(t^*) \in C_n$, clearly we are done, so suppose $X(t^*) \notin C_n$. Let $t_0\in [0, t^*)$ be the last time the trajectory leaves $C_n$, that is,
$$t_0 \coloneqq \sup \{t \in [0,t^*] : X(t) \in C_n\}.$$
Set $x_0 \coloneqq X(t_0).$ By the continuity of the trajectory $X$, we have $|x_0|=r_n$. Lemma \ref{lemm:contraction} gives
\begin{align*}
    |X(t^*)| &= |\varphi(t^*-t_0, \theta_{t_0}\omega, x_0)| \leq e^{-(t^*-t_0)}\,|x_0| +\sup_{s \in [0,t^*-t_0]} \left|\int_{s}^{t^*-t_0} \mathcal R(\chi_X(t_0+u)) \,\dd W(t_0+u)\right|\\
    &\leq r_n + \sup_{s\in [t_0, t^*]} \left|\int_{s}^{t^*} \mathcal R(\chi_X(u)) \,\dd W(u)\right|.
\end{align*}
Since $|X(t)| \in [r_n, r_n+q_n]$ for all $t \in [t_0, t^*]$, we have $\chi_X'(t) \in [b_n-\beta_n, b_n+\beta_n]$ for all $t\in [t_0, t^*]$. In particular, we can find a function $\tilde \chi \in L(b_n,\beta_n)$ that agrees with $\chi_X$ on  $[t_0, t^*]$. We get
\begin{align*}
    |X(t^*)| &\leq r_n + \sup_{s\in [t_0, t^*]} \left|\int_{s}^{t^*} \mathcal R(\tilde \chi(u)) \,\dd W(u)\right|\leq r_n +Q_{b_n,\beta_n, T} < r_n +q_n.
\end{align*}

(ii): Let $n\geq 2$ and $\omega\in \Omega$ be such that $Q_{b_n, \beta_n, T_n}< q_n$. Again, fix an arbitrary $x \in C_n$. Our goal is to show $\varphi(T_n, \omega, x) \in C_{n-1}$. Showing this for all $x \in C_n$ gives $\varphi(T_n, \omega, C_n) \subseteq C_{n-1}$. Since (i) additionally gives $\varphi(t, \omega, C_n) \subseteq C_{n+1}$ for all $t \in [0, T_n]$ and in particular $\tau^\star(\omega, C_n) > T_n$, this allows us to conclude the desired bound on $\tau^\downarrow_n(\omega)$.

Again, let $X(t) = \varphi(t,\omega,x)$. By the proof of (i), we have $|X(t)| < r_n+q_n$ for all $t\in [0,T_n]$. We distinguish two cases. First, suppose that $X(t)$ leaves the region where $\rho$ is controlled, that is, $|\varphi(t, \omega, x)| < r_n- 3 q_n$, for some $t \in [0,T_n]$. Let $t_0 \in [0,T_n]$ be the largest such $t$, that is,
$$t_0 := \sup \{t \in [0,T_n] :|X(t)| < r_n- 3 q_n\}.$$
If $t_0 = T_n$, we are done since $r_n-3q_n < r_{n-1}$, so suppose $t_0 < T_n$. Set $x_0 \coloneqq X(t_0)$ and note that, by continuity, $|x_0|= r_n-3q_n$. Also, $$|X(t)| \in [r_n-3q_n, r_n +q_n], \quad \forall \, t \in [t_0, T_n].$$ 
Arguing analogously to (i), we can show
$$|X(T_n)| \leq |X(t_0)| + Q_{b_n, \beta_n, T_n} < r_n -3q_n+q_n \leq r_{n-1}.$$

It remains to treat the case where $X(t)$ does not leave the controlled region, that is, $|\varphi(t, \omega, x)| \in [r_n-3q_n, r_n +q_n]$ for all $t \in [0, T_n]$. Thus $\chi_X'(t) \in [b_n-\beta_n, b_n+\beta_n]$ for all $t \in [0, T_n]$ and there exists a $\tilde \chi \in L(b_n, \beta_n)$ that agrees with $\chi_X$ on $[0, T_n]$. Applying Lemma \ref{lemm:contraction} and using $e^{-T_n} \leq (1+T_n)^{-1}$, we get
\begin{align*}
    |X(T_n)| &\leq e^{-T_n}|X(0)| + \sup_{s \in [0,T_n]} \left|\int_{s}^{T_n} \mathcal R(\tilde \chi(u)) \,\dd W(u)\right|\\
    &\leq \frac{1}{1+T_n} r_n + Q_{b_n, \beta_n, T_n} < r_n-3q_n + q_n \leq r_{n-1}.
\end{align*}
\end{proof}

\subsection{Proof of Theorem \ref{theo:attractors2}}
The rest of the argument is based purely on scaling symmetries. The rotational and parabolic scaling symmetries of Brownian motion translate to scaling symmetries of $Q$. Denoting equality in distribution by ``$\eqdist$'', we obtain the following result.
\begin{lemma}\label{lemm:QScaling}
 For each $b \in \mathbb R$, $\beta \geq 0$ and $T \geq 0$, we have
\begin{enumerate}
    \item[i)] $Q_{b,\beta, T} \eqdist Q_{0,\beta, T}$ and 
    \item[ii)] $Q_{0, \kappa\beta, \kappa^{-1}T} \eqdist \kappa^{-1/2}\,Q_{0,\beta, T},$ for each $\kappa >0$.
\end{enumerate}
\end{lemma}
\begin{proof}
i): For $\chi \in L(b,\beta)$ let $\tilde \chi \in L(0,\beta)$ be defined by $\tilde\chi(t) = \chi(t)-bt$. Clearly, the assignment $\chi \mapsto \tilde \chi$ is a bijection between $L(b, \beta)$ and $L(0,\beta)$. Also, define a stochastic process $\widetilde W: [0, T] \to \mathbb R^2$ by 
$$\widetilde W(t) = \int_0^t\mathcal R(bs) \,\dd W(s).$$
By rotation invariance $\widetilde W$ is also a Brownian motion, so in particular $W \eqdist \widetilde W$.
In the following, it is convenient to use the expression \eqref{eq:QDef} for $ Q_{b,\beta,T}$. For this, the integral needs to be interpreted as a pathwise integral defined via the integration-by-parts representation \eqref{eq:QDef2}.
Now we have
\begin{align*}
    Q_{b,\beta,T} &= \sup_{0\leq s<t\leq T,\,\chi \in L(b,\beta)} \left|\int_{s}^{t} \mathcal R(\chi(u))\,\dd W(u)\right|\\
    &= \sup_{0\leq s<t\leq T,\,\chi \in L(b,\beta)} \left|\int_{s}^{t} \mathcal R(\tilde\chi(u))\mathcal R(bu) \,\dd W(u)\right|\\
    &= \sup_{0\leq s<t\leq T,\,\tilde \chi \in L(0,\beta)} \left|\int_{s}^{t} \mathcal R(\tilde \chi(u))\,\dd \widetilde W(u)\right|\\
    &\eqdist \sup_{0\leq s<t\leq T,\,\tilde \chi \in L(0,\beta)} \left|\int_{s}^{t} \mathcal R(\tilde \chi(u))\,\dd W(u)\right|=Q_{0, \beta, T}.
\end{align*}
ii): Fix $\kappa >0$. For each $\chi \in  L(0,\beta)$ define $\hat \chi \in L(0,\kappa \beta)$ by $\hat \chi(t) = \chi(\kappa t)$. Also define a process $\widehat W$ by $\widehat W(t) = \kappa^{-1/2} \,W(\kappa t)$. Due to the parabolic scaling $\widehat W$ is also a Brownian motion and thus $W \eqdist \widehat W$. Now we have
\begin{align*}
    \kappa^{-1/2} \,Q_{0,\beta, T} &= \sup_{0\leq s<t\leq T,\,\chi \in L(0,\beta)} \left|\int_s^t \mathcal R(\chi(u))\kappa^{-1/2}\, \dd W(u)\right|\\
    &= \sup_{0\leq s<t\leq T,\chi \in L(0,\beta)} \left|\int_{\kappa^{-1}s}^{\kappa^{-1}t} \mathcal R(\chi(\kappa u))\kappa^{-1/2}\,\dd W(\kappa u)\right|\\
    &= \sup_{0\leq \kappa^{-1}s<\kappa^{-1}t\leq\kappa^{-1}T,\,\hat\chi \in L(0,\kappa\beta)} \left|\int_{\kappa^{-1}s}^{\kappa^{-1}t} \mathcal R(\hat\chi(u))\, \dd \widehat W(u)\right| \eqdist Q_{0,\kappa \beta, \kappa^{-1}T}.
\end{align*}
\end{proof}
We are now in a position to prove Theorem \ref{theo:attractors2}.

\begin{proof}[Proof of Theorem \ref{theo:attractors2} (a)]
Let $\hat \rho(r) := 1\vee \sup_{r'\in [0,r]}|\rho'(r')|$. Suppose \eqref{eq:Compl} holds, that is,
$$\int_0^\infty \hat \rho(r)^{-1/3}\, \dd r = \infty.$$
We aim to show strong completeness by invoking Lemma \ref{lemm:ballcrit} (i).
Set $r_1 := 1$ and recursively define $r_{n+1}$ in such a way that $r_{n+1}-r_n = \hat \rho(r_{n+1})^{-1/3}$ holds. This is always possible and the choice is unique since $\hat \rho(r)^{-1/3}$ is monotonically decreasing and continuous. Also $r_n \to \infty$ since otherwise $r_{\infty} \coloneqq \lim_{n \to \infty} r_n<\infty$ would satisfy $\hat \rho(r_{\infty})^{-1/3} = r_{\infty}-r_{\infty} = 0$, leading to a contradiction.

Set
\begin{align*}
    q_n &:= \hat \rho(r_{n+1})^{-1/3},&\beta_n^* &:= 2\hat \rho(r_{n+1})^{2/3}&&\text{and}&t_n &:=\hat \rho(r_{n+1})^{-2/3}.
\end{align*}
 Also let $(b_n)$ and $(\beta_n)$ be the sequences defined in Lemma \ref{lemm:contr2}. By the definition of $\hat \rho$, we have $$\beta_n \leq 2q_n \,\hat \rho(r_n + q_n) = 2q_n \,\hat \rho(r_{n+1}) = 2\hat \rho(r_{n+1})^{2/3} = \beta^*_n.$$
Using Lemma \ref{lemm:contr2} (i) and the monotonicity of $Q_{b,\beta, t}$ with respect to $\beta$, for each $n\geq 1$, we get
$$\mathbb P(\tau_n^{\uparrow}\geq t_n)\geq \mathbb P (Q_{b_n,\beta_n, t_n}<q_n)\geq \mathbb P (Q_{b_n,\beta_n^*, t_n}<q_n).$$
Using Lemma \ref{lemm:QScaling} and Lemma \ref{lemm:youngloeve}, we furthermore get
$$\mathbb P(\tau_n^{\uparrow}\geq t_n)\geq \mathbb P (Q_{b_n,\beta_n^*, t_n}<q_n) = \mathbb P (Q_{b_1,\beta_1^*, t_1}<q_1)>0$$
and thus in particular $\liminf_{n \to \infty}\mathbb P(\tau_n^{\uparrow}\geq t_n)>0$. It remains to show $\sum_{n=1}^\infty t_n = \infty$.
Using the fact that $r_{n+1}-r_n = \hat \rho(r_{n+1})^{-1/3} \leq 1$ and the monotonicity of $\hat \rho$, we can compute
\begin{align*}
    \sum_{n=1}^\infty t_n &=  \sum_{n=1}^\infty \hat \rho(r_{n+1})^{-2/3} = \sum_{n=1}^\infty (r_{n+1}-r_n)\,\hat \rho(r_{n+1})^{-1/3} \\
    &\geq \sum_{n=1}^\infty (r_{n+1}-r_n)\,\hat \rho(r_{n}+1)^{-1/3} \geq \int_1^\infty \hat \rho(r+1)^{-1/3}\, \dd r = \infty.
\end{align*}
\end{proof}

\begin{proof}[Proof of Theorem \ref{theo:attractors2} (b)]
    Suppose $\rho$ satisfies \eqref{eq:Attr}, that is, there is a radius $r^*\geq 0$, such that 
    \begin{equation*}\label{ineq:rhoprimebound}
        |\rho' (r)| \leq C r^3, ~\forall \, r \geq r^*.
    \end{equation*}
    for some $0<C\leq K_1$ which is less than some universal constant $K_1>0$ to be determined. For a different constant $K>0$, also to be determined later, define sequences $(r_n)$, $(q_n)$, $(\beta_n^*)$ and $(t_n)$ by
    \begin{align*}
        r_n &:= K n^{1/2},& q_n &:= \frac{3}{8}Kn^{-1/2},& \beta_n^*&:= 3CK^4n,& t_n&:= 3n^{-1}.
    \end{align*}
    Furthermore, let $(b_n)$, $(\beta_n)$, and $(T_n)$ be sequences as defined in Lemma \ref{lemm:contr2}. The conditions \eqref{eq:RQCond1} and \eqref{eq:RQCond2} can be verified by the following elementary calculations:
\begin{align*}
     r_{n+1}-r_n &= K (\sqrt{n+1}-\sqrt{n}) = \frac{K}{\sqrt{n+1}+\sqrt{n}}\geq \frac{K}{(1+\sqrt{2})\sqrt{n}}> q_n, &&\forall n \geq 1,\\
    r_n-r_{n-1} &= K (\sqrt{n}-\sqrt{n-1}) = \frac{K}{\sqrt{n}+\sqrt{n-1}} \leq \frac{K}{(1+\sqrt{1/2})\sqrt{n}} \leq 2q_n,& &\forall \, n \geq 2,\\
    r_n-3q_n &= K\left(\sqrt{n}-\frac{9}{8\sqrt{n}}\right) \geq K\left(\sqrt{2}-\frac{9}{8\sqrt{2}}\right)>0,& &\forall \, n \geq 2.
\end{align*}
Let $n^* \in \mathbb N$ denote the smallest integer $n^*\geq 2$ satisfying $r_{n^*}-3q_{n^*}\geq r^*$. For all $n\geq n^*$, we have
\begin{align*}
    \beta_n &= 2q_n \sup_{r\in[r_n-3q_n, r_n+q_n]}|\rho'(r)| \leq 2 q_n C(r_n+q_n)^3 \leq 2 q_n Cr_{n+1}^3 \\
    &= \frac{3}{4}CK^4 n^{-1/2}(n+1)^{3/2}\leq \frac{3\cdot 2^{3/2}}{4}CK^4 n < 3CK^4 n = \beta_{n}^*, \text{ and}\\
    T_n &= \frac{3q_n}{r_n-3q_n} = \frac{\tfrac{9K}{8}n^{-1/2}}{Kn^{1/2}- \tfrac{9K}{8} n^{-1/2}} = \frac{1}{\tfrac{8}{9}n-1} < \frac{3}{n} = t_n.
\end{align*}
Note that the dependence on $n$ in $\beta_n$, $q_n$ and $t_n$ matches the scaling in Lemma \ref{lemm:QScaling}. We can use this, together with Lemma \ref{lemm:contr2} and the monotone dependence of $Q$ on $\beta$ and $t$, to get the bound

    \begin{align*}
        \mathbb P\left(\tau^\downarrow_n < \tau^\uparrow_n \wedge t_n \right) &\geq \mathbb P\left(\tau^\downarrow_n < \tau^\uparrow_n \wedge T_n \right) \geq   \mathbb P(Q_{b_n, \beta_n, T_n} < q_n)\\
        &\geq \mathbb P(Q_{b_n, \beta_n^*, t_n} < q_n)= \mathbb P\left(Q_{0, \beta_1^*, t_1} < q_1\right) 
    \end{align*}
    for all $n \geq n^*$. In particular, we have
    $$\liminf_{n \to \infty}\mathbb P\left(\tau^\downarrow_n < \tau^\uparrow_n \wedge t_n \right) \geq \mathbb P\left(Q_{0, \beta_1^*, t_1} < q_1\right),$$
    so in order to establish \eqref{ineq:ballcritattrcond}, it remains to show $\mathbb P\left(Q_{0, \beta_1^*, t_1} < q_1\right)>\frac{1}{2}$.
    Recall that the constant $K>0$ was arbitrary so far and fix $K=C^{-1/4}$, such that $\beta^*_1 = 3$ and $t_1 = 3$ no longer depend on $C$. As a consequence of Lemma \ref{lemm:youngloeve} (ii), there exists a constant $Q^*> 0$ such that $\mathbb P(Q_{0, 3,3} < Q^*) > \frac{1}{2}$. Set $K_1 = \frac{1}{51}{Q^*}^{-4}$. Then 
    $$q_1 = \frac{3}{8}K = \frac{3}{8}C^{-1/4} \geq \frac{3}{8}K_1^{-1/4} = \frac{3}{8} 51^{1/4} Q^*>Q^*.$$
     and thus indeed $\mathbb P\left(Q_{0, \beta_1^*, t_1} < q_1\right)>\frac{1}{2}$. Therefore, \eqref{ineq:ballcritattrcond} holds and, by Lemma \ref{lemm:ballcrit}, the random dynamical system $\varphi$ admits a weak set attractor.

\end{proof}


\section{Lack of a random attractor and lack of strong completeness (Theorem~\ref{theo:BlowUp})}\label{sec:BlowUp}
In this section, we state the full version of Theorem \ref{theo:BlowUp} and present its proof.
 For a subset $\mathcal M\subseteq \mathbb R^2$, we call the supremum of the absolute values of points in $\mathcal M$ the \emph{radius of $\mathcal M$}, denoted by
$$\operatorname{rad}(\mathcal M) = \sup_{x \in \mathcal M} |x| \in [0,
\infty].$$

\begin{maintheorem}[Full Version]\label{theo:BlowUp}
The following hold:
\begin{enumerate} 
\item[(a)] There exists a constant $K_2>0$, independent of $\rho$, with the property that if the function $\rho$, for some $r^*>0$,  satisfies
\begin{equation}
 |\rho'(r)|\geq K_2 \,r^3 \quad \forall \, r\geq r^*,
\tag{NoAtt} \label{eq:NoAttr2}
\end{equation}
as well as
\begin{equation}\label{eq:BasicAssumption2}
    \limsup_{r \to \infty}\frac{|\rho''(r)|}{|\rho'(r)|}<\infty, \quad \text{and}\quad |\rho'(r)| \text{ increases monotonically on }[r^*,\infty),
\end{equation}
then the following hold.

\begin{enumerate}
    \item[(a.i)] The RDS $\varphi$ does not admit a weak set attractor.

    \item[(a.ii)] In particular, for every compact, connected set $\mathcal M\subseteq \mathbb R^2$ that contains at least two points, there is a positive probability that the radius of its image under the flow grows at least exponentially, that is, 
$$\mathbb P\left(\tau^\star(\omega,\mathcal M) <\infty \text{ or } \liminf_{t \to \infty} \frac{1}{t} \log (\operatorname{rad}(\varphi(t,\omega,\mathcal M)))>0\right) >0.$$

    \item[(a.iii)] Moreover, for every $\varepsilon >0$, there exists a radius $R$ such that if $\mathcal M$ contains some point $x$ with $|x|\leq 1$ and satisfies $\operatorname{rad}(\mathcal M) \geq R$, then this probability is at least $1-\varepsilon$.
\end{enumerate}

\item[(b)]If, additionally, 
\begin{equation}
\int_{r^*}^\infty \big|\rho'(r)\big|^{-1/3}\,\dd r<\infty\tag{NoCom} \label{eq:NoCompl2},
\end{equation}
then the following hold.

\begin{enumerate}
    \item[(b.i)] The RDS $\varphi$ is not strongly complete.
    \item[(b.ii)]In particular, for every compact, connected set $\mathcal M\subseteq \mathbb R^2$ that contains at least two points, there is a positive probability that there exists a point $x \in \mathcal M$ and a time $T>0$, such that
    $$\lim_{t \nearrow T} |\varphi(t,\omega,x)| = \infty.$$ 
    \item[(b.iii)] Moreover, for every $\varepsilon >0$, there exists a radius $R$ such that if $\mathcal M$ contains some point $x$ with $|x|\leq 1$ and satisfies $\operatorname{rad}(\mathcal M) \geq R$, then with probability at least $1-\varepsilon$ this happens at time $T\leq\varepsilon$.

\end{enumerate}

\end{enumerate}
\end{maintheorem}

 The remainder of this section is divided into three parts. First, we introduce and prove a general criterion for non-existence of set attractors and lack of strong completeness, see Lemma \ref{lemm:BlowUpCrit} below. Then, we state and prove a key estimate required to apply Lemma \ref{lemm:BlowUpCrit} to our main SDE \eqref{SDE2}. Finally, we use this to give the proofs for Theorem \ref{theo:BlowUp} (a) and (b).

\subsection{General Criterion}\label{sec:BlowUpCrit}

We first establish a criterion for the lack of strong completeness and non-existence of weak set attractors. Due to its independent interest, we formulate the following Lemma for general filtered RDS satisfying the strong Markov property.

\begin{lemma}\label{lemm:BlowUpCrit}
Consider a filtered local RDS $(\theta, \varphi)$ on $\mathbb R^d$ which satisfies the strong Markov property.
    As in Lemma \ref{lemm:ballcrit}, let $0<r_1<r_2<\dots$ be an increasing sequence of radii with $r_n \to\infty$ and denote by $B_n = B_{0}(r_n)$ the centered open balls with radius $r_n$ and by $C_n$ their closures. 
    Let $\tau_n(\omega,x)$ denote the $(\mathcal F_t)$-stopping time 
    $$\tau_n(\omega,x) \coloneqq \inf\{t\geq 0: |\varphi(t,\omega,x)| = r_n\}.$$
    Depending on whether the radius of $x$ is larger or smaller than $r_n$, this describes either the first entry into the closed ball $C_n$ or the first exit of the open ball $B_n$. 
    Let $t_2 \geq t_3\geq\dots$ be a decreasing sequence of positive times, $p^*\in (\frac{1}{2},1]$ and $\alpha \in (0, 2p^*-1)$. Define a ``growth rate" $m: \mathbb N \times \mathbb R_{>0}\to \mathbb N \cup \{\infty\}$ by
    \begin{equation}\label{eq:Growthrate}
        m(n,t) \coloneqq \sup \left\{m \geq n: \frac{1}{\alpha}\sum_{k=n}^{m-1} t_{k+1} < t \right\}.
    \end{equation}
    Recall that the radius of a set $\mathcal M\subseteq \mathbb R^d$ is given by $\operatorname{rad}(\mathcal M) \coloneqq \sup_{x\in \mathcal M} |x|$. 
    For any $n\in \mathbb N$, we call a set $\mathcal M$ rapidly growing with respect to $n$, if 
\begin{equation}\label{eq:rapidGrow}
  \begin{cases}
      \tau^\star(\omega, \mathcal M) \leq t \quad \text{or} \quad \operatorname{rad}(\varphi(t,\omega, \mathcal M)) \geq r_{m(n,t)},& \text{for all } t>0 \text{ with } m(n,t) < \infty, \\
      \tau^\star(\omega, \mathcal M) \leq t,& \text{for all } t>0 \text{ with } m(n,t) = \infty.
  \end{cases}
    \end{equation}

    Suppose that
    \begin{enumerate}
        \item[(i)] for all $n\geq 2$ and all $x,y \in \mathbb R^d$ with $|x|=r_n$ and $|y|=r_{n-1}$, we have
        \begin{equation}\label{eq:growthBound}
            \mathbb P(\tau_{n+1}(x)\wedge \tau_{n+1}(y) \leq \tau_{n-1}(x) \wedge t_n) \geq p^*.
        \end{equation}
        \item[(ii)]
       We have
        \begin{equation}\label{eq:anchorBound}
            \lim_{n \to \infty}\inf_{x\in B_1}\mathbb P(|\varphi(t, \omega,x)| \leq r_{m(n,t)}, ~\forall\,t\in \mathbb R_{>0}) = 1,
        \end{equation}
        where we set $r_\infty \coloneqq \infty.$
    \end{enumerate}
    Let $\mathcal I_n$ denote the family of all connected compact subsets of $\mathbb R^d$ that intersect $B_1$ and have radius at least $r_n$.
    
    Then the sequence 
    $$p_n \coloneqq \inf_{\mathcal M \in \mathcal I_{2n}} \mathbb P\left(\mathcal M \text{ grows rapidly with respect to } n\right)$$
    is positive for sufficiently large $n$ and converges to 1 as $n \to \infty$.
\end{lemma}

\begin{corollary}\label{cor:BlowUp}
    Let the assumptions of Lemma \ref{lemm:BlowUpCrit} hold. 
    \begin{enumerate}
        \item[(i)] Then $\varphi$ does not admit a weak set attractor.
        \item[(ii)] If furthermore $\sum_{i=2}^\infty t_i < \infty$, then $\varphi$ is not strongly complete, and for every $\varepsilon>0$, there exists an $n\in \mathbb N$, such that
        \begin{equation}\label{eq:quickBlowUp}
            \mathbb P(\tau^\star(\omega, \mathcal M) <\varepsilon)>1-\varepsilon\quad \forall\, \mathcal M \in \mathcal I_n,
        \end{equation}
        where $\mathcal I_n$ is the family of sets defined in Lemma \ref{lemm:BlowUpCrit}.
    \end{enumerate}
\end{corollary}
\begin{proof}
    (ii): Suppose that the assumptions of Lemma \ref{lemm:BlowUpCrit} hold with $\sum_{i = 2}^\infty t_i <\infty.$ Then for every $n$ there exists some $t$ such that $m(n,t)=\infty.$ Thus, every rapidly growing subset $\mathcal M$ must satisfy $\tau^\star(\omega, \mathcal M)<\infty$ and $\varphi$ is not strongly complete. Furthermore, for every $\varepsilon$ there exists a natural number $N_1$ such that $m(n,\varepsilon/2)=\infty$ for all $n\geq N_1$ and a natural number $N_2$ such that $p_n >1-\varepsilon$ for $n \geq N_2$. Consequently, \eqref{eq:quickBlowUp} holds for $n \coloneqq 2\max(N_1,N_2).$

    (i): Suppose that the assumptions of Lemma \ref{lemm:BlowUpCrit} hold. If $\sum_{i = 2}^\infty t_i <\infty$, by (ii), $\varphi$ is not strongly complete and hence does not admit a set attractor. Hence, we assume $\sum_{i = 2}^\infty t_i = \infty$ and argue by contradiction. Suppose there exists a weak set attractor $A(\omega)$. By Definition \ref{def:Attr}, each deterministic bounded set $\mathcal M$ must satisfy 
    $$\operatorname{dist}(\varphi(t,\omega, \mathcal M), A(\theta_t\omega)) \to 0, \quad \text{in probability.}$$
    At the same time, by Lemma \ref{lemm:BlowUpCrit}, there exists some compact set $\mathcal M$ such that 
    \begin{equation}\label{eq:growingProb}
        p\coloneqq\mathbb P\left(\lim_{t\to \infty} \operatorname{rad}(\varphi(t,\omega, \mathcal M)) = \infty\right)>0.
    \end{equation}
    Choose a large radius $R$, such that $\mathbb P(\operatorname{rad}(A(\omega))>R)\leq p/4$, and a large time $T$, such that $\mathbb P(\operatorname{dist}(\varphi(t,\omega,\mathcal M), A(\theta_t \omega))> 1) \leq p/4$ for all $t\geq T$. Since 
    $$\operatorname{rad}(\varphi(t,\omega, \mathcal M))\leq \operatorname{rad}(A(\theta_t\omega)) + \operatorname{dist}(\varphi(t,\omega, \mathcal M), A(\theta_t\omega)),$$
    we obtain that $\mathbb P(\operatorname{rad}(\varphi(t,\omega, \mathcal M))>R+1)\leq p/2$, which contradicts \eqref{eq:growingProb}.
\end{proof}

Structurally, Lemma \ref{lemm:BlowUpCrit} is similar to Lemma \ref{lemm:ballcrit}, the criterion for strong completeness and for the existence of set attractors introduced in the previous section. In fact, in the proof of Theorem \ref{theo:BlowUp}, the sequences $r_n$ and $t_n$ will be chosen essentially in the same way as for Theorem \ref{theo:attractors2}. The main difference between Lemma \ref{lemm:ballcrit} and Lemma \ref{lemm:BlowUpCrit} is that in Lemma \ref{lemm:ballcrit} it is sufficient to control the evolution of the balls $B_n$. Once such a ball reaches the boundary of the next larger ball, or is contained in the next smaller ball, we forget the exact shape of the image and continue tracking only the evolution of the smallest ball the image is contained in. This scheme is insufficient to establish a lack of strong completeness or the non-existence of set attractors, since the latter require a lower rather than an upper bound on the size of the image of some set. Instead, we need to control the probability that a set that intersects the boundary of some ball $B_n$ reaches the boundary of $B_{n+1}$ before being fully contained in $B_{n-1}$. However, if we tried to control this probability for any set intersecting $\partial B_n$, then this would in particular hold for singleton sets, and a related condition for blow-up would also imply that the 1-point Markov process is transient. Thus, such a condition would never be satisfied for any SDE with a stationary distribution, as the one we consider in this article.

Instead, we only consider sets $\mathcal M$ that are connected and intersect both $\partial B_n$ and $\partial B_{n-1}$. Then, condition \eqref{eq:growthBound} implies that, with probability at least $p^*>\frac{1}{2}$, the set $\mathcal M$ reaches the boundary of $B_{n+1}$ before its being fully contained in $B_{n-1}$ or some fixed time $t_n$ having elapsed. Once either of the three events has occurred, we reassign $\mathcal M$ to be its image under the flow, and set $n$ equal to $n+1$ if the new $\mathcal M$ intersects $B_{n+1}$ and equal to $n-1$ otherwise. We assume that $\mathcal M$ also contains a point with much smaller radius, in particular a point inside $B_{n-1}$ for the new $n$. Since $\mathcal M$ is connected, this leaves us in a situation where the new $\mathcal M$ intersects both $\partial B_{n-1}$ and $\partial B_n$ for the new $n$. Assuming that $n\geq 2$ we can continue iterating. Since the probability of $n$ increasing is at least $p^*>\frac{1}{2}$, the process $n$ has a positive drift of at least $2p^*-1$. Since we can control the maximal amount of time each step in the iteration can take, we obtain a lower bound for the growth rate of the radius of the set $\mathcal M$. This is what we call rapid growth in Lemma \ref{lemm:BlowUpCrit}. If we start with a sufficiently large $n$, the probability that $n$ never drops below 2 is large. It remains to argue that after each step, the new set $\mathcal M$ still contains some point in $ B_{n-1}$. For this, we assume that the connected set $\mathcal M$ at the start of the procedure also contains a point $x_0$ close to the origin, in particular a point in $B_1$. This point serves as an anchor for our set $\mathcal M$. We require that with high probability, after each iterative step, the image of the anchor $x_0$ be contained in $B_{n-1}$. Since $\mathcal M$ is connected, this ensures that after every step $\mathcal M$ intersects $\partial B_{n-1}$. This is the reason why we need the somewhat counterintuitive condition \eqref{eq:anchorBound} that provides an upper bound for the growth rate of a single point.

\begin{figure}
    \centering
    \begin{overpic}[width=0.5\linewidth]{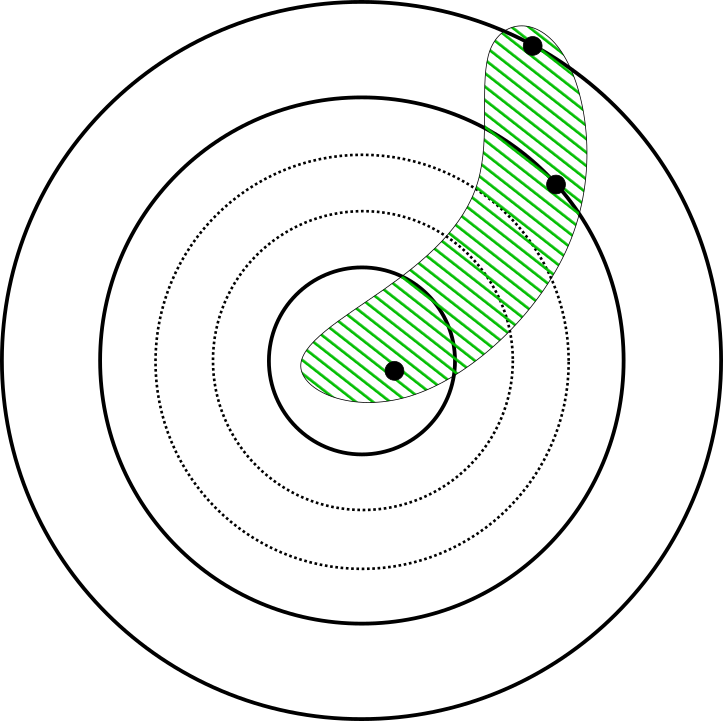}
        \put(52,66){\color{Green}$\mathcal M_i$}
        \put(52,51) {$x_i$}
        \put(71, 89) {$y_i$}
        \put(72,71){$z_i$}
        \put(80,6){$\partial B_{N_i}$}
        \put(60,11){$\partial B_{N_{i}-1}$}
        \put(45,32.5){$\partial B_{K_i}$}
    \end{overpic}
    \caption{The main quantities of the recursion $\mathcal M_i$, $x_i$, $y_i$, and $z_i$ visualized.}
    \label{fig:blowupCirt}
\end{figure}
\begin{proof}[Proof of Lemma \ref{lemm:BlowUpCrit}]
    Suppose that assumptions (i) and (ii) of Lemma \ref{lemm:BlowUpCrit} hold. Fix some natural number $n \in \mathbb N$ and some set $\mathcal M \in \mathcal I_{2n}$, that is, some connected set $\mathcal M$ that intersects $B_1$ and contains at least one point outside $B_{2n}$. The following proof is divided into three steps. First, we set up a recursive scheme, that allows us to treat the continuous time problem as a time-discrete process. There is a possibility that we reach a point at which we are unable to continue this recursion and have to stop. We distinguish between stopping ``successfully'' and stopping ``in failure''. In Step 2 we show that if the recursion does not stop in failure, the set $\mathcal M$ is growing rapidly with respect to $n$. Finally, in Step 3, we find an upper bound for the probability that the recursion stops in failure. Using the result of Step 2, this provides a lower bound for $p_n$.
    
    \textbf{Step 1 -- Introducing the Recursion:}  We start by setting up the recursive scheme, which involves multiple quantities. 
    For this, let $N_0=2n$, $\mathcal M_0 = \mathcal M$, $T_0 = 0$ and choose points $x_0\in B_1 \cap \mathcal M$, $y_0 \in \partial B_{2n}\cap \mathcal M$ and $z_0 \in \partial B_{2n-1}\cap \mathcal M$. Note that such points must exist since $\mathcal M$ is connected. Also, let $(K_i)_{i \in \mathbb N}$ be the deterministic sequence given by $K_i\coloneqq n+\lceil \alpha i\rceil$. The values of $N_i$, $T_i$, $\mathcal M_i$, $x_i$, $y_i$ and $z_i$ will be determined recursively in such a way that $2 \leq N_{i+1} \in \{N_i-1,N_i+1\}$, $T_{i+1}=T_i+\Delta T_{i+1}$ for some $\Delta T_{i+1}>0$ and 
    \begin{equation}\label{eq:recursionCond}
\mathcal M_i = \varphi(T_i,\omega,\mathcal M),\quad
x_i = \varphi(T_i,\omega,x_0) \in C_{K_i}, \quad
y_i \in \partial B_{N_i}\cap \mathcal M_i, \quad
z_i \in \partial B_{N_i-1}\cap \mathcal M_i,
\end{equation}
see Figure \ref{fig:blowupCirt}.
    Furthermore, there is a possibility for the recursion to stop at some iteration $i^*$, after which, for all $i>i^*$, we set $N_i \coloneqq N_{i^*}, T_i \coloneqq T_{i^*},$ and leave the remaining quantities undefined. In this case, we will qualify whether the recursion stopped ``successfully" or ``in failure", which will be relevant for Step 2 below.

    Suppose that after $i$ iterations the recursion did not stop and that $\mathcal M_i, x_i, y_i$ and $z_i$ satisfy \eqref{eq:recursionCond}. Then we set
    $$N_{i+1} \coloneqq \begin{cases}N_i+1, & \text{if } \tau_{N_i+1}(\theta_{T_i}\omega, y_i) \wedge \tau_{N_i+1}(\theta_{T_i}\omega, z_i) \leq \tau_{N_i-1}(\theta_{T_i}\omega, y_i) \wedge t_{N_i},\\
    N_i-1, & \text{otherwise,}
    \end{cases}
   $$
   and $T_{i+1} = T_i +\Delta T_{i+1}$ with
   $$\Delta T_{i+1} \coloneqq \tau_{N_i+1}(\theta_{T_i}\omega, y_i) \wedge \tau_{N_i+1}(\theta_{T_i}\omega, z_i) \wedge \tau_{N_i-1}(\theta_{T_i}\omega, y_i) \wedge t_{N_i} > 0.$$
    If $\tau^\star(\omega, \mathcal M) \leq T_{i+1}$, we stop the recursion successfully. If it did not stop successfully and either $|\varphi(T_{i+1}, \omega, x_0)| >r_{K_{i+1}}$ or ${N_{i+1} \leq K_{i+1}}$, we stop the recursion in failure. Otherwise, we set 
    $${\mathcal M_{i+1} \coloneqq \varphi(T_{i+1}, \omega, \mathcal M) = \varphi(\Delta T_{i+1}, \theta_{T_i}\omega, \mathcal M_i)} \quad\text{and}\quad x_{i+1} =  \varphi(T_{i+1},\omega,x_0) \in C_{K_{i+1}}.$$
    We claim that the set $\mathcal M_{i+1}$ contains at least one point outside $B_{N_{i+1}}$.\\
    \textit{Case 1:} If $N_{i+1} = N_i+1$, then $\Delta T_{i+1}$ is equal to either $\tau_{N_i+1}(\theta_{T_i}\omega, y_i)$ or $\tau_{N_i+1}(\theta_{T_i}\omega, z_i)$ and consequently either 
    $$\varphi(\Delta T_{i+1}, \theta_{T_i} \omega, y_i) \in \partial B_{N_i+1} \cap \mathcal M_{i+1}\quad\text{or}\quad\varphi(\Delta T_{i+1}, \theta_{T_i} \omega, z_i) \in \partial B_{N_i+1} \cap \mathcal M_{i+1}.$$
    \textit{Case 2:} If $N_{i+1} = N_i-1$, since $\Delta T_{i+1} \leq \tau_{N_i-1}(\theta_{T_i}\omega, y_i)$, it follows that $$\varphi(\Delta T_{i+1}, \theta_{T_i} \omega, y_i)\in \mathcal M_{i+1}\setminus B_{N_i-1}.$$
     Since $x_{i+1} \in \mathcal M_{i+1}\cap C_{K_{i+1}} \subseteq \mathcal M_{i+1}\cap C_{N_{i+1}-1}$ and using the fact that $\mathcal M_{i+1}$ is connected, the set $\mathcal M_{i+1}$ must intersect both $\partial B_{N_{i+1}}$ and $\partial B_{N_{i+1}-1}$. This allows us to choose $y_{i+1}$ and $z_{i+1}$ in such a way that $y_{i+1}$ and $z_{i+1}$ are measurable with respect to $\mathcal F_{T_{i+1}}$ and satisfy $y_{i+1} \in \partial B_{N_{i+1}}\cap \mathcal M_{i+1}$ and $z_{i+1} \in \partial B_{N_{i+1}-1}\cap \mathcal M_{i+1}$.
    The measurable selection can be achieved, for example, by choosing the smallest elements using the lexicographic ordering of $\mathbb R^d$.
    It is easy to check that $\mathcal M_i$, $x_i$, $y_i$ and $z_i$ satisfy \eqref{eq:recursionCond}, allowing us to continue the recursion.

    \textbf{Step 2 -- Showing Rapid Growth:} Suppose that the recursion does not stop in failure, that is, it either stops successfully or it never stop. We will show that then $\mathcal M$ grows rapidly with respect to $n$. Let $t \geq 0$. We distinguish two cases.
    
    \textit{Case 1:} Suppose that $t<T_\infty \coloneqq \lim_{i \to \infty} T_i$. Then there exists a unique integer $i^*\in \mathbb N_0$, such that $t \in [T_{i^*}, T_{i^*+1})$. In particular, the recursion did not stop at or before step $i^*$. Since $\Delta T_{i+1} \leq t_{N_i}$, $N_i> K_i$, for all $i\leq i^*$, and the sequence $t_n$ is decreasing, we have
    \begin{equation}\label{ineq:TiBound}
        t <T_{i^*+1} = \sum_{i = 0}^{i^*} \Delta T_{i+1} \leq \sum_{i = 0}^{i^*} t_{N_i} \leq \sum_{i = 0}^{i^*} t_{K_i+1} < \frac{1}{\alpha} \sum_{k = n+1}^{K_{i^*}+1} t_{k}. 
    \end{equation}
    Here, the last inequality comes from an elementary computation which we provide in Proposition \ref{prop:elemIneq} below. 
    By the definition of $m(n,t)$ this implies that $m(n,t) \leq K_{i^*} <\infty$.
    If $\tau^\star(\omega, \mathcal M) \leq t$, \eqref{eq:rapidGrow} holds. Otherwise, by the construction of our recursion, we have $\operatorname{rad}(\varphi(t, \omega, \mathcal M) )\geq r_{N_{i^*}-1}$ and $N_{i^*} > K_{i^*}$, so in particular $\operatorname{rad}(\varphi(t, \omega, \mathcal M)) \geq r_{K_{i^*}} \geq r_{m(n,t)}$, establishing \eqref{eq:rapidGrow}.

    \textit{Case 2:} Suppose that $t \geq T_\infty$. If the recursion stops successfully at some step $i^*$, then $\tau^\star(\omega, \mathcal M) \leq T_{i^*}\leq t$, establishing \eqref{eq:rapidGrow}. Hence, we assume that the recursion never stops. Then 
    \begin{equation}\label{eq:bu}
        \limsup_{t \nearrow T_\infty} \operatorname{rad}(\varphi(t, \omega, \mathcal M)) \geq \lim_{i \to \infty} r_{K_i} = \infty.
    \end{equation}
    If $\tau^\star(\omega, \mathcal M)> t\geq T_\infty$, then $\varphi(\cdot, \omega, \cdot): [0,T_\infty] \times \mathcal M \to \mathbb R^d$ is a continuous map that is bounded since its domain is compact. This contradicts \eqref{eq:bu}, so $\tau^\star(\omega, \mathcal M) \leq T_{\infty}\leq t$, establishing \eqref{eq:rapidGrow}.

    \textbf{Step 3 -- Bounding the Probability:} We complete the proof by finding an upper bound for the probability that the stops in failure. For this, let $E_0$ be the event in which the recursion stops successfully, and $E_1$ and $E_2$ be the events given by
    \begin{align*}
        E_1 &\coloneqq \{\omega \in \Omega\setminus E_0 : \exists\, i \geq 1,~ |\varphi(T_{i},\omega, x_0)| >r_{K_{i}}\},\text{ and}\\
        E_2 &\coloneqq \{\omega \in \Omega\setminus (E_0\cup E_1) : \exists\, i \geq 1,~ N_i \leq K_i\}.
    \end{align*}
    Clearly, the event that the recursion stops in failure is contained in $E_1 \cup E_2$. In the following, we will show that $\mathbb P(E_1)$ and $\mathbb P(E_2)$ converge to $0$ as $n \to \infty$ uniformly in $\mathcal M$.
    
    We start by bounding the probability for $E_1$. Recalling the estimate \eqref{ineq:TiBound} and the definition of $m(n,t)$, we get $m(n,T_{i+1}) \leq K_{i}\leq K_{i+1}$ for all $i\geq 0$.
    Thus, $E_1$ is contained in the larger event $\tilde E_1\supseteq E_1$ given by
    $$\tilde E_1 \coloneqq \left\{\omega \in \Omega: \exists\, t> 0: |\varphi(t,\omega,x_0)| > r_{m(n,t)}\right\}.$$
    By \eqref{eq:anchorBound} the probability of $\tilde E_1$ goes to $0$ as $n\to \infty$, so the same holds for $E_1$.

    To control the probability for $E_2$, note that, by Proposition \ref{prop:tau} (iii),  $(T_i)_{i\geq 1}$ is a sequence of $(\mathcal F_t^+)$-stopping times and by \eqref{eq:growthBound} and the strong Markov property \eqref{eq:StrongMarkov}, we have
    \begin{align*}
        &\quad~\mathbb E \left[\mathds 1_{N_{i+1}= N_i+1}\mid \mathcal F^+_{T_i}\right](\omega)\\
        &= \mathbb P_{\omega'}\big(\tau_{N_i(\omega)+1}(\omega',y_i(\omega)) \wedge \tau_{N_i(\omega)+1}(\omega',z_i(\omega)) \leq \tau_{N_i(\omega)-1}(\omega', y_i(\omega)) \wedge t_{N_i(\omega)}\big)\geq p^*,
    \end{align*}
    for all $i \geq 0$ and all $\omega$ such that the recursion did not stop before step $i$. Here, $\mathbb P_{\omega'}$ denotes the probability with respect to $\omega'$ for $\omega$ fixed. By a comparison argument, the probability of $E_2$ is upper bounded by the probability
    $$\tilde p_n \coloneqq \mathbb P\left(\exists \,i \geq 1: \tilde N_i\leq K_i \mid \tilde N_0 = 2n\right)\geq \mathbb P(E_2),$$
    where $\tilde N$ is a homogeneous discrete time Markov process on $\mathbb Z$ with 
    \begin{align*}
        \mathbb P\left(\tilde N_{i+1} = m+1 \mid \tilde N_i = m\right) &= p^*, \quad\quad \text{ and}\\ 
        \mathbb P\left(\tilde N_{i+1} = m-1 \mid \tilde N_i = m\right) &= 1-p^*.
    \end{align*}
    Clearly, $\tilde N$ has a positive drift with rate $2p^*-1>\alpha$ and $\tilde p_n \to 0$ as $n \to \infty$.
    Thus, we have
    \begin{align*}
        \mathbb P(\mathcal M \text{ grows rapidly with respect to } n) &\geq 1 - \mathbb P(E_1) -\mathbb P(E_2)\\
        &\geq 1 - \mathbb P(\tilde E_1) - \tilde p_n \to 1\quad \text{as } n \to \infty,
    \end{align*}
    where the quantity $1 - \mathbb P(\tilde E_1) - \tilde p_n$ is independent of $\mathcal M$. The claim follows.
    \end{proof}

\begin{proposition}\label{prop:elemIneq}
    As in the proof of Lemma \ref{lemm:BlowUpCrit}, let $n \in \mathbb N$, $\alpha \in (0,1)$, and $t_2, t_3 , \dots$ be a decreasing sequence of positive real numbers. Also, let $K_i\coloneqq n+\lceil \alpha i\rceil$. Then, for all $i^*\in \mathbb N$ we have
    $$\sum_{i = 0}^{i^*} t_{K_i+1} < \frac{1}{\alpha} \sum_{k = n}^{K_{i^*}} t_{k+1}.$$
\end{proposition}
\begin{proof}
We estimate
    \begin{align*}
        \sum_{i = 0}^{i^*} t_{K_i+1} &= t_{n+1} + \sum_{i = 1}^{i^*} t_{K_i+1} =t_{n+1} +\int_{0}^{i^*} t_{n+1+\lceil\alpha \lceil s\rceil \rceil} \, \dd s \leq t_{n+1} +\int_{0}^{i^*} t_{\lceil n+1+\alpha  s \rceil} \, \dd s\\
        &=t_{n+1} +\frac{1}{\alpha} \int_{n+1}^{n+1+\alpha i^*} t_{\lceil s\rceil}\, \dd s \leq t_{n+1} +\frac{1}{\alpha} \int_{n+1}^{n+1+\lceil \alpha i^*\rceil} t_{\lceil s\rceil}\, \dd s\\
        &= t_{n+1}+\frac{1}{\alpha} \sum_{k = n+2}^{K_{i^*}+1} t_k< \frac{1}{\alpha} \sum_{k = n}^{K_{i^*}} t_{k+1},
    \end{align*}
    which shows the claim.
\end{proof}

\subsection{A key estimate}
Our proof of Theorem \ref{theo:BlowUp} is based on applying Lemma \ref{lemm:BlowUpCrit} to the SDE \eqref{SDE2}.
The upper bound on the growth rate of the absolute value of a single point required for \eqref{eq:anchorBound} is relatively simple to obtain for the SDE \eqref{SDE2}, see Proposition \ref{prop:anchorCond} below. Thus, the crucial condition that needs to be verified is \eqref{eq:growthBound}. For this, we express the SDE \eqref{SDE2} in polar coordinates, giving

\begin{equation}\label{newPolarSDE}
\begin{aligned}
        \dd R(t) &= \left(-c\,R(t)+\frac{\sigma^2}{2R(t)}\right) \dd t +\sigma\,[\cos(\phi(t))\,\dd W_1(t)+\sin(\phi(t))\,\dd W_2(t)],\\
    \dd \phi(t) &= -\rho(R(t))\, \dd t + \frac{\sigma}{R(t)}\,[-\sin(\phi(t)) \,\dd W_1(t) + \cos(\phi(t)) \,\dd W_2(t)],
\end{aligned}
\end{equation}
see \eqref{eq:SDE_polar}. To obtain the uniform-in-$n$ lower bound required for \eqref{eq:growthBound}, for each $n\geq 2$, we apply the affine linear transformation $R \mapsto (R-r_n)/(r_n-r_{n-1})$ to the radius. This maps $r_n \mapsto 0$ and $r_{n-1} \mapsto -1$. Furthermore, our choice for the radii $(r_m)_{m\geq 1}$ will imply that $r_{n+1}-r_n \leq r_n - r_{n-1}$, so $r_{n+1} \mapsto 1-\delta$ for some $\delta \in (0,1]$. After an additional time rescaling, this transforms \eqref{newPolarSDE} into an SDE of the form \eqref{neuesde} below. The following lemma is the key estimate that allows us to derive \eqref{eq:growthBound}.

\begin{lemma}\label{newapproach} Consider the SDE on $E=\Gamma \times S^1$
where $\Gamma \subseteq\R$ is an open interval containing $[-2,1]$, given by
\begin{align}\label{neuesde}
\begin{split}
\dd R&= -g(R)\,\dd t + \cos (\phi)\,\dd W_1(t)+\sin (\phi) \,\dd W_2(t),\\
\dd \phi &= -h(R)\,\dd t + \gamma(R) \big( -\sin (\phi)\,\dd W_1(t) + \cos (\phi)\,\dd W_2(t)\big),
\end{split}
\end{align}
where  $W_1$ and $W_2$ are independent standard Brownian motions and $\gamma \in C^3(\Gamma, \mathbb R)$ and $g,h \in C^2(\Gamma,\R)$. Let $\phi_1^0, \phi_2^0 \in S^1$ and denote by $(R_1(t),\phi_1(t))$ and $(R_2(t),\phi_2(t))$ the forward solutions to \eqref{neuesde} with $(R_1(0),\phi_1(0)) = (0,\phi_1^0)$ and $(R_2(0),\phi_2(0)) = (-1,\phi_2^0)$. For some $\delta \in [0,1)$, define three stopping times $\tau^\uparrow_1, \tau^\downarrow_1$ and $\tau^\uparrow_2$ by 
\begin{align}
    \tau^\uparrow_1 &\coloneqq \inf\{t \geq 0: R_1(t) = 1-\delta\},\nonumber\\
    \tau^\downarrow_1 &\coloneqq \inf\{t \geq 0: R_1(t) = -1\},~\text{and}\nonumber\\
    \tau^\uparrow_2 &\coloneqq \inf\{t \geq 0: R_2(t) = 1-\delta \text{ and } R_2(t')> -2, ~\forall\, t'\in [0,t]\} .\label{eq:tau2up}
\end{align}
 There exist  $\varepsilon^* \in (0,1)$ and $T^* \geq 1$, independent of $\delta, \phi_1^0, \phi_2^0, g, h,$ and $\gamma$, such that
\begin{equation}\label{onehalf}
 p^*:=\inf \P(\tau^\uparrow_1 \wedge \tau^\uparrow_2 \leq \tau^\downarrow_1 \wedge T^*)>\frac 12,
 \end{equation}
where the infimum is taken over all $\delta\in [0,1), \phi_1^0,\phi_2^0 \in S^1$ and all functions $g, h$ and $\gamma$ satisfying 
\begin{align}
    g(R) \in [0, \varepsilon^*] \text{ and } |g'(R)|, |g''(R)| \in [0, 3],&\quad \text{for all }R \in [-2,1],\label{eq:gCond}\\
    h'(R) \in \left[\frac{1}{2}, \frac{3}{2}\right] \text{ and } |h(R)|, |h''(R)| \in [0, 3],&\quad \text{for all }R \in [-2,1],\label{eq:hCond}\\
    |\gamma(R)|,|\gamma'(R)| \in [0,\varepsilon^*] \text{ and } |\gamma''(R)|,|\gamma'''(R)| \in [0,3],&\quad \text{for all }R \in [-2,1].\label{eq:gammaCond}
\end{align}

\end{lemma}

\begin{remark}\label{Upsilon}

Note that if $R_2$ reaches the value $-2$ before reaching the value $1-\delta$ then the set in \eqref{eq:tau2up} is empty and by convention $\tau^\uparrow_2 = \infty$. This is such that the probability in \eqref{onehalf} depends exclusively on the values of $g, h$ and $\gamma$ on the interval $[-2,1]$.There exists as $\Upsilon>0$ such that, for any set of functions $g,h,\gamma$ satisfying (\ref{eq:gCond}--\ref{eq:gammaCond}), there exists a set of functions $\tilde g,\tilde h,\tilde \gamma$ defined on $\R$ that agree with $g, h, \gamma$ on $[-2,1]$ and for which $|\tilde g|,|\tilde g'|,|\tilde g''|, |\tilde h|, |\tilde h'|,|\tilde h''|, |\tilde \gamma|,|\tilde \gamma'|,|\tilde \gamma''|,|\tilde \gamma'''|$ are globally bounded by $\Upsilon \geq 3$.
By the previous observation, the probabilities in \eqref{onehalf} for $g,h,\gamma$ and $\tilde g,\tilde h,\tilde \gamma$ are identical.
Thus, we may, without loss of generality, consider only the case $\Gamma = \R$ and only functions $g,h, \gamma$ for which the absolute values and the absolute values of all relevant derivatives are globally bounded by $\Upsilon$. In particular, this ensures that all coefficients in \eqref{neuesde} are globally Lipschitz continuous and \eqref{neuesde} is strongly complete. Additionally, this property will be convenient for the rough-path-theory estimates employed later.

\end{remark}

\begin{proof}[Proof of Lemma \ref{newapproach}] 
Note that $\tau^\uparrow_1$ and $\tau^\uparrow_2$ are decreasing in $\delta$, so the probability in \eqref{onehalf} is increasing in $\delta$. Thus, we may additionally assume $\delta= 0$ without loss of generality.

The idea of the proof is as follows: For $0<\varepsilon_1\ll 1$ small and $g(R) \in [0,\varepsilon_1]$, the process $R_1(t)$ is a standard Brownian motion with a small negative drift, so its first exit of the interval $[-1,1]$ is via 1 with probability only slightly below 1/2. 
Furthermore, if $ T_1 \gg 1$ is large, this probability of this happening before time $T_1$ is still close to 1/2. More precisely, setting $$p_1(\varepsilon_1, T_1) \coloneqq \inf \mathbb P\left(\tau^\uparrow_1 \leq \tau^\downarrow_1 \wedge T_1\right),$$ where the infimum is over all functions $g$ with $g(R) \in [0,\varepsilon_1]$ for all $R \in [-2,1]$, we have
\begin{equation}\label{eq:p1Lim}
    \lim_{\varepsilon_1 \to 0} \,\lim_{T_1\to \infty} p_1(\varepsilon_1,T_1)= \frac{1}{2}.
\end{equation}
Additionally, we consider the probability 
$p_2$ given by $$p_2(\varepsilon_2, T_2) \coloneqq \inf \mathbb P\left(\tau^\uparrow_2 \leq \tau^\downarrow_1\wedge T_2 \text{ and } \tau^\downarrow_1 < \tau^\uparrow_1\right),$$ where the infimum is as in \eqref{newapproach} but with $\varepsilon^*$ replaced by $\varepsilon_2$.

The heart of the proof of Lemma \ref{newapproach} is to show that there exists $\varepsilon^*_2>0$ and $T^*_2>0$ such that $p_2(\varepsilon^*_2, T^*_2) >0$.
Then, by \eqref{eq:p1Lim}, there exists an $\varepsilon_1^*$ and a $T_1^*$ such that $p_1(\varepsilon_1^*, T_1^*) > 1/2 - p_2(\varepsilon^*_2, T^*_2)$. Since both $p_1$ and $p_2$ are decreasing in $\varepsilon_1$, resp.~$\varepsilon_2,$ and increasing in $T_1$, resp.~$T_2$, we can choose $\varepsilon^* \coloneqq \varepsilon_1^* \wedge \varepsilon_2^*$ and $T^* \coloneqq 
T_1^* \vee T_2^*$ and get 
$$p_1(\varepsilon^*, T^*) + p_2(\varepsilon^*, T^*) >1/2.$$
Since the events for $p_1$ and $p_2$ are disjoint and both contained in the event for $p^*$ considered in \eqref{onehalf}, we can conclude
$$p^*\geq p_1(\varepsilon^*, T^*) + p_2(\varepsilon^*, T^*) >1/2.$$

Thus it remains to find $\varepsilon_2^*>0$ and $T_2^*>0$ such that $p_2(\varepsilon^*_2, T^*_2) >0$. For this we use the estimate
\begin{align*}
p_2(\varepsilon_2,T_2)&\geq \inf \P\big(\tau^\uparrow_2 \leq T_2,\, |R_1(s)|\leq \tfrac{1}{3}~ \forall\, s \in [0,T_2], \text{ and } \tau^\downarrow_1 < \tau^\uparrow_1  \, \big)\\
&\geq \frac 13\inf \P\big(\tau^\uparrow_2 \leq T_2 \text{ and }  |R_1(s)|\leq \tfrac{1}{3}~ \forall\, s \in [0,T_2]\, \big) \eqqcolon \frac{1}{3}\, \tilde p_2(\varepsilon_2,T_2),
\end{align*}
where the infimum is as before. The second inequality follows from the fact that the events $\{\omega:\tau^\uparrow_2 \leq T_2\}$ and $\{\omega:|R_1(s)|\leq \tfrac{1}{3},~ \forall\, s \in [0,T_2]\}$ are measurable with respect to $\mathcal F_{0,T_2}$ and $R_1(t)$ stopped at $\tau^\downarrow_1 \wedge \tau^\uparrow_1$ is a supermartingale which satisfies
$$\mathbb P \big( \tau^\downarrow_1 < \tau^\uparrow_1\,\big|\,\mathcal F_{0,T_2}\big)\geq \frac{1-R_1(T_2)}{2} \geq \frac{1}{3}, \quad \text{for all } \omega \text{ such that }|R_1(s)|\leq \tfrac{1}{3}~ \forall\, s \in [0,T_2].$$

The following lemma shows that there exists an $\varepsilon_2^*>0$ such that, for $T_2^*=13$, we have $\tilde p_2(\varepsilon_2^*, T^*_2) > 0$ and thus $p_2(\varepsilon_2^*, T^*_2) > 0$, completing the proof.
\end{proof}

\begin{lemma}\label{Key}
Assume the setting of Lemma \ref{newapproach} with $\delta = 0$. There exists $\varepsilon>0$  such that
\begin{equation}\label{key}
\tilde p_2\big(\varepsilon
\big)\coloneqq \tilde p_2(\varepsilon, 13)=\inf \P\Big(\tau^\uparrow_2 \leq 13,\text{ and } |R_1(t)|\leq \tfrac 13 ~ \forall\, t \in [0,13]\Big)>0,
\end{equation}
where the infimum is taken over all $\phi_1^0,\phi_2^0 \in S^1$, and $g,h,\gamma$ satisfying (\ref{eq:gCond}--\ref{eq:gammaCond}) with $\varepsilon^*$ replaced by $\varepsilon$. 
\end{lemma}
This statement, possibly with different choices for $\varepsilon_2^*$ and $T_2^*$, would have been elementary if $R_1(t)$ and $R_2(t)$ had been independent. However, their diffusions are linked to the covariance based on the difference of angles $\phi_1(t)$ and $\phi_2(t)$. In fact, if $\phi_1(t) = \phi_2(t)$ for some $t \geq 0$, then, at that time, the noise acts identically on $R_1$ and $R_2$ and the systems preserves the order of $R_1$ and $R_2$. To show that the probability \eqref{key} is positive, $R_2$ must overtake $R_1$. Thus, a key step in the proof is to wait until the shear has separated the angles $\phi_1$ and $\phi_2$ by approximately 90°, at which point the diffusions acting on $R_1$ and $R_2$ are almost independent for a short time.
\begin{proof}

To show \eqref{key},  we represent $R_1(t)$ and $R_2(t)$ in the  form:
\begin{align}\label{umf}
\begin{split}
\dd R_1(t)&=- g(R_1(t))\,\dd t + \dd \overline{W}_1(t),\\
\dd R_2(t)&= - g(R_2(t))\,\dd t +  \cos(\vartheta(t))\, \dd  \overline W_1(t)+\sin(\vartheta(t))\,\dd \overline{W}_2(t),\\
\dd \vartheta(t) &= \big(h(R_1(t))-h(R_2(t))\big)\,\dd t\\ 
&\quad + \gamma(R_2(t))\,\big(  -\sin(\vartheta(t))\,\dd \overline W_1(t) +   \cos(\vartheta(t))\,\dd \overline W_2(t)\big)  - \gamma(R_1(t))\,\dd \overline W_2(t),    
\end{split}
\end{align}
where $\overline {W}_1$ and $\overline{W}_2$ are independent standard Brownian motions given by
\begin{align*}
\dd \overline{W}_1(t)&=\cos\big( \phi_1(t)\big)\,\dd W_1(t)   +\sin \big( \phi_1(t)\big)\,\dd W_2(t),\\
\dd \overline{W}_2(t)&=-\sin \big( \phi_1(t)\big)\,\dd W_1(t)  +\cos\big( \phi_1(t)\big)\,\dd W_2(t) 
\end{align*}
and $\vartheta(t):=\phi_2(t)-\phi_1(t)$.\\  

It will be convenient to represent the system \eqref{umf} in Stratonovich form: 
\begin{align}\label{umfo}
\begin{split} 
\dd R_1(t)&=-  g(R_1(t))\,\dd t + \dd \overline W_1(t),\\
\dd R_2(t)&= \Big(- g(R_2(t))  -\frac 12 \gamma(R_2(t)) +\frac 12 \gamma(R_1(t))\cos (\vartheta(t))\Big)\,\dd t        \\
& \quad +\cos(\vartheta(t))\circ \dd \overline W_1(t)+\sin(\vartheta(t))\circ \dd  \overline W_2(t),\\
\dd \vartheta(t) &= \Big(h(R_1(t))-h(R_2(t)) -\frac{1}{2}\,\gamma(R_1(t))\,\gamma(R_2(t))\,\sin(\vartheta(t))\Big)\,\dd t\\
&\quad + \gamma(R_2(t))\,\big(  -\sin(\vartheta(t))\circ\dd \overline W_1(t) +   \cos(\vartheta(t))\circ\dd \overline W_2(t)\big)  - \gamma(R_1(t))\circ\dd \overline W_2(t).   
\end{split}   
\end{align}

Define 
\begin{equation}\label{eq:Tdef}
    \mathcal T:=\inf\Big\{t \geq 0:\vartheta(0)+\big(h(0)-h(-1)\big)t = \frac \pi 2 \mod 2\pi\Big\}.
\end{equation}
The value $\mathcal T$ is chosen so that at time $\mathcal T$ the angular difference $\vartheta_\mathcal T$ mod $2\pi$ is close to  
$\pi/2$ with not too small probability  and so the noise in the radial direction is sufficiently independent for the two trajectories at time $\mathcal T$. 
Recalling the assumptions on $h$, note that $\mathcal T$ depends on the initial value $\vartheta(0)\in [0,2\pi)$ and that $\mathcal T \in [0,4\pi)$, so $\mathcal T+\frac{1}{12} < 13$.

In order to  show the required lower bound for $\tilde p_2$ we consider the (deterministic) solution $(\hat R_1, \hat R_2, \hat \vartheta)$ of 
\eqref{umfo} with $\overline W_1$ replaced by 0 and $\overline W_2$ replaced by the path that takes the value 0 on $[0,\mathcal T]$ and the value $36(t-\mathcal T)$ for $t \geq \mathcal T$. Observe that in this way the process 
$\hat R_2(t)$ experiences a strong upward kick after time $\mathcal T$ causing $\hat R_2(t)$ to pass $\hat R_1(t)$. We must ensure that the solution of the (random) equation \eqref{umfo} remains close to $(\hat R_1, \hat R_2, \hat \vartheta)$ on $[0,\mathcal T]$ with probability bounded away from 0 uniformly with respect to the choice of the coefficients. To obtain such bounds we use results from rough paths theory. 

\textbf{Step 1 -- Deterministic Control via Rough Paths Theory:}
Let $\overline {\bf W}=(\overline {W}, \overline{\mathbb{W}})$ be the geometric rough path obtained from the 2-dimensional Brownian motion $\overline W$ enhanced with $\overline{\mathbb{W}}:= \overline{\mathbb{W}}^{\textrm{Strat}}$ (see \cite[Proposition 3.5]{FrizHairer1}). 
For some $\alpha \in \big(\frac 13,\frac 12\big)$,  which we fix for the remainder of the proof, we consider $\overline {\bf W}$ as a random element of the space ${\mathscr{C}}_{\operatorname g,0}^{0,\alpha}\big([0,13],\R^2\big)$ of $\alpha$-H\"older geometric rough paths starting at 0 and equipped with the $\alpha$-H\"older rough path metric $\varrho = \varrho_{\alpha,13}$ (see \cite[Definition 2.3]{FrizHairer1} or \cite[Definition 2.4]{FrizHairer2}). The letter $\operatorname g$ in ${\mathscr{C}}_{\operatorname g,0}^{0,\alpha}$ stands for {\em geometric} and has nothing to do with the function $g$.
For each  $T \in [0,13]$, the support of the law of $\overline {\bf W}$  contains the geometric rough path ${\bf U}_T=\big(U_T,\mathbb{U}_T\big)$ generated by $U_T(t)=(0,36(t-T)_+)^\top\coloneqq(0,0 \vee 36(t-T))^\top$, $t \in [0,13]$, defined as
$$U_T(s,t) = U_T(t)-U_T(s) \quad \text{and} \quad \mathbb U_T(s,t) = \frac{1}{2}U_T(s,t) \otimes U_T(s,t),$$
see \cite[Exercise 9.6]{FrizHairer1} or \cite[Theorem 13.54]{FrizVictoir}. Following the discussion above, we are particularly interested in $\mathbf U_{\mathcal T}$ for $\mathcal T$ defined in \eqref{eq:Tdef}. Since the set $\mathcal U \coloneqq \{\mathbf U_T : T \in [0, 13]\}$ is compact in ${\mathscr{C}}_{\operatorname g,0}^{0,\alpha}\big([0,13],\R^2\big)$, for each $D>0$, it can be covered with finitely many balls of radius $D/2$ and center in $\mathcal U$. Such balls have a positive mass under the law of $\overline{\bf W}$ and hence, for every $D>0$, we have
\begin{equation}\label{Kuh}
    Q(D):=\inf_{T \in [0,13]}\P \big(\varrho\big( \overline {\bf W}   , {\bf U}_T   \big)\leq D\big)>0.
\end{equation}
 The {\em It\^o-Lyons map} maps a rough path ${\bf V}=(V,\mathbb{V})$ in ${\mathscr{C}}_{\operatorname g,0}^{0,\alpha}\big([0,13],\R^2\big)$ to the solution $ Y^{\bf V}=(R_1^{\bf V}, R_2^{\bf V},\vartheta^{\bf V} )$ of the rough differential equation (RDE) associated to \eqref{umfo}, with $\overline {\bf W}=(\overline {W}, \overline{\mathbb{W}})$ replaced by ${\bf V}=(V, \mathbb{V})$ and with initial condition $(R_1^{\bf V}(0), R_2^{\bf V}(0),\vartheta^{\bf V}(0)) = (0,-1, \phi_2^0-\phi_1^0)$ as above. This map is continuous with respect to $\varrho$ (see \cite[Theorem 8.5]{FrizHairer1} or \cite[Theorem 8.5]{FrizHairer2}) but we need more detailed information on how the modulus of continuity depends on $g$, $h$, $\gamma$, and $\vartheta(0) = \phi_2^0-\phi_1^0$. 
Such a quantitative continuity result is Theorem \ref{RoughDE} which is a special case of \cite[Theorem 12.10]{FrizVictoir}. 
 In order to apply Theorem \ref{RoughDE}  we  extend the range of $\vartheta$ from $S^1$ to $\R$ in a periodic manner and only consider the case $\Gamma = \R$, see Remark \ref{Upsilon}. We have
 \begin{equation}\label{eq:itoLyonsStab}
     \sup_{0 \leq t \leq 13} |Y^{\bf V}(t)-Y^{{\bf U}_{\mathcal T}}(t)| \leq C  \varrho ({\bf V},{\bf U}_{\mathcal T}),
 \end{equation}
 where $C$  depends only on $\alpha$, upper bounds for
  $\varrho ({\bf V},{\bf 0})$, $\varrho ({\bf U}_{\mathcal T},{\bf 0})$, and on an upper bound for the supremum norms of $|g|,|g'|,|g''|,|\gamma|,|\gamma'|,|\gamma''|,|\gamma'''|$, $|h|$, $|h'|$, and $|h''|$. By Remark \ref{Upsilon} we can assume, without loss of generality, that the latter are upper bounded by some universal constant $\Upsilon \geq 3$.
    Since $\mathcal U \coloneqq \{\mathbf U_T : T \in [0, 13]\}$ is compact, we have $\sup_{T \in [0, 13]}\varrho(\mathbf U_T, \mathbf 0) <\infty$ and the constant in \eqref{eq:itoLyonsStab} can be chosen to depend only on $\varrho({\bf V},{\bf 0})$. In particular, we obtain
  \begin{equation}\label{estim}
  \sup_{0 \leq t \leq 13} |Y^{\bf \overline{W}}(t)-Y^{{\bf U}_{\mathcal T}}(t)| \leq C(\varrho({\bf \overline{W}},{\bf 0}))\,\varrho ({\bf \overline{W}},{\bf U}_{\mathcal T}),
  \end{equation}
  where $C$ is increasing in its argument. For $\varrho ({\bf \overline{W}},{\bf U}_\mathcal T) \leq 1$, we have
  $$C(\varrho ({\bf \overline{W}},{\bf 0})) \leq C(\varrho ({\bf U}_\mathcal T,{\bf 0})+1).$$
  In particular, for every $\eta >0$, there exists a $D(\eta)\in (0, 1]$, depending only on $\eta$, such that
  $$\varrho ({\bf \overline{W}},{\bf U}_\mathcal T) \leq D(\eta) \quad \Rightarrow \quad \sup_{0 \leq t \leq 13} |Y^{\bf \overline{W}}(t)-Y^{{\bf U}_\mathcal T}(t)| \leq \eta.$$
  Since $Y^{\bf \overline {W}}$ coincides with the Stratonovich solution $Y = (R_1, R_2, \vartheta)$ of  \eqref{umfo} (see \cite{FrizVictoir}), we have
  \begin{align*}
      \mathbb P\left(\sup_{0 \leq t \leq 13} |Y(t)-Y^{{\bf U}_\mathcal T}(t)| \leq \eta\right) &= \mathbb P\left(\sup_{0 \leq t \leq 13} |Y^{\bf \overline{W}}(t)-Y^{{\bf U}_\mathcal T}(t)| \leq \eta\right)\\
      &\geq \mathbb P \left(\varrho ({\bf \overline{W}},{\bf U}_\mathcal T) \leq D(\eta)\right) \geq Q(D(\eta)),
  \end{align*}
  where the right-hand side is positive and depends only on $\eta$. 
  
      For simplicity, we write $$Y^{{\bf U}_\mathcal T}=(R_1^{{\bf U}_\mathcal T}, R_2^{{\bf U}_\mathcal T},\vartheta^{{\bf U}_\mathcal T}) \eqqcolon (\hat R_1, \hat R_2, \hat \vartheta).$$
Our goal is to show that there exist universal constants $0<\varepsilon<1$ and $0 < \eta < \frac{1}{6}$ such that for $g, h$ and $\gamma$ satisfying (\ref{eq:gCond}--\ref{eq:gammaCond}) with $\varepsilon^*$ replaced by $\varepsilon$, we have
 \begin{align}
      |\hat R_1(t)| \leq \frac{1}{6},~ \forall \, t\in [0,13], \label{eq:R1U}\\
      \exists\, t' \in [0, \tfrac{1}{12}], \text{ s.t. }\hat R_2(\mathcal T+t') = 1+ \eta \quad \text{and}\quad  \hat R_2(t) \geq -\frac{11}{6},\quad\forall\, t\in [0, \mathcal T+t'].\label{eq:R2U}
  \end{align}
Once this is established, we have the implication
$$\sup_{0 \leq t \leq 13} |Y(t)-Y^{{\bf U}_\mathcal T}(t)| \leq \eta\quad \Rightarrow\quad \tau^\uparrow_2 \leq 13,\text{ and } |R_1(t)|\leq \tfrac 13 ~ \forall\, t \in [0,13],$$
and thus
$$\P\Big(\tau^\uparrow_2 \leq 13,\text{ and } |R_1(t)|\leq \tfrac 13 ~ \forall\, t \in [0,13]\Big)\geq Q(D(\eta))>0.$$
Since $Q(D(\eta))$ is independent of $g,h,\gamma,\phi_1^0$ and $\phi_2^0$, the infimum in \eqref{key} is positive as well. 
Thus it remains to prove \eqref{eq:R1U} and \eqref{eq:R2U} for appropriately chosen $\varepsilon$ and $\eta$.
  \paragraph{Step 2 -- Estimates for the Deterministic Trajectory:}

We have
    \begin{align*}
\begin{split} 
\dd \hat R_1(t)&=-  g(\hat R_1(t))\,\dd t\\
\dd \hat R_2(t)&= \Big(- g(\hat R_2(t))  -\frac 12 \gamma(\hat R_2(t)) +\frac 12 \gamma(\hat R_1(t))\cos (\hat \vartheta(t))\Big)\,\dd t   +\mathds 1_{t \geq \mathcal T}\,36\sin(\hat \vartheta(t)) \,\dd t,\\
\dd \hat \vartheta(t) &= \Big(h(\hat R_1(t))-h(\hat R_2(t)) -\frac{1}{2}\,\gamma(\hat R_1(t))\,\gamma(\hat R_2(t))\,\sin(\hat \vartheta(t))\Big)\,\dd t\\
&\quad +\mathds 1_{t \geq \mathcal T}\,36 
\big(\gamma(\hat R_2(t))\,\cos(\hat\vartheta(t)) - \gamma(\hat R_1(t))\big)\, \dd t, 
\end{split}   
\end{align*}
with $(\hat R_1(0), \hat R_2(0), \hat \vartheta(0)) = (0,-1, \vartheta(0))$.
Clearly,
$|\hat R_1(t)| \in [0,13\, \varepsilon]$, for all $t\in [0,13]$.
Thus \eqref{eq:R1U} holds as long as $13\,\varepsilon \leq 1/6.$ For \eqref{eq:R2U}, we need to treat the time intervals $[0, \mathcal T]$ and $[\mathcal T, \mathcal T+ \tfrac{1}{12}]$ separately. For $t \leq \mathcal T$, we have $$|\hat R_2(t)+1| \in [0,2\mathcal T \varepsilon] \subseteq [0, 26\,\varepsilon].$$
Furthermore, we have
\begin{align*}
    \hat \vartheta(\mathcal T) &= \vartheta(0) + \int_0^\mathcal T\Big(h(\hat R_1(t))-h(\hat R_2(t)) -\frac{1}{2}\,\gamma(\hat R_1(t))\,\gamma(\hat R_2(t))\,\sin(\hat \vartheta(t))\Big)\,\dd t\\
    &= \vartheta(0) + (h(0) - h(-1)) \mathcal T + \int_0^\mathcal T \big(h(\hat R_1(t))-h(0)\big)\, \dd t - \int_0^\mathcal T \big(h(\hat R_2(t))-h(-1)\big)\, \dd t\\
    &\quad - \int_0^{\mathcal T}\frac{1}{2} \,\gamma(\hat R_1(t))\,\gamma(\hat R_2(t))\,\sin(\hat \vartheta(t))\,\dd t.
\end{align*}
Recall from \eqref{eq:Tdef}, that $\mathcal T$ was defined in such a way that 
$$\vartheta(0) + (h(0) - h(-1)) \mathcal T = \frac{\pi}{2},\quad \text{mod } 2\pi.$$
Since the lift of the initial value $\vartheta(0)$ from $S^1$ to $\mathbb R$ was arbitrary, we may assume without loss of generality that this equality holds in $\mathbb R$ as well.
Using the previous estimates on $|\hat R_1(t)|$ and $|\hat R_2(t)+1|$ and the fact that the derivative of $h$ is at most $3/2$ on $[-2,1]$, we get
\begin{equation}\label{eq:thetaOfT}
    \left| \hat \vartheta(\mathcal T) - \frac{\pi}{2}\right| \leq 13\, \varepsilon\, \frac{3}{2}\, \mathcal T + 26\, \varepsilon\, \frac{3}{2}\, \mathcal T + \frac{1}{2}\, \varepsilon^2\, \mathcal T\leq 767\,\varepsilon.
\end{equation}

Note that to show \eqref{eq:R2U}, we are required to control the trajectory of $\hat R_2$ after it leaves the interval $[-2,1]$ on which $g, h$ and $\gamma$ are controlled by (\ref{eq:gCond}--\ref{eq:gammaCond}). Here, the fact that $|g'|, |h'|$ and $|\gamma'|$ are globally bounded by the universal constant $\Upsilon\geq 3$ is useful. By choosing $\eta \coloneqq \varepsilon/ \Upsilon$, for $\varepsilon$ still to be determined, the intermediate value theorem gives 
$$|g(r)| \leq 2\varepsilon, \quad|h(r)| \leq 3+\varepsilon<4, \quad\text{and}\quad |\gamma(r)|\leq 2\varepsilon, \quad \forall\, r \in [-2,1+\eta].$$
Let $t^* \in (0,\tfrac{1}{12}]$ be the given by
$$t^*\coloneqq \frac{1}{12} \wedge \inf \{t\geq 0: \hat R_2(\mathcal T+ t) \notin [-2, 1+\eta]\}.$$
We  can control the trajectory of $\hat \vartheta$ on the interval $[\mathcal T, \mathcal T + t^*]$. For $t \in [0, t^*]$, we have
\begin{align*}
    \left|\hat \vartheta(\mathcal T+t) -\hat \vartheta(\mathcal T)\right| &\leq \int_\mathcal T^{\mathcal T+t}\Big|h(\hat R_1(s))-h(\hat R_2(s)) -\frac{1}{2}\,\gamma(\hat R_1(s))\,\gamma(\hat R_2(s))\,\sin(\hat \vartheta(s))\Big|\,\dd s\\
    &\quad + 36\int_\mathcal T^{\mathcal T+t} 
\big|\gamma(\hat R_2(s))\,\cos(\hat\vartheta(s)) - \gamma(\hat R_1(s))\big|\, \dd s\\
&\leq \frac{1}{12} \,(4+4+2\varepsilon^2) + \frac{36}{12}\cdot 4\varepsilon \leq \frac{2}{3} + 13 \varepsilon.
\end{align*}
Together with \eqref{eq:thetaOfT}, and requiring $\varepsilon$ to be very small ($\varepsilon\leq 10^{-4}$ suffices), we get
$$ \left|\hat \vartheta(\mathcal T+t) -\frac{\pi}{2}\right| \leq \left|\hat \vartheta(\mathcal T) -\frac{\pi}{2}\right| +\left|\hat \vartheta(\mathcal T+t) -\hat \vartheta(\mathcal T)\right| \leq \frac{2}{3} + 780\,\varepsilon \leq \frac{\pi}{4}, \quad \forall \, t \in [0, t^*].$$
In particular, we have $\sin(\hat \vartheta(t)) \geq \sqrt 2/2$ for $t \in [\mathcal T, \mathcal T + t^*]$.
This bound allows us to finally control the trajectory of $\hat R_2$ on the interval $[\mathcal T, \mathcal T + t^*]$. For $t \in [0, t^*]$, we have
\begin{align*}
    \hat R_2(\mathcal T+ t) &\geq \hat R_2(\mathcal T) + \int_\mathcal T^{\mathcal T+t}\Big(- g(\hat R_2(s))  -\frac 12 \gamma(\hat R_2(s)) +\frac 12 \gamma(\hat R_1(s))\cos (\hat \vartheta(s))\Big)\,\dd s \\
    &\quad +36\int_\mathcal T^{\mathcal T+t}\sin(\hat \vartheta(s)) \,\dd s\\
&\geq -1-26\,\varepsilon - \frac{1}{12}\cdot 4\,\varepsilon +36 \, t \, \frac{\sqrt{2}}{2} \geq -1 - 27 \,\varepsilon + 25\, t.
\end{align*}
For $\varepsilon$ sufficiently small, we can deduce that $\hat R_2(t)\geq -11/6$ during $[\mathcal T, \mathcal T+t^*].$ Furthermore, if $\varepsilon$ is small enough such that $-1-27\varepsilon +25/12 > 1+ \varepsilon/3 \geq 1+\eta$, by the definition of $t^*$, we know that $t^* < 1/12$, and hence $\hat R_2(\mathcal T + t^*) = 1+\eta$. This establishes \eqref{eq:R2U} and thereby the proof of Lemma \ref{Key} and Lemma \ref{newapproach}.
\end{proof}
\subsection{Proof of Theorem \ref{theo:BlowUp}}
In this subsection, we complete the proof of Theorem \ref{theo:BlowUp} by verifying the conditions of Lemma \ref{lemm:BlowUpCrit} for the RDS $\varphi$ induced by our main SDE of consideration, namely
\begin{equation}\label{SDE3}
    \dd X(t)=-X(t)\,\dd t+\rho(|X(t)|) \begin{pmatrix}0&1\\-1&0 \end{pmatrix}\,X(t)\,\dd t+\dd W(t).
\end{equation}

For \eqref{eq:growthBound}, we will use the key estimate derived in Lemma \ref{newapproach}, and for \eqref{eq:anchorBound}, we will use the following proposition.

\begin{proposition}\label{prop:anchorCond}
    Let $\varphi$ be the RDS given by \eqref{SDE3}. Assume that, in the setting of Lemma \ref{lemm:BlowUpCrit}, the sequences $(r_n)$, $(t_n)$, and the constant $\alpha$ are chosen in such a way that the function $m$ defined in \eqref{eq:Growthrate} satisfies
    $$r_{m(n,t)} \geq c_n \,e^{ct}, \quad \forall\, n \in \mathbb N, t>0$$
    for some $c>0$ and some sequence of real numbers $c_n$ with $c_n\to \infty$. 
    
    Then \eqref{eq:anchorBound} holds, that is,
    \begin{equation}\label{eq:anchorBound2}
            \lim_{n \to \infty}\inf_{x\in B_1}\mathbb P(|\varphi(t, \omega,x)| \leq r_{m(n,t)}, ~\forall\,t\in \mathbb R_{>0}) = 1.
        \end{equation}
\end{proposition}
\begin{proof}
    Note that \eqref{eq:anchorBound2} depends only on the law of $|X_t|$ for individual solutions to \eqref{SDE3}. Since this law is independent of $\rho$, we may assume $\rho=0$ without loss of generality. In this case, the RDS $\varphi$ can be expressed as 
    $$ \varphi(t, \omega, x) = e^{-t}\left(x+\widetilde W(s(t),\omega)\right),$$
    where $s(t)\coloneqq (e^{2t}-1)/2$ and $\widetilde W$ is a 2-dimensional Brownian motion. Now \eqref{eq:anchorBound2} is a consequence of the law of the iterated logarithm. In fact, the proposition would still hold if the exponential growth assumption were to be replaced by an assumption of the form $r_{m(n,t)} \gtrsim \sqrt{\log(2+t)}$, see for example \cite{Chen15}.
\end{proof}

We are now in a position to prove Theorem \ref{theo:BlowUp}.
\begin{proof}[Proof of Theorem \ref{theo:BlowUp} (a)]
    For $\varepsilon^*>0$ as in Lemma \ref{newapproach}, let $K_2 \coloneqq (\varepsilon^*/2)^{-3}.$ Assume that $\rho$ satisfies the assumptions of Theorem \ref{theo:BlowUp}, that is, there exists an $r^*>0$ for which
    \begin{equation}\label{eq:noSetAttr}
        \limsup_{r \to \infty} \frac{|\rho''(r)|}{|\rho'(r)|} <\infty, \quad\quad |\rho'(r)| \geq K_2\, r^3, \quad \forall\, r\geq r^*
    \end{equation}
    and $|\rho'(r)|$ increases monotonically on $[r^*,\infty)$. In particular, this implies that $\rho'(r)$ has a constant sign on $[r^*,\infty)$.
    Without loss of generality, we assume that $\rho'(r)>0$ on $[r^*,\infty)$. If not, we consider the \eqref{SDE3} with $\rho$ replaced by $-\rho$ which corresponds to replacing $X_1$ by $-X_1$.
    
    For some large $r_1\geq r^*$, whose exact value will be determined later, we recursively define $r_{n+1}$ in such a way that
    \begin{equation}\label{eq:rRecursion}
        r_{n+1}-r_n = \rho'(r_{n+1})^{-1/3}
    \end{equation}
    holds. This is always possible and the choice is unique since $\rho'(r)^{-1/3}$ is monotonically decreasing and continuous. Note that this is the same definition we used in the proof of Theorem \ref{theo:attractors2} (a). First, we argue that $r_n \to \infty$. If not, taking the limit $n \to \infty$ in Equation \eqref{eq:rRecursion} gives $\rho'(r_{\lim})^{-1/3} = 0$ for the finite limit $r_{\lim} \coloneqq \lim_{n \to \infty} r_n$, which leads to a contradiction. Next, for $n\geq 2$, we set $q_n \coloneqq  \rho'(r_{n})^{-1/3} = r_n - r_{n-1}$ and $t_n = T^* \,q_n^2$, where $T^*$ is as in Lemma \ref{newapproach}. Note that $t_n$ is decreasing as required. We proceed by checking the two main assumptions \eqref{eq:growthBound} and \eqref{eq:anchorBound} of Lemma \ref{lemm:BlowUpCrit}.

    \textbf{Step 1 -- Coordinate Transformation:} For the former, we will show that the condition \eqref{eq:growthBound} holds for the $p^*>\frac{1}{2}$ given in Lemma \ref{newapproach}, that is, for all $n \geq 2$ and all $x, y \in \mathbb R^2$ with $|x|= r_n$ and $|y| = r_{n-1}$, we have
    \begin{equation}\label{ineq:growthGoal}
        \mathbb P(\tau_{n+1}(x)\wedge \tau_{n+1}(y) \leq \tau_{n-1}(x) \wedge t_n) \geq p^*,
    \end{equation}
    where $\tau_n$ and $\tau_{n-1}$ are the stopping times defined in Lemma \ref{lemm:BlowUpCrit}.
     Recall from \eqref{eq:SDE_polar} the polar coordinate representation of \eqref{SDE3}, which evolves according to
     \begin{equation}\label{sde:polar2}
    \begin{aligned}
\dd R(t)&=-\left(R(t)-\frac{1}{2R(t)}\right)\,\dd t +\cos(\phi(t))\,\dd W_1(t) + \sin(\phi(t))\,\dd W_2(t),\\
\dd \phi(t)&=-\rho(R(t))\,\dd t +\frac 1{R(t)}\big(-\sin(\phi(t))\,\dd W_1(t) + \cos(\phi(t))\,\dd W_2(t)\big).
\end{aligned}
\end{equation}
Fix some $n\geq 2$. In order to show \eqref{ineq:growthGoal}, we specifically care about the dynamics of \eqref{sde:polar2} for values of $R$ between $r_{n-1}$ and $r_{n+1}$. To study these, we introduce new coordinates that zoom in on this region. At the same time, we also slow down the time by a factor of $q_n^2$ and introduce a moving reference frame for $\phi$. The new coordinates are given by 
$$\tilde R(t) = \frac{1}{q_n} (R(q_n^2\, t)-r_n) \quad \quad \tilde \phi(t) = \phi(q_n^2 t) + q_n^2\rho(r_n)\, t.$$
In these coordinates, \eqref{sde:polar2} is given by
     \begin{equation}\label{sde:polar3}
    \begin{aligned}
\dd \tilde R(t)&=-g_n(\tilde R(t))\,\dd t +\cos(\tilde \phi(t))\,\dd \widetilde W_1(t) + \sin(\tilde \phi(t))\,\dd \widetilde W_2(t),\\
\dd \tilde \phi(t)&=-h_n(\tilde R(t))\,\dd t +\gamma_n(\tilde R(t))\big(-\sin(\tilde \phi(t))\,\dd \widetilde W_1(t) + \cos(\tilde \phi(t))\,\dd \widetilde W_2(t)\big),
\end{aligned}
\end{equation}
where the functions $g_n, h_n,$ and $\gamma_n$ are given by
\begin{align*}
    g_n(\tilde R) &\coloneqq q_n\left(r_n+q_n \tilde R-\frac{1}{2(r_n+q_n\tilde R)}\right),\\
    h_n(\tilde R) &\coloneqq q_n^2 \left(\rho(r_n+q_n\tilde R) - \rho(r_n)\right),\text{ and}\\
    \gamma_n(\tilde R) &\coloneqq \frac{q_n}{r_n+q_n\tilde R},
\end{align*}
and $\widetilde W_1, \widetilde W_2$ are independent Brownian motions that depend on $n$ but not on the trajectory $(\tilde R(t), \tilde \phi(t))$.
Note that \eqref{sde:polar3} is of the form \eqref{neuesde} studied in Lemma \ref{newapproach}. Furthermore, $R = r_{n-1}$ corresponds to $\tilde R = -1$, $R=r_n$ corresponds to $\tilde R = 0$, and $R = r_{n+1}$ corresponds to $\tilde R = 1-\delta_n$, for $\delta_n \coloneqq 1-q_{n+1}/q_n \in [0,1)$. Moreover, the time $t_n = q_n^2 \, T^*$ in \eqref{sde:polar2} corresponds to the time $T^*$ in \eqref{sde:polar3}. Thus, in order to establish \eqref{ineq:growthGoal}, it is sufficient to show that $g_n, h_n$ and $\gamma_n$ satisfy the conditions of Lemma \ref{newapproach}, given that $r_1$ was chosen sufficiently large.

\textbf{Step 2 -- Verifying (\ref{eq:gCond}--\ref{eq:gammaCond}):} First note that, for each $n\geq 2$, we have
\begin{equation}\label{ineq:rnqn}
    r_n \, q_n = r_n \, \rho'(r_n)^{-1/3} \leq \varepsilon^*/2<1/2.
\end{equation}
Fix $n\geq 2$ and $\tilde R \in [-2,1]$. Write $S \coloneqq S(n, \tilde R) \coloneqq r_n+q_n \tilde R \in [r_n-2q_n, r_n+ q_n]$. We have
$$g_n(\tilde R) = q_n\left(S-\frac{1}{2S}\right),\quad h_n(\tilde R) = q_n^2 (\rho(S)-\rho(r_n)), \text{ and}\quad \gamma_n (\tilde R) = \frac{q_n}{S}.$$
 Prescribe $r_1\geq 2$. Consequently $r_n\geq 2$ and $q_n \leq \varepsilon^*/4 < 1/4$. In particular, we have $S\geq r_n-2q_n > 3/2$. Thus $g_n(\tilde R) > 0$. Furthermore, 
$$g_n(\tilde R) \leq q_n \,S \leq q_n r_n +q_n^2 \leq \frac{\varepsilon^*}{2} + \frac{\varepsilon^{*2}}{4} \leq \varepsilon^*.$$
We verified $g_n(\tilde R) \in [0,\varepsilon^*]$. For the first two derivatives of $g_n$, we have
\begin{align*}
    |g'_n(\tilde R)| &= q_n^2\left(1+\frac{1}{2S^2}\right) < \frac{1}{16}\left(1+\frac{1}{2(3/2)^2}\right)< 3,\quad\text{and}\\
    |g''_n(\tilde R)| &= \frac{q_n^3}{S^3} < \frac{(1/4)^3}{(3/2)^3}<3.
\end{align*}
We proceed to check the conditions for $h_n$ and its derivatives. By \eqref{eq:noSetAttr}, there exists an $L>0$ and an $R^*\geq r^*$ such that $|\rho''(r)|/\rho'(r) \leq L$ for all $r\geq R^*$. In addition to the previous assumptions, prescribe $r_1 \geq (R^*+1/2) \vee 4L$ such that, in particular, $S\geq R^*$ and $q_n\, L \leq q_n\, r_n/4 < 1/8$. We have
$$h_n'(\tilde R) = q_n^3\rho'(S) = \frac{\rho'(S)}{\rho'(r_n)},$$
and thus
\begin{align*}
    |\log(h'_n(\tilde R))| &= |\log(\rho'(S))-\log(\rho'(r_n))| = \left|\int_{r_n}^S\frac{\rho''(r)}{\rho'(r)}\, \dd r\right|\\
    &\leq |S-r_n|\, L \leq 2q_n\,L  <\frac{1}{4}.
\end{align*}
Since $1/2<e^{-1/4}$ and $e^{1/4} < 3/2$, can conclude $h_n'(\tilde R) \in [1/2, 3/2]$. By the same argument, we have $|h_n'(R)| \leq 3/2$ for all $R\in [-2,1]$. Together with $h_n(0) = 0$, we can conclude $|h_n(\tilde R)| \leq 3$ by the mean value theorem. For the second derivative we have
$$|h_n''(\tilde R)| = q_n^4\, |\rho''(S)| =  q_n \frac{|\rho''(S)|}{\rho'(r_n)} \leq q_n \, L \, \frac{\rho'(S)}{\rho'(r_n)} <\frac{1}{4}\,h_n'(\tilde R) \leq \frac{3}{8}<3. $$
Finally, we check the conditions for $\gamma_n$ and its derivatives. We have
$$|\gamma_n(\tilde R)| = \frac{q_n}{S}, \quad |\gamma'_n(\tilde R)| = \frac{q_n^2}{S^2}, \quad |\gamma''_n(\tilde R)| = 2\frac{q_n^3}{S^3}, \quad \text{and}\quad |\gamma'''_n(\tilde R)| = 6\frac{q_n^4}{S^4}.$$
The required bounds follow directly from $q_n/S < (\varepsilon^*/4)/(3/2) = \varepsilon^*/6$ and $\varepsilon^*<1$. Thus, Lemma \ref{newapproach} establishes \eqref{eq:growthBound}.

\textbf{Step 3 -- Controlling the One-Point Growth:}
In order to verify \eqref{eq:anchorBound} using Proposition \ref{prop:anchorCond}, we show that $t \mapsto m(n,t)$ is growing exponentially. Using $r_{k+1}\, q_{k+1} \leq \varepsilon^*/2$, see \eqref{ineq:rnqn}, we get
$$\log(r_{k+1})-\log(r_k) = \int_{r_k}^{r_{k+1}} \frac{1}{s}\, \dd s \geq \frac{r_{k+1}-r_k}{r_{k+1}} = \frac{q_{k+1}}{r_{k+1}} \geq \frac{2}{\varepsilon^*} \, q_{k+1}^2 = \frac{2}{\varepsilon^*T^*} \, t_{k+1},$$
for all $k \in \mathbb N$. Summing this from $k=n$ to $k=m(n,t)-1$ for some $t>0,$ we obtain
\begin{align*}
    \log(r_{m(n,t)}) -\log(r_n)\geq \frac{2}{\varepsilon^*T^*} \sum_{k=n}^{m(n,t)-1}t_{k+1}.
\end{align*}
By the definition of $m(n,t)$, we have $\sum_{k=n}^{m(n,t)}t_{k+1}\geq \alpha t$ for all $t$ such that $m(n,t)<\infty.$
Thus,
\begin{align}
    \log(r_{m(n,t)}) &\geq \log(r_n) + \frac{2}{\varepsilon^*T^*} \sum_{k=n}^{m(n,t)-1}t_{k+1}\nonumber\\
    &\geq \log(r_n) + \frac{2}{\varepsilon^*T^*} (\alpha t- t_{m(n,t)})\nonumber\\
    &\geq \log(r_n) + \frac{2}{\varepsilon^*T^*} (\alpha t- t_2)\label{eq:expGrowth}
\end{align}
Thus the assumption \eqref{eq:anchorBound} follows from Proposition \ref{prop:anchorCond}.

\textbf{Step 4 -- Completing the Proof:} By Corollary \ref{cor:BlowUp}, $\varphi$ does not admit a set attractor, showing Theorem \ref{theo:BlowUp} (a.i). Also, by \eqref{eq:expGrowth}, every rapidly growing set has an exponentially growing radius. Thus Lemma \ref{lemm:BlowUpCrit} establishes Theorem \ref{theo:BlowUp} (a.iii). For Theorem \ref{theo:BlowUp} (a.ii), we need to show that every compact connected set $\mathcal M$ with at least two points has a positive probability to have an exponentially growing radius. This is not covered by Lemma \ref{lemm:BlowUpCrit}, which only establishes this property for sets $\mathcal M' \in \mathcal I_{2n}$ for some sufficiently large $n$. Using a control theoretic argument, we show in Lemma \ref{lemm:separation} in the appendix that for any two distinct points $x_1 \neq x_2 \in \mathbb R^2$ and any $n\in \mathbb N$, there is a time $T >0$ such that $|\varphi(T,\omega, x_1)|\leq1$ and $|\varphi(T,\omega, x_2)|\geq r_{2n}$ with positive probability. Thus for any compact connected set $\mathcal M$ with at least two points, there is a time $T>0$ such that there is a positive probability such that either $\tau^\star(\omega,\mathcal M) \leq T$ or $\varphi(T,\omega, \mathcal M) \in \mathcal I_{2n}$. In the former case we are done, and in the latter case we can apply Lemma \ref{lemm:BlowUpCrit} and the strong Markov property to conclude Theorem \ref{theo:BlowUp} (a.ii).

\end{proof}

\begin{proof}[Proof of Theorem \ref{theo:BlowUp} (b)]
    In addition to the assumptions of Theorem \ref{theo:BlowUp} (a), assume \eqref{eq:NoCompl}, that is,
    \begin{equation*}
\int_{r^*}^\infty \rho'(r)^{-1/3}\,\dd r<\infty,
\end{equation*}
where we assume $ \rho'(r)>0$ for $r\geq r^*$, as before, without loss of generality.  Using $t_i = T^* q_i^2$ and $q_i = r_i -r_{i-1} = \rho'(r_i)^{-1/3}$, as well as the fact that $\rho'(r)^{-1/3}$ for all $r>r_1\geq r^*$ is decreasing, we can compute
    \begin{align*}
        \sum_{i = 2}^\infty t_i = T^* \,\sum_{i = 2}^\infty q_i^2 = T^* \,\sum_{i = 2}^\infty (r_i-r_{i-1}) \rho'(r_i)^{-1/3} \leq T^* \,\int_{r_1}^\infty \rho'(r)^{-1/3}\, \dd r < \infty.
    \end{align*}
    By Corollary \ref{cor:BlowUp} (ii), $\varphi$ is not strongly complete, and for every $\varepsilon>0$, there exists an $n \in \mathbb N$, such that
    $$
            \mathbb P(\tau^\star(\omega, \mathcal M) <\varepsilon)>1-\varepsilon\quad \forall\, \mathcal M \in \mathcal I_n,
$$
establishing (b.i) and (b.iii). The remaining statement (b.ii) can be deduced using Lemma \ref{lemm:separation}, analogously to (a.ii).
\end{proof}

\section*{Acknowledgements}
The authors thank Xue-Mei Li, Franco Flandoli and Peter Friz for their insightful feedback. 

The work of DC is funded
by the Deutsche Forschungsgemeinschaft (DFG, German Research Foundation) - SPP 2410
Hyperbolic Balance Laws in Fluid Mechanics: Complexity, Scales, Randomness (CoScaRa).
ME thanks the DFG SPP 2298 Foundation of Deep Learning (FoDL), the Einstein Foundation, and the Dutch Research Council NWO (VI.Vidi.233.133) for support.

\section*{Statement on the use of artificial intelligence}
Artificial intelligence, specifically GPT-5.6 Sol, was used to assist in proof reading, to suggest minor technical computations, and to find relevant references. All results and the central ideas of their proofs were developed without the assistance of artificial intelligence. The presented manuscript is entirely human generated.
\appendix

\section{Rough differential equations}\label{app:rough}

The following theorem is an immediate consequence of \cite[Theorem 12.10]{FrizVictoir}. For unexplained notation, we refer the reader  to \cite{FrizVictoir} or  \cite{FrizHairer1} or \cite{FrizHairer2}.

\begin{theorem}\label{RoughDE} Let $\alpha \in \Big(\frac 13,\frac 12\Big)$ and $T>0$. 
Let ${\bf{U}}=(U,{\mathbb{U}})$ and ${\bf{V}}=(V,\mathbb{V})$ belong to the class $\mathscr{C}^{0,\alpha}_{g,0}([0,T],\R^d)$ of geometric rough paths starting at 0 and equipped with the  $\alpha$-H\"older rough path metric $\varrho_{\alpha, T}$. Let ${\bf Y^U}=\big(Y^{\bf U},\mathbb{Y^{\bf U}})$ be the solution of the rough differential equation
\begin{equation}\label{roughSDE}
\dd Y(t)=A(Y(t))\,\dd t+ B(Y(t))\,\dd {\bf U}(t),\quad Y(0)=y_0\in \R^e,
\end{equation}
(and similarly for ${\bf V}$ instead of ${\bf U}$), where $A\in C^2(\R^e,\R^e)$ and $B \in C^3(\R^e, \R^{e\times d})$ are such that $B$ together with its derivatives up to order 3 and $A$ together with its derivatives up to order 2 are bounded by a number $M$. 

Then
$$
\sup_{0 \leq t \leq T} |Y^{\bf V}(t)-Y^{\bf U}(t)|\leq C \varrho_{\alpha,T}(\bf{V},\bf{U}),
$$
where $C$ depends on $\alpha$,  $M$, $d$, $e$, and upper bounds of $T$, $\varrho_{\alpha,T}({\bf U},{\bf 0})$, and  $\varrho_{\alpha,T}({\bf V},{\bf 0})$ only.
\end{theorem}

The assumptions in the theorem are   stronger than necessary but good enough for our purpose. In particular, we only vary the driving rough path but neither the coefficients of the SDE nor the initial condition.

\section{Separation Lemma}

In this appendix, we show the following lemma, which is needed for the proof of Theorem \ref{theo:BlowUp}.

\begin{lemma}\label{lemm:separation}
Suppose that $\rho: [0,\infty) \to \mathbb R$ satisfies
\begin{equation}\label{eq:separationCondition}
    \liminf_{r \to \infty} |\rho'(r) r| = \infty.
\end{equation}
Then, for any $x_1, x_2 \in \mathbb R^2$ with $x_1\neq x_2$ and every $M >0$ there exists a time $T\geq 0$ such that
$$\mathbb P\left(|\varphi(T,\omega, x_1)|\leq 1\text{ and } \left|\varphi(T,\omega, x_2)\right|\geq M \right)>0.$$
\end{lemma}

\begin{figure}
    \centering
    \begin{overpic}[width =0.5\linewidth] {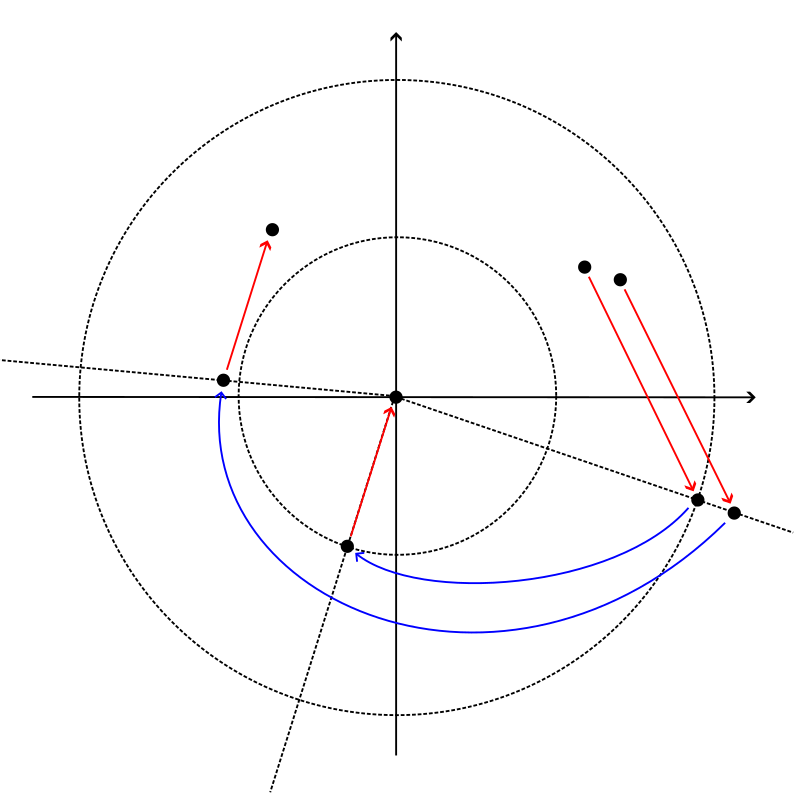}
    \put(68,68){$x_1$}
    \put(79,66){$x_2$}
    \put(82,35){$y_1$}
    \put(91,30){$y_2$}
    \put(40,34){$z_1$}
    \put(23,55){$z_2$}
    \end{overpic}
    \caption{Visualization of the proof of Lemma \ref{lemm:separation}. The blue arrows show the trajectory of the points $y_1, y_2$ described in Step 1, while the red arrows show the approximate trajectory of the points $x_1,x_2, z_1, z_2$ described in Step 2. The two dashed circles have radius $r^*/2$ and $r^*$.}
    \label{fig:separation}
\end{figure}

\begin{remark}
    An assumption similar to \eqref{eq:separationCondition} is necessary in the sense that a simple calculation shows that, for $\rho$ with $|\rho'(r) r| \leq 2$ for all $r\geq 0$, the distance between any two points is decreasing monotonically almost surely.
\end{remark}
\begin{proof}
The condition \eqref{eq:separationCondition} implies that the sign of $\rho'(r)$ is constant for sufficiently large $r$. Without loss of generality, we assume that it is eventually positive implying that $\liminf_{r \to \infty} \rho'(r) r = \infty.$ The following proof is most easily understood by referring to Figure \ref{fig:separation}.

\textbf{Step 1:} Let $\omega_0 \in \Omega$ be the constant path, such that $(\varphi(t, \omega_0, \cdot))_{t\geq 0}$ is the deterministic flow induced by the ODE 
\begin{equation}\label{eq:sepODE}
    \dot X(t)=-\,X(t)+\rho(|X(t)|) \begin{pmatrix}0&1\\-1&0 \end{pmatrix}X(t).
\end{equation}
We first show that for any $\tilde M>0$, $\delta>0$ and any angle $\phi \in [0,2\pi)$, there exists a radius $r^*>0$ and a time $\tilde T\geq 0$ such that
\begin{equation}\label{eq:sepLemmStep1Goal}
    |\varphi(\tilde T, \omega_0, y_1) - \varphi(\tilde T, \omega_0, y_2)| \geq \tilde M, \quad \text{where } y_1 \coloneqq r^* \begin{pmatrix}\cos \phi\\ \sin \phi\end{pmatrix},~y_2 \coloneqq (r^*+\delta) \begin{pmatrix}\cos \phi\\ \sin \phi\end{pmatrix}.
\end{equation}
Suppose that $\tilde M>\delta$, since otherwise $\tilde T = 0$ and any $r^*>0$ suffices. Choose $r^* \geq \tilde M>\delta$ such that
\begin{equation}\label{eq:sepLemmaShearCond}
    \rho'(r) r \geq \frac{4\pi \tilde M}{\log (2) \delta}, \quad \forall\, r\geq \frac {r^*} 2.
\end{equation}
Let $(Y_1(t))_{t\geq 0}$, $(Y_2(t))_{t\geq 0}$ be the solutions to the ODE \eqref{eq:sepODE} with initial values $Y_1(0) = y_1$ and $Y_2(0) = y_2$.
Let $R_1(t), R_2(t)$ and $\phi_1(t), \phi_2(t)$ be the polar coordinate representations of $Y_1(t), Y_2(t)$ such that
$$Y_{1/2}(t) = R_{1/2}(t) \begin{pmatrix}\cos \phi_{1/2}(t)\\ \sin \phi_{1/2}(t)\end{pmatrix}.$$
We have $R_1(0) = r^*,$ $R_2(0) = r^*+\delta$, $\phi_1(0) = \phi_2(0) = \phi$ and by \eqref{eq:SDE_polar} with $c = 1$ and $\sigma = 0$
$$\dot R_{1/2}(t) = - R_{1/2}(t), \quad \dot \phi_{1/2}(t) = -\rho(R_{1/2}(t)).$$
We can compute $R_1(t) = r^* e^{-t}$ and $R_2(t) = (r^*+\delta) e^{-t}$. In particular, for all $t \in [0,\log 2]$, we have $r^*/2 \leq R_{1/2}(t) \leq r^* +\delta \leq 2r^*$ and $R_2(t)-R_1(t) \geq  \delta/ 2$. By the mean value theorem and by \eqref{eq:sepLemmaShearCond}, we have
\begin{align*}
    \rho(R_2(t)) - \rho(R_1(t)) &\geq (R_2(t)-R_1(t))\inf_{r \in [R_1(t), R_2(t)]} \rho'(r) \\
    &\geq \frac{\delta}{2} \frac{4\pi \tilde M}{\log (2)\delta R_2} \geq \frac{\pi \tilde M}{ \log(2) r^*},
\end{align*}
for each $t\in [0,\log 2]$. Choosing compatible lifts on $\mathbb R$ of the angles $\phi_{1/2}(t) \in [0,2\pi)$, we get
$$\phi_1(\log 2) -\phi_2(\log 2) = \int_0^{\log 2} (\rho(R_2(t)) - \rho(R_1(t)))\,\dd t \geq \frac{\pi \tilde M}{ r^*}.$$
Let $\tilde T \in (0,\log 2]$ be the first time that $\phi_1(\tilde T) -\phi_2(\tilde T) = \pi \tilde M/ r^*\leq \pi$. We use the cosine law and the inequality $\sin(\pi\alpha/2) \geq \alpha, ~\forall\, \alpha \in [0,1]$, to compute
\begin{align*}
    |Y_1(\tilde T)-Y_2(\tilde T)| &= \sqrt{R_1(\tilde T)^2 + R_2(\tilde T)^2 - 2 R_1(\tilde T)R_2(\tilde T) \cos(\phi_1(\tilde T) -\phi_2(\tilde T))}\\
    &\geq \frac{r^*}{2} \sqrt{2-2 \cos\Big(\frac{\pi \tilde M}{ r^*}\Big)} = r^* \sin\Big(\frac{\pi \tilde M}{ 2 r^*}\Big) \geq \tilde M.
\end{align*}
Furthermore, it follows from exact solutions to $R_1$ and $R_2$ that
$$\max(|Y_1|, |Y_2|) \leq {r^*+\delta} \leq 2r^*.$$

\textbf{Step 2:} We apply the support theorem to each step described in Figure \ref{fig:separation}. Let $x_1\neq x_2 \in \mathbb R^2$ be two distinct points and $M>0$. Let $\delta \coloneqq |x_1-x_2|$ be their distance and $\phi \in [0, 2\pi]$ be the phase of their difference such that 
$$x_2-x_1 = \delta \begin{pmatrix}\cos \phi\\\sin \phi\end{pmatrix}.$$
Also, let $\tilde M = M+3$ and $\tilde T\geq 0$, $r^*>0$, $y_1, y_2\in \mathbb R^2$ as in Step 1. Note that $y_1-x_1 = y_2 -x_2$. By continuity of the flow, there exists an $\varepsilon>0$, such that 
\begin{align*}
     |\varphi(\tilde T, \omega_0, \tilde y_1) - \varphi(\tilde T, \omega_0, \tilde y_2)| &\geq \tilde M-1= M+2, \quad \text{and}\\
     |\varphi(\tilde T, \omega_0, \tilde y_1)| \vee |\varphi(\tilde T, \omega_0, \tilde y_2)|&\leq 2r^*+1,
\end{align*}
for all $\tilde y_1 \in B_{y_1}(\varepsilon), \, \tilde y_2 \in B_{y_2}(\varepsilon)$. 
By the Stroock-Varadhan support theorem \cite{StrookVaradhan72}, we furthermore have
\begin{equation}\label{eq:sepLemmaProb1}
    \mathbb P(|\varphi(\tilde T, \omega, \tilde y_1) - \varphi(\tilde T, \omega, \tilde y_2)| \geq M+1 \text{ and }  |\varphi(\tilde T, \omega, \tilde y_1)|\vee |\varphi(\tilde T, \omega, \tilde y_2)|\leq 2r^*+2)>0,
\end{equation}
By applying the support theorem to the two-point motion, there exists a small time  ${0<T_\varepsilon\ll 1}$, such that
\begin{align}
    &\quad~\mathbb P(\varphi(T_\varepsilon, \omega, x_1) \in B_{y_1}(\varepsilon) \text{ and } \varphi(T_\varepsilon, \omega, x_2)\in B_{y_2}(\varepsilon))\nonumber \\
    &= \mathbb P(\varphi(T_\varepsilon, \omega, x_1)-x_1 \in B_{y_1-x_1}(\varepsilon) \text{ and } \varphi(T_\varepsilon, \omega, x_2)-x_2\in B_{y_2-x_2}(\varepsilon))>0,\label{eq:sepLemmaProb2}
\end{align}
where we use the fact that $y_1-x_1 = y_2 -x_2$, as well as,
\begin{align}
&\quad~\mathbb P (|\varphi(T_\varepsilon, \omega, z_1)|\leq 1 \text{ and } |\varphi(T_\varepsilon, \omega, z_2)| \geq M)\nonumber\\
&\geq \mathbb P (\varphi(T_\varepsilon, \omega, z_1) - z_1 \in B_{-z_1}(1) \text{ and } \varphi(T_\varepsilon, \omega, z_2) - z_2 \in B_{-z_1}(1))>0\label{eq:sepLemmaProb3},
\end{align}
for all $z_1, z_2 \in \mathbb R^2$ with $|z_1-z_2|\geq M+1$ and $|z_1|, |z_2| \leq 2r^* +2.$

\textbf{Step 3:} We conclude the proof of the lemma. Set
\begin{align*}
\tilde y_{1/2} \coloneqq \varphi(T_\varepsilon, \omega, x_{1/2}),\quad\text{and}\quad z_{1/2} \coloneqq \varphi(T_\varepsilon+\tilde T, \omega, x_{1/2}) = \varphi(\tilde T, \theta_{T_\varepsilon} \omega, \tilde y_{1,2}).
\end{align*}
By \eqref{eq:sepLemmaProb2}, we have 
$$\mathbb P(\tilde y_1 \in B_{y_1}(\varepsilon) \text{ and }\tilde y_2 \in B_{y_2}(\varepsilon)) >0.$$
By \eqref{eq:sepLemmaProb1}, we have
$$\mathbb P(|z_1-z_2|\geq M+1\text{ and } |z_1|\vee |z_2| \leq 2r^*+2\mid \tilde y_1 \in B_{y_1}(\varepsilon) \text{ and }\tilde y_2 \in B_{y_2}(\varepsilon)) >0.$$
Finally, by \eqref{eq:sepLemmaProb3}, we have
\begin{multline*}
    \mathbb P\big(|\varphi(T_\varepsilon,\theta_{T_\varepsilon+\tilde T}\omega, z_1)|\leq 1 \text{ and } |\varphi(T_\varepsilon,\theta_{T_\varepsilon+\tilde T}\omega, z_2)|\geq M \\
    \big|\, |z_1-z_2|\geq M+1 \text{ and } |z_1|\vee|z_2|\leq 2r^*+2\big)>0.
\end{multline*}
Setting $T\coloneqq 2T_\varepsilon + \tilde T$, we conclude 
\begin{align*}
    \mathbb P\left(|\varphi(T,\omega, x_1)|\leq 1\text{ and } \left|\varphi(T,\omega, x_2)\right|\geq M \right)>0.
\end{align*}

\end{proof}

\bibliographystyle{alpha}
\bibliography{mybibnew}

\end{document}